\documentclass[11pt, reqno]{amsart}
\usepackage{amsmath,amssymb,amsthm,mathrsfs,enumerate,bm,xcolor,multirow,pbox}
\usepackage{graphicx,color,framed,tikz,caption,subcaption}
\usepackage{enumitem}
\setlist{leftmargin=5mm}
\usepackage[colorlinks,linkcolor=red,citecolor=blue,urlcolor=blue]{hyperref}
\allowdisplaybreaks[4]
\numberwithin{equation}{section}

\newcommand{\R}{\mathbb{R}}

\newcommand{\E}{\mathbb{E}}
\newcommand{\Prob}{\mathbb{P}}

\newcommand{\pnorm}[2]{\lVert#1\rVert_{#2}}
\newcommand{\bigpnorm}[2]{\big\lVert#1\big\rVert_{#2}}

\newcommand{\abs}[1]{\lvert#1\rvert}

\renewcommand{\epsilon}{\varepsilon}

\renewcommand{\d}[1]{\mathrm{d}#1}

\newcommand{\ceil}[1]{\left\lceil #1 \right\rceil}

\renewcommand{\hat}{\widehat}
\renewcommand{\tilde}{\widetilde}
\newcommand{\diag}{\text{diag}}

\DeclareMathOperator{\tr}{tr}

\DeclareMathOperator{\ave}{ave}
\DeclareMathOperator{\op}{op}

\newcommand{\beq}{\begin{equation}}
\newcommand{\eeq}{\end{equation}}
\newcommand{\beqa}{\begin{equation} \begin{aligned}}
\newcommand{\eeqa}{\end{aligned} \end{equation}}
\newcommand{\beqas}{\begin{equation*} \begin{aligned}}
\newcommand{\eeqas}{\end{aligned} \end{equation*}}

\newcommand{\bit}{\begin{itemize}}
	\newcommand{\eit}{\end{itemize}}
\newcommand{\bmat}{\begin{bmatrix}}
	\newcommand{\emat}{\end{bmatrix}}
\DeclareMathOperator{\argmin}{\mathrm{argmin}}
\DeclareMathOperator{\argmax}{\mathrm{argmax}}

\theoremstyle{definition}\newtheorem{problem}{Problem}[section]
\theoremstyle{definition}
\theoremstyle{remark}

\theoremstyle{remark}\newtheorem{remark}[problem]{Remark}
\theoremstyle{definition}
\theoremstyle{plain}\newtheorem{theorem}[problem]{Theorem}
\theoremstyle{plain}
\theoremstyle{plain}\newtheorem{lemma}[problem]{Lemma}
\theoremstyle{plain}\newtheorem{proposition}[problem]{Proposition}
\theoremstyle{plain}
\theoremstyle{plain}

\AtBeginDocument{%
	\def\MR#1{}
}

\begin{document}

\title[]{Uncertainty quantification in the Bradley-Terry-Luce model}

\author[C. Gao]{Chao Gao}\thanks{The research of C. Gao is partially supported by NSF CAREER award DMS-1847590 and NSF grant CCF-1934931.}
\address[C. Gao]{
Department of Statistics, University of Chicago, Chicago, IL 60637, USA.
}
\email{chaogao@uchicago.edu}
\author[Y. Shen]{Yandi Shen}

\address[Y. Shen]{
Department of Statistics, University of Chicago, Chicago, IL 60637, USA.
}
\email{ydshen@uchicago.edu}

\author[A. Zhang]{Anderson Y. Zhang}\thanks{The research of A.Y. Zhang is supported in part by NSF grant DMS-2112988.}

\address[A. Zhang]{
Department of Statistics, University of Pennsylvania, Philadelphia, PA 19104, USA.
}
\email{ayz@wharton.upenn.edu}

\date{\today}


\begin{abstract}
The Bradley-Terry-Luce (BTL) model is a benchmark model for pairwise comparisons between individuals. Despite recent progress on the first-order asymptotics of several popular procedures, the understanding of uncertainty quantification in the BTL model remains largely incomplete, especially when the underlying comparison graph is sparse. In this paper, we fill this gap by focusing on two estimators that have received much recent attention: the maximum likelihood estimator (MLE) and the spectral estimator. Using a unified proof strategy, we derive sharp and uniform non-asymptotic expansions for both estimators in the sparsest possible regime (up to some poly-logarithmic factors) of the underlying comparison graph. These expansions allow us to obtain: (i) finite-dimensional central limit theorems for both estimators; (ii) construction of confidence intervals for individual ranks; (iii) optimal constant of $\ell_2$ estimation, which is achieved by the MLE but not by the spectral estimator. Our proof is based on a self-consistent equation of the second-order remainder vector and a novel leave-two-out analysis. 
\end{abstract}

\maketitle



\section{Introduction}

\subsection{Overview} In this paper, we study the problem of uncertainty quantification in the Bradley-Terry-Luce (BTL) model. Given $n$ individuals with unknown merit scores $\theta^* = (\theta^*_1,\ldots,\theta^*_n)^\top$ and an Erd\H{o}s-R\'{e}nyi comparison graph $A = \{A_{ij}\}_{i<j}$ with success probability $p$, individuals $i$ and $j$ are compared $L$ times if $A_{ij} = 1$, and each binary result $y_{ij\ell}$ is independently in favor of the former with probability
\begin{align*}
\Prob(y_{ij\ell} = 1) \equiv \frac{e^{\theta^*_i}}{e^{\theta^*_i} + e^{\theta^*_j}}.
\end{align*}
Equipped with the realization of the comparison graph $A$ and all the pairwise comparison results, the goal is to conduct inference on the unknown merit vector $\theta^*$. We give a detailed review of the BTL model in Section \ref{sec:main} ahead. Along with its generalizations \cite{luce1959individual,mcfadden1973conditional,plackett1975analysis} and close cousins in assortive networks \cite{holland1981exponential,park2004statistical,chatterjee2011random}, the BTL model has found extensive applications in web search \cite{dwork2001rank,cossock2006subset}, competitive sports \cite{motegi2012network,sha2016chalkboarding}, and social network analysis \cite{newman2003structure,robins2007recent}.

Among many statistical procedures developed for the BTL model, the following two approaches have received particular attention in recent years due to their wide applicability in practice.
\begin{itemize}
\item (Maximum likelihood estimator) In the current formulation, the maximum likelihood estimator (MLE) leads to a convex program and can be solved by efficient algorithms developed in \cite{zermelo1929die,ford1957solution,lange2000optimization,hunter2004mm}.
\item (Spectral estimator) The spectral method, also named \textit{rank centrality} in its original discovery \cite{negahban2017rank}, associates the pairwise comparisons with a Markov chain on the underlying comparison graph. The estimator is then taken to be the stationary measure of the corresponding sample transition probability matrix, and hence can be solved efficiently via power iterations \cite{negahban2017rank}.
\end{itemize} 
In the past few years, a series of works \cite{negahban2017rank,chen2019spectral,chen2020partial,chen2021optimal} have studied in depth the theoretical properties of these two estimators, with focus on their $\ell_2$/$\ell_\infty$ estimation accuracy and performance in partial/full ranking; see Section \ref{sec:main} ahead for a detailed review of related results.

Compared to the above thorough study of the first-order asymptotics of the estimators, the understanding of the associated limiting distribution theory is much less complete. A critical step in establishing optimal error bounds for partial recovery of ranking in \cite{chen2020partial} is sharp tail probability estimates for both the MLE and the spectral method. These results suggest, though do not directly imply, asymptotic normality. In the seminal work \cite{simons1999asymptotics}, asymptotic normality was first rigorously proved for the MLE when the comparison graph is fully connected (i.e. the underlying Erd\H{o}s-R\'{e}nyi comparison graph satisfies $p = 1$). Using the same proof strategy, \cite{han2020asymptotic} later extended the above result to the regime $p \asymp n^{-1/10}$ (up to some poly-logarithmic factors). Since the Erd\H{o}s-R\'{e}nyi graph is connected with high probability as soon as $np\geq (1+\epsilon)\log n$ for any $\epsilon > 0$, this leaves open the most challenging but practically relevant regime $p\asymp n^{-1}$ (up to some poly-logarithmic factors), as many real world networks only have near constant degrees. The goal of this paper is to fill this gap by providing a sharp characterization of both estimators in this regime. 

\subsection{Main contribution} Using a unified proof strategy, our main theoretical results provide a non-asymptotic expansion of both estimators in the regime $p\asymp n^{-1}$ (up to some poly-logarithmic factors). Let us start with the MLE, which we will denote by $\hat{\theta}$. Our first main result states that (see Theorem \ref{thm:mle_expansion} ahead for the formal statement)
\begin{align}\label{eq:mle_exp_intro}
\hat{\theta}_i - \theta_i^* \approx \frac{b_i}{d_i},
\end{align}
where
\begin{align*}
b_i =\sum_{j:j\neq i} A_{ij}\big(\bar{y}_{ij} - \psi(\theta_i^*-\theta_j^*)\big),\quad d_i = \sum_{j:j\neq i}A_{ij}\psi'(\theta_i^*-\theta_j^*).
\end{align*}
Here $\bar{y}_{ij} \equiv L^{-1}\sum_{\ell=1}^L y_{ij\ell}$ is the averaged comparison result between individuals $i$ and $j$ when $A_{ij} = 1$, $\psi(x) \equiv e^x/(1+e^x)$ is the logistic function, and the approximation $\approx$ in (\ref{eq:mle_exp_intro}) is uniform over the coordinates $i\in[n]$. To the best of our knowledge, (\ref{eq:mle_exp_intro}) as formally stated in Theorem \ref{thm:mle_expansion} is the first result in the literature with an explicit non-asymptotic expansion in the sparse regime $p\asymp n^{-1}$ (up to some poly-logarithmic factors). 

Thanks to the tractable form of the main terms $\{b_i/d_i\}$ (we will defer their interpretation to Section \ref{sec:main} ahead), we are able to characterize the asymptotic behavior of the MLE $\hat{\theta}$. In particular, we will present three concrete applications of the expansion in (\ref{eq:mle_exp_intro}): (i) a finite-dimensional central limit theorem (CLT) of $\hat{\theta}$; (ii) construction of confidence intervals for individual ranks; (iii) optimal constant of $\ell_2$ estimation in the BTL model; see Section \ref{sec:application} ahead for details.

Using the same proof strategy underlying (\ref{eq:mle_exp_intro}), we are able to derive a similar expansion for the spectral estimator which we denote by $\tilde{\theta}$ (see Theorem \ref{thm:spec_expansion} ahead for the formal statement):
\begin{align}\label{eq:spec_exp_intro}
\tilde{\theta}_i - \theta_i^* \approx \frac{\tilde{b}_i}{\tilde{d}_i}, 
\end{align}
where
\begin{align*}
\tilde{b}_i =\sum_{j:j\neq i} A_{ij}(\pi_i^* + \pi_j^*)\big(\bar{y}_{ij} -\psi(\theta_i^*-\theta_j^*)\big),\quad \tilde{d}_i = \pi_i^*\cdot \sum_{j:j\neq i}A_{ij}\psi(\theta_j^*-\theta_i^*).
\end{align*}
Here $\pi^* = (\pi^*_1,\ldots,\pi^*_n)^\top$ is the stationary measure of the Markov chain associated with the spectral estimator (see Section \ref{subsec:btl_rev} ahead for details), and the approximation $\approx$ in (\ref{eq:spec_exp_intro}) is again uniform over the coordinates $i\in[n]$. To the best of our knowledge, (\ref{eq:spec_exp_intro}) is the first result in the literature that gives a sharp non-asymptotic expansion of the spectral estimator in the sparse regime $p\asymp n^{-1}$ (up to some poly-logarithmic factors). Analogous to the MLE $\hat{\theta}$, the expansion (\ref{eq:spec_exp_intro}) also leads to a finite-dimensional CLT for the spectral estimator and its exact constant of $\ell_2$ error. Interestingly, along with a lower bound in the local minimax framework (cf. Theorem \ref{thm:lower} ahead), we are able to conclude that \textit{the MLE achieves the optimal constant of $\ell_2$ estimation in the BTL model but the spectral method does not}. An analogous observation in terms of ranking performance has previously been made in \cite{chen2020partial}.

Let us now discuss briefly the proof strategy underlying expansions (\ref{eq:mle_exp_intro}) and (\ref{eq:spec_exp_intro}). The previous proof strategy adopted in \cite{simons1999asymptotics,han2020asymptotic} for the MLE $\hat{\theta}$ proceeds by building and then inverting a self-consistent equation of $\hat{\theta}$, where the key technical step is to approximate the inverse of the Hessian of the likelihood equation in an entrywise manner. In the fully connected case $p = 1$, such a technical result is available from \cite{simons1998approximating}, but its approximation accuracy deteriorates quickly as $p$ gets smaller, resulting in the final condition $n^{1/10}p\rightarrow\infty$ in \cite{han2020asymptotic}. In this paper we adopt a higher-order analogue of the above approach by first building a self-consistent equation over the remainder vector
\begin{align*}
\delta_i \equiv \hat{\theta}_i - \frac{b_i}{d_i}
\end{align*}
in the case of the MLE $\hat{\theta}$. This approach allows us to bypass the technical difficulty of approximating the inverse of the Hessian, and reduces the problem to a sharp enough $\ell_\infty$ control of $\delta$, which we then tackle with a leave-two-out analysis. As will be detailed in Section \ref{sec:proof_idea} ahead, leaving two out instead of one as in previous works \cite{chen2019spectral,chen2020partial} is essential for our analysis. We refer to Section \ref{sec:proof_idea} for a detailed discussion of proof strategies. 

In a broader context, this paper can be categorized under the general theme of uncertainty quantification/statistical inference for models with growing/infinite dimensions, see \cite{nickl2013confidence, javanmard2014confidence, vandegeer2014asymptotically, zhang2014confidence, ren2015asymptotic, ning2017general, vanderpas2017uncertainty, neykov2018unified, chen2019inference, sur2019modern, farrell2021deep, monard2021statistical} for an incomplete list of recent works in other benchmark statistical models. This paper, together with the prior works \cite{simons1999asymptotics, han2020asymptotic}, resolves uncertainty quantification in the BTL model in both sparse and dense regimes.

\subsection{Organization}

The rest of the paper is organized as follows. Section \ref{sec:main} starts with a review of the BTL model and then presents our main results. Section \ref{sec:proof_idea} discusses in detail our proof strategy. Section \ref{sec:application} develops three concrete applications of our main theoretical expansions. Proofs for the MLE and spectral method are given in Sections \ref{sec:proof_main} and \ref{sec:proof_spec} respectively, followed by proofs for the applications in Sections \ref{subsec:proof_app_clt}-\ref{sec:proof_app_constant}. Some additional auxiliary results are presented in Appendix \ref{sec:appendix}.

\subsection{Notation}\label{section:notation}

For any positive integer $n$, let $[n]$ denote the set $\{1,\ldots,n\}$. For $a,b \in \R$, $a\vee b\equiv \max\{a,b\}$ and $a\wedge b\equiv\min\{a,b\}$. For $a \in \R$, let $a_+\equiv a\vee 0$ and $a_- \equiv (-a)\vee 0$. For $x \in \R^n$, let $\pnorm{x}{q}=\pnorm{x}{\ell_q(\R^n)}$ denote its $q$-norm $(0< q\leq \infty)$ with $\pnorm{x}{2}$ abbreviated as $\pnorm{x}{}$. By $\bm{1}_n$ we denote the vector of all ones in $\R^n$. For a matrix $M \in \R^{n\times n}$, let $\pnorm{M}{\op}$ and $\pnorm{M}{F}$ denote the spectral and Frobenius norms of $M$ respectively. Let $M^\dagger$ denote its pseudo-inverse. We use $\{e_j\}$ to denote the canonical basis, whose dimension should be clear from the context. We use $\mathbb{I}$ to denote the indicator function.


For two nonnegative sequences $\{a_n\}$ and $\{b_n\}$, we write $a_n\lesssim b_n$ or $a_n = \mathcal{O}(b_n)$ if $a_n\leq Cb_n$ for some absolute constant $C$. We write $a_n\asymp b_n$ if $a_n\lesssim b_n$ and $b_n\lesssim a_n$. We write $a_n\ll b_n$ or $a_n = \mathfrak{o}(b_n)$ (respectively~$a_n\gg b_n$) if $\lim_{n\rightarrow\infty} (a_n/b_n) = 0$ (respectively~$\lim_{n\rightarrow\infty} (a_n/b_n) = \infty$). We follow the convention that $0/0 = 0$. 

Let $\varphi,\Phi$ be the density and the cumulative distribution function of a standard normal random variable. For any $\alpha \in (0,1)$, let $z_\alpha$ be the normal quantile defined by $\Prob(\mathcal{N}(0,1)> z_\alpha) = \alpha$. 


\section{Main results}\label{sec:main}

\subsection{BTL model: a review}\label{subsec:btl_rev}

We start by laying down the details of our problem setting. Consider $n$ individuals where each one is associated with some latent merit parameter $\theta_i^*\in\R$ for $i\in[n]$. Comparisons among these $n$ individuals are then made on a random Erd\H{o}s-R\'{e}nyi graph $A  = \{A_{ij}\}_{i<j}$ with edge probability $p$, i.e. $\{A_{ij}\}_{i<j}$ are i.i.d. Bernoulli variables with success probability $p$. For each connected edge $A_{ij} = 1$, $L$ independent comparisons denoted by $\{y_{ij\ell}\}_{\ell=1}^L$ are made between individuals $i$ and $j$, where $y_{ij\ell}$ are i.i.d. Bernoulli variables with success probability $\psi(\theta_i^* - \theta_j^*)$ with $\psi$ being the logistic function $\psi(t) \equiv e^t/(1+e^t)$. We observe the comparison graph $A$ and the comparison results $\{y_{ij\ell}\}$, and the goal is to conduct statistical inference of the merit parameters $\theta^* = (\theta_1^*,\ldots,\theta_n^*)^\top$. 

Throughout the paper, we will consider the following parameter space for $\theta^*$.
\begin{align}\label{def:theta_space}
\Theta(\kappa) \equiv \Big\{\theta^*\in\R^n: \max_{i\in[n]}\theta_i^* - \min_{i\in[n]}\theta_i^* \leq \kappa, \bm{1}_n^\top\theta^* = 0\Big\}.
\end{align}
Here $\kappa$ is known as the \textit{dynamic range}, and since the BTL model is only identifiable up to a global shift in the merit parameter $\theta^*$, we make the centering condition $\bm{1}_n^\top\theta^* = 0$. In our theory below, $n$ is taken to be the main asymptotic parameter, and all other problem parameters $p, \theta^*, \kappa, L$ are allowed to change with $n$. We will mostly be interested in the following regime: 
\begin{align}\label{def:regime}
\kappa = \mathcal{O}(1), \quad np\gg (\log n)^\alpha \text{ for some }\alpha \geq 1.
\end{align}
We make a few remarks on the above conditions.
\begin{enumerate}
\item Here $\kappa = \mathcal{O}(1)$ is assumed to facilitate theoretical analysis. Direct conveniences brought by this condition include that the comparison success probability $\psi(\theta_i^*-\theta_j^*)$ is bounded away from $0$ and $1$, and the Hessian of the log-likelihood function $\ell_n(\theta)$ (see (\ref{def:hessian}) below) will be well-conditioned. This condition is commonly assumed in the existing literature \cite{chen2015spectral,chen2019spectral,chen2020partial}.
\item It is well-known that the underlying Erd\H{o}s-R\'{e}nyi graph is connected with high probability when $np \geq (1+\epsilon)\log n$ and disconnected if  $np \leq (1-\epsilon)\log n$ for any $\epsilon > 0$. Therefore the second condition allows the sparsest possible regime (without losing identifiability) of the graph up to some poly-logarithmic factors. 
\end{enumerate}
More discussions of these conditions will follow after the statement of the main results; see Remark \ref{remark:mle_extension} ahead.

In the BTL model, arguably the most natural procedure is the maximum likelihood estimator (MLE), which can be traced back to \cite{zermelo1929die,ford1957solution}. After some elementary algebra, the (normalized) negative log-likelihood function is given by
\begin{align}\label{eq:mle_likelihood}
\ell_n(\theta) = \sum_{i<j} A_{ij}\Big[\bar{y}_{ij}\log\frac{1}{\psi(\theta_i - \theta_j)} + \bar{y}_{ji}\log\frac{1}{\psi(\theta_j - \theta_i)}\Big].
\end{align}
Here for each $i < j$, we let $\bar{y}_{ij}\equiv \big(\sum_{\ell=1}^Ly_{ij\ell}\big)/L$ and we take $\bar{y}_{ji} = 1- \bar{y}_{ij}$ by convention. In accordance with the identifiability condition $\bm{1}_n^\top \theta^* = 0$, we define the MLE to be 
\begin{align}\label{def:mle}
\hat{\theta} \in \argmin_{\theta:\bm{1}_n^\top \theta = 0} \ell_n(\theta).
\end{align}
Under conditions in (\ref{def:regime}) with any $\alpha \geq 1$, the MLE exists and is unique with probability $1-\mathcal{O}(n^{-10})$ since with the prescribed probability, the above optimization can be constrained to the compact set $\{\theta:\pnorm{\theta-\theta^*}{\infty}\leq 5\}$ (cf. \cite[Proposition 8.1]{chen2020partial}) and the likelihood is strongly convex on this set (cf. Lemma \ref{lem:hessian_eigen} ahead). In what follows, we work on the event where $\hat{\theta}$ is well-defined. 

Another popular approach, named \textit{rank centrality} in its original discovery \cite{negahban2017rank}, associates the pairwise comparisons with a random walk on the underlying Erd\H{o}s-R\'{e}nyi graph. More precisely, consider a Markov chain with $n$ states (corresponding to the $n$ individuals), and let the sample transition matrix $P$ be defined by
\begin{align*}
P_{ij} \equiv 
\begin{cases}
\frac{1}{d}A_{ij}\bar{y}_{ji}, & i\neq j,\\
1-\frac{1}{d}\sum_{k:k\neq i}A_{ik}\bar{y}_{ki}, & i=j.
\end{cases}
\end{align*}
Here we take $d = 2np$ throughout so that with probability $1 - \mathcal{O}(n^{-10})$, the matrix $P$ has nonnegative entries. Its population version conditioning on $A$ is
\begin{align*}
P^*_{ij} \equiv \E(P_{ij}|A) = 
\begin{cases}
\frac{1}{d}A_{ij}\psi(\theta_j^* - \theta_i^*), &i\neq j,\\
1-\frac{1}{d}\sum_{k:k\neq i} A_{ik}\psi(\theta_k^*-\theta_i^*), & i = j,
\end{cases}
\end{align*}
hence the imaginary random walker on the graph has a higher tendency of moving to an adjacent node (individual) with a larger merit parameter. By direct verification, this population transition matrix admits the stationary measure
\begin{align}\label{def:pi}
\pi^* \equiv \Big(\frac{e^{\theta_1^*}}{\sum_{k=1}^n e^{\theta_k^*}},\ldots, \frac{e^{\theta_n^*}}{\sum_{k=1}^n e^{\theta_k^*}}\Big),
\end{align}
and the approach of rank centrality estimates it using the stationary measure $\hat{\pi}$ of $P$: 
\begin{align}
\hat{\pi}^\top P = \hat{\pi}^\top.
\end{align}
In accordance with the identifiability condition $\bm{1}_n^\top \theta^* = 0$, the associated estimator of $\theta^*$ is defined as
\begin{align}\label{def:spec}
\tilde{\theta}_i \equiv \log \hat{\pi}_i - \frac{1}{n}\sum_{k=1}^n \log \hat{\pi}_k.
\end{align}
Since the above rank centrality algorithm builds on a long list of spectral ranking methods (see \cite{vigna2016spectral} for a comprehensive survey), we will henceforth call it the spectral method. For the rest of the paper, we reserve the notation $\hat{\theta}$ for the MLE and $\tilde{\theta}$ for the spectral method. 

In the last few years, a series of works have made significant contribution towards understanding the performances of both estimators, in terms of both $\ell_2$ and $\ell_\infty$ accuracy.
\begin{proposition}[\cite{negahban2017rank,chen2019spectral,chen2020partial}]\label{prop:mle_prelim}
Suppose the conditions in (\ref{def:regime}) hold with $\alpha = 1$. Then there exists some $C = C(\kappa) > 0$ such that the following hold with probability at least $1-\mathcal{O}(n^{-10})$ uniformly over $\theta^*\in\Theta(\kappa)$.
\begin{align*}
(\textrm{MLE}) \quad\quad&\pnorm{\hat{\theta} - \theta^*}{}^2 \leq \frac{C}{pL} \quad \text{ and } \quad \pnorm{\hat{\theta} - \theta^*}{\infty}^2 \leq \frac{C\log n}{npL},\\
(\textrm{spectral}) \quad\quad &\pnorm{\tilde{\theta}-\theta^*}{}^2 \leq \frac{C}{pL} \quad \text{ and } \quad \pnorm{\tilde{\theta}-\theta^*}{\infty}^2 \leq \frac{C\log n}{npL}.
\end{align*}

\end{proposition} 
In fact, both error rates are known to be minimax rate optimal \cite{negahban2017rank,chen2019spectral} over $\Theta(\kappa)$; see also \cite[Lemmas 11.1 and 11.3]{chen2020partial} for some more refined $\ell_\infty$ tail estimates. Other aspects of the estimators, notably their ranking performances, have also been studied in depth in \cite{negahban2017rank,chen2019spectral,chen2020partial}.

\subsection{Non-asymptotic expansions}
Compared to the above progress in the first-order asymptotics, the understanding of limiting distribution theory in the BTL model is much less complete. In the seminal work \cite{simons1999asymptotics}, the authors derived the $\ell_\infty$ consistency and asymptotic normality of the MLE in the full comparison case $p = 1$. Adopting a similar proof technique, the recent work \cite{han2020asymptotic} extended \cite{simons1999asymptotics} to the sparsity regime $n^{1/10}p\rightarrow\infty $ (up to some poly-logarithmic factors). This leaves open the most challenging but practically relevant sparsity regime (\ref{def:regime}) where the average number of comparisons is only slowly growing.

The following first main result of this paper provides a sharp non-asymptotic expansion of the MLE $\hat{\theta}$ in the regime (\ref{def:regime}), from which asymptotic normality and a few other interesting corollaries can be readily deduced. We present its proof in Section \ref{sec:proof_main}. 
\begin{theorem}\label{thm:mle_expansion}
Suppose that the conditions in (\ref{def:regime}) hold with $\alpha = 3/2$. Then for any $i\in[n]$, we have
\begin{align}\label{eq:mle_expansion}
\hat{\theta}_i - \theta^*_i = (1+\epsilon_{1,i})\frac{b_i}{d_i} + \epsilon_{2,i}.
\end{align}
Here $\epsilon_1,\epsilon_2$ are two vectors in $\R^n$ such that $\pnorm{\epsilon_1}{\infty} = \mathfrak{o}(1)$ and $\pnorm{\epsilon_2}{\infty} = \mathfrak{o}(1/\sqrt{npL})$ with probability $1-\mathcal{O}(n^{-10})$, and $b,d\in\R^n$ are given by
\begin{align*}
b_i =\sum_{j:j\neq i} A_{ij}\big(\bar{y}_{ij} - \psi(\theta_i^*-\theta_j^*)\big),\quad d_i = \sum_{j:j\neq i}A_{ij}\psi'(\theta_i^*-\theta_j^*).
\end{align*}
\end{theorem}

\begin{remark}
Using standard concentration arguments, we have $b_i \asymp \sqrt{np/L}$ and $d_i\asymp np$ for each $i\in[n]$, so the main term satisfies $b_i/d_i \asymp 1/\sqrt{npL}$. Hence upon proper normalization, $\epsilon_2$ is a small term in an entrywise manner. 
\end{remark}
\begin{remark}\label{remark:mle_extension}
We comment on two possible generalizations of the above theorem.
\begin{enumerate}
\item (\textit{Heterogeneity}) A direct generalization of the current BTL model is to allow heterogeneity in the pairwise comparisons, that is, replace the Erd\H{o}s-R\'{e}nyi probability $p$ by $p_{ij}$ and the number of comparisons $L$ by $L_{ij}$ for each pair $(i,j)$ of comparison. With these changes, the above theorem continues to hold upon changing the conditions in (\ref{def:regime}) to $\min_{i<j}p_{ij} \gg (\log n)^\alpha/n$ and $\max_{i<j}p_{ij}/\min_{i<j}p_{ij}\leq C$ for some universal $C > 0$, and the main term $b_i/d_i$ to 
\begin{align*}
\frac{\sum_{j:j\neq i}A_{ij}(y_{ij} - L_{ij}\psi(\theta_i^*-\theta_j^*))}{\sum_{j:j\neq i}A_{ij}L_{ij}\psi'(\theta_i^*-\theta_j^*)},
\end{align*}
where $y_{ij}$ is now $\sum_{\ell=1}^{L_{ij}} y_{ij\ell}$; see \cite[Section 3.1]{benaych2020spectral} for some related probabilistic properties of the inhomogeneous Erd\H{o}s-R\'{e}nyi graph. The applications to be introduced in Section \ref{sec:application} ahead will also hold upon proper modifications. 
\item (Sparsity condition) The slightly worse exponent in $np\gg (\log n)^{3/2}$ (instead of $1$) results from the bound $\pnorm{\delta}{\infty} \lesssim \sqrt{(\log n)^3/(np)^2}$$\mathcal{O}(1/\sqrt{npL})$ in Lemma \ref{lem:delta_main}, which is derived from an earlier rougher estimate $\pnorm{\delta}{\infty} \lesssim \sqrt{(\log n)^1/(np)^0}\mathcal{O}(1/\sqrt{npL})$ in Lemma \ref{lem:loo_delta_max_bound}. Here $\delta$, defined in (\ref{def:delta}) for the MLE and (\ref{eq:spec_expansion}) for the spectral method, is the key remainder vector we are trying to bound; see the proof idea in Section \ref{sec:proof_idea} ahead for more details. At the cost of a much lengthier proof, we may be able to iterate the above arguments and improve the condition to $np\gg (\log n)^{(3+2k)/(2+2k)}$ after the $k$th iteration, reaching exponent $1+\eta$ after $\ceil{1/(2\eta)}$ iterations. An interesting theoretical question is whether the expansion (\ref{eq:mle_expansion}) still holds under the condition $np\gg \log n$, or even better, under $np\geq C\log n$ with some large universal $C > 0$ as established for some first-order results in \cite{chen2020partial}. We leave the complete resolution of this problem to a future work.
\end{enumerate}
\end{remark}


To the best of our knowledge, Theorem \ref{thm:mle_expansion} is the first result in the literature with an explicit non-asymptotic expansion in the regime (\ref{def:regime}). 
To help with its understanding, we now explain the form $\{b_i/d_i\}$ of the main term. For each fixed $i\in[n]$, let $\ell^{(i)}_n(\theta_i|\theta_{-i})$ be the local negative likelihood of the $i$th individual given the rest of the merit parameters $\theta_{-i}$:
\begin{align*}
\ell_n^{(i)}(\theta_i|\theta_{-i}) = \sum_{j:j\neq i}A_{ij}\Big[\bar{y}_{ij}\log\frac{1}{\psi(\theta_i - \theta_j)} + \bar{y}_{ji}\log\frac{1}{\psi(\theta_j-\theta_i)}\Big].
\end{align*}
It can be readily verified that: (i) $\hat{\theta}_i$ is a minimizer of $\ell_n^{(i)}(\theta_i|\hat{\theta}_{-i})$ and (ii) $b_i = -\partial_{\theta_i}\ell_n^{(i)}(\theta_i^*|\theta^*_{-i})$ and $d_i = \partial^2_{\theta_i}\ell_n^{(i)}(\theta_i^*|\theta^*_{-i})$, so a local quadratic expansion of $\ell_n^{(i)}(\theta_i|\hat{\theta}_{-i})$ around $\theta_i^*$ yields that
\begin{align}\label{eq:mle_local_quad}
\notag\hat{\theta}_i &\approx \argmin_{\theta_i}\Big[\ell_n^{(i)}(\theta_i^*|\hat{\theta}_{-i}) +\partial_{\theta_i}\ell_n^{(i)}(\theta_i^*|\hat{\theta}_{-i})\cdot(\theta_i - \theta_i^*) + \frac{1}{2}\partial^2_{\theta_i}\ell_n^{(i)}(\theta_i^*|\hat{\theta}_{-i})\cdot(\theta_i - \theta_i^*)^2\Big]\\
&\approx \argmin_{\theta_i}\Big[\ell_n^{(i)}(\theta_i^*|\hat{\theta}_{-i}) -b_i(\theta_i - \theta_i^*) + \frac{1}{2}d_i(\theta_i - \theta_i^*)^2\Big]= \theta_i^* + \frac{b_i}{d_i}.
\end{align}
The key task of Theorem \ref{thm:mle_expansion} is to give a tight control of the above approximation error; see Section \ref{sec:proof_main} ahead for details.

As we will elaborate in Section \ref{sec:proof_idea} ahead, the previous proof idea adopted in \cite{simons1999asymptotics,han2020asymptotic} is no longer applicable in the sparse regime (\ref{def:regime}), and our Theorem \ref{thm:mle_expansion} is based on a completely different proof technique. Interestingly, this new proof technique also allows us to derive the following expansion of the spectral estimator $\tilde{\theta}$; see Section \ref{sec:proof_spec} for its proof.

\begin{theorem}\label{thm:spectral_expansion}\label{thm:spec_expansion}
Suppose that the conditions in (\ref{def:regime}) hold with $\alpha = 3/2$. Then for any $i\in[n]$, we have
\begin{align}\label{eq:spec_expansion}
\tilde{\theta}_i - \theta_i^* = (1+\tilde{\epsilon}_{1,i})\frac{\tilde{b}_i}{\tilde{d}_i} + \tilde{\epsilon}_{2,i}.
\end{align}
Here $\tilde{\epsilon}_1,\tilde{\epsilon}_2$ are two vectors in $\R^n$ such that $\pnorm{\tilde{\epsilon}_1}{\infty} = \mathfrak{o}(1)$ and $\pnorm{\tilde{\epsilon}_2}{\infty} = \mathfrak{o}(1/\sqrt{npL})$ with probability $1-\mathcal{O}(n^{-10})$, and $\tilde{b},\tilde{d}\in\R^n$ are given by
\begin{align*}
\tilde{b}_i =\sum_{j:j\neq i} A_{ij}(\pi_i^* + \pi_j^*)\big(\bar{y}_{ij} -\psi(\theta_i^*-\theta_j^*)\big),\quad \tilde{d}_i = \pi_i^*\cdot \sum_{j:j\neq i}A_{ij}\psi(\theta_j^*-\theta_i^*).
\end{align*}
Here $\pi^*$ is the population stationary measure defined in (\ref{def:pi}).
\end{theorem}

We now explain the form $\{\tilde{b}_i/\tilde{d}_i\}$ of the main term. Using a first-order Taylor expansion, we have
\begin{align}\label{eq:spec_prelim_expansion}
\tilde{\theta}_i - \theta_i^* = \log\Big(\frac{\hat{\pi}_i - \pi_i^*}{\pi_i^*} + 1\Big) - \frac{1}{n}\sum_{k=1}^n \log\Big(\frac{\hat{\pi}_k - \pi_k^*}{\pi_k^*} + 1\Big) \approx \frac{\hat{\pi}_i - \pi_i^*}{\pi_i^*}.
\end{align}
Then it follows from the definition of $\hat{\pi}$ that 
\begin{align}\label{eq:pi_property}
\hat{\pi}_i = \frac{\sum_{j:j\neq i}A_{ij}\bar{y}_{ij}\hat{\pi}_j}{\sum_{j:j\neq i}A_{ij}\bar{y}_{ji}} \approx \frac{\sum_{j:j\neq i}A_{ij}\bar{y}_{ij}\pi^*_j}{\sum_{j:j\neq i}A_{ij}\bar{y}_{ji}}.
\end{align}
Combining the above two displays, we have
\begin{align*}
\tilde{\theta}_i - \theta_i^*\approx \frac{\sum_{j:j\neq i}A_{ij}\big(\bar{y}_{ij}\pi^*_j-\bar{y}_{ji}\pi_i^*\big)}{\pi_i^*\cdot \sum_{j:j\neq i}A_{ij}\bar{y}_{ji}}\approx \frac{\sum_{j:j\neq i}A_{ij}\big(\bar{y}_{ij}\pi^*_j-\bar{y}_{ji}\pi_i^*\big)}{\pi_i^*\cdot \sum_{j:j\neq i}A_{ij}\psi(\theta_j^*-\theta_i^*)} = \frac{\tilde{b}_i}{\tilde{d}_i},
\end{align*}
upon noting that $\bar{y}_{ij}\pi^*_j - \bar{y}_{ji}\pi_i^* = (\pi_i^* + \pi_j^*)\big(\bar{y}_{ij} - \psi(\theta_i^* - \theta_j^*)\big)$. The key task of Theorem \ref{thm:spec_expansion} is to give a tight control of the above approximation error; see Section \ref{sec:proof_spec} ahead for details. 

Theorems \ref{thm:mle_expansion} and \ref{thm:spec_expansion} provide us with valuable information on the asymptotic behavior of the two estimators, and allow us to compare the two procedures on a much more refined basis than existing literature \cite{negahban2017rank,chen2019spectral,chen2020partial}. In particular, we will show that the MLE $\hat{\theta}$ achieves the optimal constant of $\ell_2$ estimation in the BTL model but the spectral estimator $\tilde{\theta}$ does not; see Section \ref{subsec:app_constant} for more details. 

\section{Proof Strategy}\label{sec:proof_idea}
This section discusses the unified high level proof idea underlying our main theorems, focusing on its difference from the previous one adopted in \cite{simons1999asymptotics,han2020asymptotic}; detailed proofs are given in Sections \ref{sec:proof_main} and \ref{sec:proof_spec} ahead.
\subsection{Summary of previous proof strategy}
We start by briefly summarizing the proof outline of \cite{simons1999asymptotics} when $p=1$ and explaining the additional technical difficulties in the sparse regime (\ref{def:regime}). Differentiating the likelihood in (\ref{eq:mle_likelihood}) leads to the score equations:
\begin{align*}
\frac{a_i}{L} = \sum_{j:j\neq i} \frac{e^{\hat{\theta}_i}}{e^{\hat{\theta}_i} + e^{\hat{\theta}_j}}, \quad i\in[n],
\end{align*}
with its population version given by $\E a_i = L\cdot \sum_{j:j\neq i}e^{\theta_i^*}/(e^{\theta_i^*} + e^{\theta_j^*})$. Here $a_i = \sum_{j:j\neq i}y_{ij}$ is the number of victories of the $i$th individual. Let $\Delta u_i \equiv (e^{\hat{\theta}_i} - e^{\theta^*_i})/e^{\theta_i^*}$, which can be seen as a proxy for $\hat{\theta}_i - \theta_i^*$ using a first-order Taylor expansion. Then some algebra yields that
\begin{align*}
\frac{e^{\hat{\theta}_i}}{e^{\hat{\theta}_i} + e^{\hat{\theta}_j}} - \frac{e^{\theta_i^*}}{e^{\theta_i^*} + e^{\theta_j^*}} &= \frac{e^{\theta_i^*}e^{\theta_j^*}\big(\Delta u_{i} - \Delta u_j\big)}{\big(e^{\theta_i^*} + e^{\theta_j^*}\big)^2}\times \frac{e^{\theta_i^*}+e^{\theta_j^*}}{e^{\hat{\theta}_i}+e^{\hat{\theta}_j}}\\ 
&\approx \psi'(\theta_i^*-\theta_j^*)\big(\Delta u_{i} - \Delta u_j\big).
\end{align*}
Combining the above two displays yields, in matrix form, that
\begin{align}\label{eq:proof_idea_expansion}
a - \E a \approx H(\theta^*) \Delta u \approx H(\theta^*)(\hat{\theta} - \theta^*),
\end{align}
where $a = (a_1,\ldots,a_n)^\top$, $\Delta u = (\Delta u_1, \ldots,\Delta u_n)^\top$, and $H(\theta^*)\in\R^{n\times n}$ is the Hessian given by 
\begin{align}\label{def:hessian}
H(\theta^*) \equiv \frac{\partial^2 \ell_n(\theta)}{\partial \theta^2}\Big|_{\theta = \theta^*} = \Big(\sum_{k:k\neq i}A_{ik}\psi'(\theta_i^* - \theta_k^*)\mathbb{I}_{i=j} - A_{ij}\psi'(\theta_i^*-\theta_j^*)\mathbb{I}_{i\neq j}\Big)_{i,j},
\end{align}
Given that the sample-mean type quantities $\{a_i - \E a_i\}$ are easy to analyze, it is natural to invert the Hessian $H(\theta^*)$ and find a close approximation $S$ of ideally simple form such that (note that $H(\theta^*)$ is singular along $\bm{1}_n$ so we take the pseudo-inverse)
\begin{align*}
\hat{\theta} - \theta^* \approx \big(H(\theta^*)\big)^{\dagger}(a-\E a) \approx S(a-\E a),
\end{align*}
an expansion similar to our (\ref{eq:mle_expansion}). The construction of such an $S$ is the main technical ingredient of \cite{simons1999asymptotics}, whose entrywise closeness to $\big(H(\theta^*)\big)^{\dagger}$ at the order of $n^{-2}$ is highly non-trivial to verify even in the case $p=1$ \cite{simons1998approximating}. For general $p$, this approximation error deteriorates to the order $(n^2p^3)^{-1}$ as reported by \cite[Lemma 7]{han2020asymptotic}, which eventually leads to the sub-optimal condition $n^{1/10}p\rightarrow \infty$.

We also mention a recent result by \cite{liu2021lagrangian} that achieves the sparsity level $p\asymp \frac{\log n}{n}$ via a debiasing technique. However, their method requires $L$ to be of greater order than $(n\log n)^2$, while our result and the results of \cite{simons1998approximating,han2020asymptotic} are all valid even for $L=1$.

\subsection{New proof strategy: remainder expansion and leave-two-out technique}
To bypass the above difficulty of approximating $H^{\dagger}(\theta^*)$, we take the following three-step approach. To avoid redundancy, we elaborate on the MLE $\hat{\theta}$ and briefly discuss the spectral estimator $\tilde{\theta}$.
\begin{enumerate}
\item With the approximation remainder vector defined by
\begin{align*}
\delta_i \equiv \hat{\theta}_i - \theta_i^* - \frac{b_i}{d_i},
\end{align*}
instead of directly expanding the target $\hat{\theta}-\theta^*$ as in (\ref{eq:proof_idea_expansion}), we seek to construct a self-consistent equation for the remainder vector $\delta$.
\item From the above self-consistent equation of $\delta$, derive a sharp bound for $\pnorm{\delta}{}$.
\item Starting from the above bound for $\pnorm{\delta}{}$ and a rough bound on $\pnorm{\delta}{\infty}$, we derive a sharp bound for $\pnorm{\delta}{\infty}$ in the regime (\ref{def:regime}). This step leverages a leave-two-out technique, which is inspired by, and improves substantially upon, the leave-one-out technique in \cite{chen2019spectral, chen2020partial}. 
\end{enumerate}
We now discuss these three steps in more detail. 

\subsubsection{Step 1: self-consistent equation of $\delta$}
By the definitions of $b_i$ and $d_i$ (in Theorem \ref{thm:mle_expansion}), we have
\begin{align}\label{eq:delta_id}
\delta_i = \hat{\theta}_i - \theta_i^* - \frac{\sum_{j:j\neq i}A_{ij}\big(\bar{y}_{ij} - \psi(\theta_i^*-\theta_j^*)\big)}{\sum_{j:j\neq i}A_{ij}\psi'(\theta_i^*-\theta_j^*)}.
\end{align}
On the other hand, a local quadratic expansion of the likelihood (see (\ref{eq:mle_local_quad}) and recall $\ell_n^{(i)}(\theta_i|\theta_{-i})$ defined therein) yields that
\begin{align*}
 \hat{\theta}_i - \theta_i^* \approx -\frac{\partial_{\theta_i}\ell_n^{(i)}(\theta_i^*|\hat{\theta}_{-i})}{\partial^2_{\theta_i}\ell_n^{(i)}(\theta_i^*|\theta^*_{-i})}\approx \frac{\sum_{j:j\neq i}A_{ij}\big(\bar{y}_{ij} - \psi(\theta_i^*-\hat{\theta}_j)\big)}{\sum_{j:j\neq i}A_{ij}\psi'(\theta_i^*-\theta^*_j)}.
\end{align*}
Combining the above two displays and Taylor expanding once more in the numerator yields that
\begin{align*}
\notag\delta_i &\approx\frac{\sum_{j:j\neq i}A_{ij}\big(\psi(\theta_i^*-\theta^*_j) - \psi(\theta_i^*-\hat{\theta}_j)\big)}{\sum_{j:j\neq i}A_{ij}\psi'(\theta_i^*-\theta^*_j)}\\
&\approx \frac{\sum_{j:j\neq i}A_{ij}\psi'(\theta_i^*-\theta_j^*)(\hat{\theta}_j-\theta_j^*)}{\sum_{j:j\neq i}A_{ij}\psi'(\theta_i^*-\theta^*_j)}.
\end{align*}
To make $\delta$ also appear on the right side, we plug in $\hat{\theta}_j -\theta_j^* = \delta_j + (b_j/d_j)$ for each $j\neq i$ in the above display, which yields the following self-consistent equation of $\delta$:
\begin{align}\label{eq:delta_mle_consistent}
d_i\cdot \delta_i \approx \sum_{j:j\neq i}A_{ij}\psi'(\theta_i^*-\theta_j^*)\delta_j + \sum_{j:j\neq i}A_{ij}\psi'(\theta_i^*-\theta_j^*)\frac{b_j}{d_j}.
\end{align}
In the next two steps, we will use this equation to obtain sharp bounds of $\pnorm{\delta}{}$ and $\pnorm{\delta}{\infty}$.

The derivation for the spectral estimator $\tilde{\theta}$ is analogous. The approximation remainder $\delta$ is now defined as 
\begin{align*}
\delta_i \equiv \frac{\hat{\pi}_i - \pi_i^*}{\pi_i^*} - \frac{1}{n}\sum_{k=1}^n\frac{\hat{\pi}_k - \pi_k^*}{\pi_k^*} - \frac{\tilde{b}_i}{\tilde{d}_i'}, 
\end{align*}
where $\tilde{d}_i' \equiv \pi_i^* \cdot \sum_{j:j\neq i}A_{ij}\bar{y}_{ji}$ is the sample version of $\tilde{d}_i$. Due to the preliminary expansion in (\ref{eq:spec_prelim_expansion}), we have $\delta_i \approx \tilde{\theta}_i - \theta_i^* -(\tilde{b}_i/\tilde{d}_i)$. On the other hand, the definition of $\hat{\pi}$ (see (\ref{eq:pi_property})) yields that
\begin{align*}
\frac{\hat{\pi}_i - \pi_i^*}{\pi_i^*} = \frac{\sum_{j:j\neq i}A_{ij}\big(\bar{y}_{ij}\hat{\pi}_j - \bar{y}_{ji}\pi_i^*\big)}{\pi_i^* \cdot \sum_{j:j\neq i}A_{ij}\bar{y}_{ji}} = \frac{\sum_{j:j\neq i}A_{ij}\bar{y}_{ij}(\hat{\pi}_j - \pi_j^*)}{\pi_i^* \cdot \sum_{j:j\neq i}A_{ij}\bar{y}_{ji}} + \frac{\tilde{b}_i}{\tilde{d}_i'}
\end{align*}
Combining the above two displays (and ignoring the small term $n^{-1}\sum_{k=1}^n(\hat{\pi}_k - \pi_k^*)/\pi_k^*$) yields that
\begin{align*}
\delta_i \approx \frac{\sum_{j:j\neq i}A_{ij}\bar{y}_{ij}(\hat{\pi}_j - \pi_j^*)}{\pi_i^* \cdot \sum_{j:j\neq i}A_{ij}\bar{y}_{ji}}. 
\end{align*}
By plugging in $\hat{\pi}_j - \pi_j^* \approx \pi_j^*\big(\delta_j + (\tilde{b}_j/\tilde{d}_j)\big)$, we arrive at the following self-consistent equation for the spectral estimator:
\begin{align}\label{eq:delta_spec_consistent}
\tilde{d}_i\cdot \delta_i \approx \sum_{j:j\neq i}A_{ij}\bar{y}_{ij}\pi_j^* \cdot \delta_j + \sum_{j:j\neq i}A_{ij}\bar{y}_{ij}\pi_j^*\frac{\tilde{b}_j}{\tilde{d}_j}.
\end{align}

\subsubsection{Step 2: $\ell_2$ control of $\delta$ by leave-one-out analysis}
This is an essential intermediate step towards the bound of $\pnorm{\delta}{\infty}$ in the next step. By rearranging the self-consistent equation in (\ref{eq:delta_mle_consistent}), we have in matrix form
\begin{align}\label{eq:mle_consistent}
H(\theta^*) \delta \approx R,
\end{align}
where $H(\theta^*)$ is the Hessian in (\ref{def:hessian}) and $R\in\R^n$ is some cumulated approximation error. This equation, which expands $\delta$ instead of $\hat{\theta}$ directly, can be seen as a higher order analogue of (\ref{eq:proof_idea_expansion}). The bound $\pnorm{\delta}{} = \mathfrak{o}(1/\sqrt{pL})$ (cf. Proposition \ref{prop:delta_bound}) is now readily obtained by (i) the bound of $\pnorm{R}{}$ by preliminary estimates in Proposition \ref{prop:mle_prelim} obtained via a leave-one-out analysis and (ii) the eigenvalue bound of $H(\theta^*)$ in Lemma \ref{lem:hessian_eigen}, i.e., with probability $1-\mathcal{O}(n^{-10})$,
\begin{align*}
np\lesssim\lambda_{\min,\perp}\big(H(\theta^*)\big)\lesssim \lambda_{\max}\big(H(\theta^*)\big)\lesssim np;
\end{align*}
see Section \ref{subsec:proof_delta_l2} for details. Here $\lambda_{\min,\perp}\big(H(\theta^*)\big)  = \min_{x\in\R^n: \pnorm{x}{}=1,x^\top \bm{1}_n = 0} x^\top H(\theta^*)x$ is the smallest eigenvalue of $H(\theta^*)$ orthogonal to the direction $\bm{1}_n$.

The derivation for the spectral estimator $\tilde{\theta}$ is analogous: by rearranging (\ref{eq:delta_spec_consistent}), we have
\begin{align*}
L \delta \approx \tilde{R},
\end{align*}
where $L$ is defined by $L_{ij} \equiv -A_{ij}\bar{y}_{ij}\pi_j^*$ for $i\neq j$ and $L_{ii} \equiv \sum_{j:j\neq i}A_{ij}\bar{y}_{ji}\pi_i^*$, and $\tilde{R}\in\R^n$ is some cumulated approximation error. Note that even though $L$ is not symmetric, its population version $\E(L|A)$ is a symmetric Laplacian matrix similar to $H(\theta^*)$, hence the bound for $\pnorm{\delta}{}$ can be similarly obtained as above.

\subsubsection{Step 3: $\ell_\infty$ control of $\delta$ by leave-two-out analysis}
For the more refined $\ell_\infty$ bound of $\delta$, we need to bound directly the right side of (\ref{eq:delta_mle_consistent}). For the terms therein, the term $\sum_{j:j\neq i}A_{ij}\psi'(\theta_i^*-\theta_j^*)(b_j/d_j)$ (and other approximation error terms) can be readily bounded using concentration arguments, and the most difficult term is 
\begin{align*}
\sum_{j:j\neq i}A_{ij}\psi'(\theta_i^*-\theta_j^*)\delta_j.
\end{align*}
If $\{A_{ij}\}_{j:j\neq i}$ was independent of $\{\delta_j\}_{j:j\neq i}$, then the above term could be bounded by standard concentration arguments along with the bound on $\pnorm{\delta}{}$ from the previous step. To decorrelate these two terms, we will define a proxy $\delta^{(i)}_j$ of each $\delta_j$ by leaving out the $i$th observation, so that $\delta^{(i)}_j$ is independent of $\{A_{ij}\}_{j:j\neq i}$ while being close to $\delta_j$; see Section \ref{subsubsec:leave_two_out} for details. The definition and analysis of $\delta^{(i)}_j$ further requires the analysis of a leave-two-out version of $\hat{\theta}$, which is essential for our purpose and fundamentally different from the bound given by the original leave-one-out analysis used in \cite{chen2019spectral, chen2020partial}. Indeed, as commented in Remark \ref{remark:mle_extension}, a leave-one-out analysis as in \cite{chen2019spectral,chen2020partial} will only yield the rough bound
\begin{align}\label{ineq:loo_rough}
\pnorm{\delta}{\infty} \lesssim \sqrt{\frac{(\log n)^1}{(np)^0}}\mathcal{O}(\frac{1}{\sqrt{npL}}),
\end{align}
which is not useful (i.e. of the order $\mathfrak{o}(1/\sqrt{npL})$) even for the case $p = 1$. Starting from the second stage of this hierarchy (i.e. leave-two-out analysis), this bound is improved to $\sqrt{(\log n)^3/(np)^2}\mathcal{O}(1/\sqrt{npL})$ (cf. Proposition \ref{prop:delta_entrywise_bound}), which is accurate enough under the sparsity condition $np\gg (\log n)^{1.5}$. We believe a higher order analysis (i.e. leave-$k$-out for a large but fixed $k$) will yield a better (but fixed) exponent closer to $1$, but we will not pursue this direction here to avoid digression.

To summarize, our new proof technique adds two novel technical components to the existing theoretical analysis in \cite{simons1998approximating, negahban2017rank,chen2019spectral,chen2020partial}: (i) a self-consistent equation for the remainder vector $\delta$ of first-order approximation; (ii) a leave-two-out technique that plays an essential role in the bound of $\pnorm{\delta}{\infty}$. The first component allows us to bypass the technical difficulty of approximating the inverse Hessian as in \cite{simons1998approximating} in the case of vanishing $p$, which leads to a sub-optimal condition as shown by an earlier analysis. The second component allows us to improve over the rough bound (\ref{ineq:loo_rough}), which is not needed for rate-optimal first-order bounds \cite{negahban2017rank,chen2019spectral,chen2020partial}. The rest of the proof combines concentration inequalities and the rate-optimal bounds in Proposition \ref{prop:mle_prelim}; we refer to Sections \ref{sec:proof_main} and \ref{sec:proof_spec} for proof details of the MLE and spectral estimator respectively.  

\section{Applications of the main expansions}\label{sec:application}

In this section, we present three applications of the expansions in Theorems \ref{thm:mle_expansion} and \ref{thm:spec_expansion}. Section \ref{subsec:app_clt} is devoted to a finite-dimensional CLT for the MLE $\hat{\theta}$ and spectral estimator $\tilde{\theta}$, followed by a construction of confidence intervals for the rank vector in Section \ref{subsec:app_rank}. Lastly, we discuss the implication for the optimal constant in $\ell_2$-estimation in Section \ref{subsec:app_constant}. Some simulation results are presented in Section \ref{subsec:simulation}.

\subsection{Application I: Finite-dimensional CLT}\label{subsec:app_clt}

A direct consequence of the expansions in Theorems \ref{thm:mle_expansion} and \ref{thm:spec_expansion} is the following CLTs for any finite $k$-dimensional components of the estimators. We start with the MLE $\hat{\theta}$. Without loss of generality, we state it for the first $k$ components. 
\begin{proposition}\label{prop:app_clt_mle}
Let $\hat{\theta}$ be the MLE in (\ref{def:mle}) and suppose the conditions of Theorem \ref{thm:mle_expansion} hold. Then for any $\theta^*\in\Theta(\kappa)$ and fixed $k\in[n]$, we have 
\begin{align}\label{eq:mle_clt}
\Big(\rho_1(\bar{\theta})(\hat{\theta}_1 - \theta^*_1),\ldots,\rho_k(\bar{\theta})(\hat{\theta}_k - \theta^*_k)\Big)\leadsto \mathcal{N}_k(0,I_k)
\end{align}
for both $\bar{\theta}\in\{\hat{\theta},\theta^*\}$. Here the sequence $\{\rho_i(\theta)\}_{i=1}^n$ can be either
\begin{align*}
\rho_i(\theta) = \sqrt{L\cdot \sum_{j:j\neq i}A_{ij}\psi'(\theta_i - \theta_j)} \quad \text{ or }\quad \sqrt{pL\cdot \sum_{j:j\neq i}\psi'(\theta_i - \theta_j)}.
\end{align*}
\end{proposition}

The proof of the above result is presented in Section \ref{subsec:proof_app_clt}. Using the completely data-driven version of $\rho_i(\hat{\theta})$, the above result can be used to produce multiple and simultaneous confidence intervals for finite-dimensional contrasts of the truth vector $\theta^*$. 

We note that previous works \cite{simons1999asymptotics,han2020asymptotic} that consider the asymptotic normality of the MLE are based on a different identifiability condition. Namely, they assume that there are $n+1$ individuals with merit parameters $(\theta_0^*,\ldots,\theta_n^*)$, $\theta_0^*$ and $\hat{\theta}_0$ are set to be $0$ by convention, and asymptotics is studied on the rest of the coordinates $(\hat{\theta}_1,\ldots,\hat{\theta}_n)$. As we will now show, our CLT in Proposition \ref{prop:app_clt_mle} can be naturally transformed to be applicable in their setting as well. Therefore, our result is a strict improvement over previous results in the regime (\ref{def:regime}). 

Let $(\hat{\theta}_0,\ldots, \hat{\theta}_n)\in \argmin_{\theta\in\R^{n+1}: \bm{1}_{n+1}^\top \theta = 0} \ell_{n+1}(\theta)$ be the MLE under the truth $\bm{1}_n^\top (\theta_0^*,\ldots,\theta_n^*) = 0$, where $\ell_{n+1}(\theta)$ is the variant of (\ref{eq:mle_likelihood}) for $n+1$ individuals. By Proposition \ref{prop:app_clt_mle}, we have, heuristically speaking, $(\hat{\theta}_0-\theta_0^*,\ldots, \hat{\theta}_{k} - \theta_k^*)\leadsto \mathcal{N}_{k+1}(0,V^{-1}_{k+1}(\theta^*))$, where
\begin{align*}
V_{k+1}(\theta^*) = L\cdot \diag\Big(\sum_{\substack{\ell:0\leq \ell\leq n\\\ell\neq 0}}A_{0\ell}\psi'(\theta_0^*-\theta_\ell^*),\ldots, \sum_{\substack{\ell:0\leq \ell\leq n\\\ell\neq k}}A_{k\ell}\psi'(\theta_k^*-\theta_\ell^*)\Big);
\end{align*}
note that this representation is not completely rigorous since $V_{k+1}(\theta^*)$ depends on $n$. Now under the identifiability condition ($\theta^*_0 = 0$) in \cite{simons1999asymptotics}, the MLE for $(\theta_1^*,\ldots,\theta_k^*)$ becomes
\begin{align*}
\hat{\theta}^{\textrm{SY}}\equiv \big(\hat{\theta}_1 - \hat{\theta}_0, \ldots, \hat{\theta}_k - \hat{\theta}_0\big), 
\end{align*}
where `SY' stands for Simons-Yao. Then we have $\big(\hat{\theta}^{\textrm{SY}}_1 - \theta_1^*,\ldots, \hat{\theta}^{\textrm{SY}}_k - \theta_k^*\big) \leadsto \mathcal{N}_k\big(0, \big(V^{\textrm{SY}}(\theta^*)\big)^{-1}\big)$, with $\big(V^{\textrm{SY}}(\theta^*)\big)^{-1}\in\R^{k\times k}$ given by
\begin{align*}
\big(V^{\textrm{SY}}(\theta^*)\big)^{-1}_{ii} &= \frac{1}{L\cdot\sum_{\ell:0\leq \ell\leq n, \ell\neq i}A_{i\ell}\psi'(\theta_i^*-\theta_\ell^*)} + \frac{1}{L\cdot\sum_{1\leq\ell\leq n}A_{0\ell}\psi'(\theta_\ell^*)},\\
\big(V^{\textrm{SY}}(\theta^*)\big)^{-1}_{ij} &= \frac{1}{L\cdot\sum_{1\leq\ell\leq n}A_{0\ell}\psi'(\theta_\ell^*)} \quad\quad \textrm{if $i\neq j$}.
\end{align*}
In the special case $p=1$, this recovers \cite[Theorem 2]{simons1999asymptotics}.

Similar to Proposition \ref{prop:app_clt_mle}, we have the following finite-dimensional CLT for the spectral estimator $\tilde{\theta}$. 
\begin{proposition}\label{prop:app_clt_spec}
Let $\tilde{\theta}$ be the spectral estimator $\tilde{\theta}$ defined in (\ref{def:spec}) and suppose the conditions of Theorem \ref{thm:spec_expansion} hold. Then for any $\theta^*\in\Theta(\kappa)$ and fixed $k\in[n]$, we have
\begin{align}\label{eq:spec_clt}
\Big(\bar{\rho}_1(\bar{\theta})(\tilde{\theta}_1 - \theta^*_1),\ldots,\bar{\rho}_k(\bar{\theta})(\tilde{\theta}_k - \theta^*_k)\Big)\leadsto \mathcal{N}_k(0,I_k)
\end{align}
for both $\bar{\theta}\in\{\tilde{\theta},\theta^*\}$. Here the sequence $\{\bar{\rho}_i(\theta)\}_{i=1}^n$ can be either
\begin{align*}
\bar{\rho}_i(\theta) = \sqrt{L\cdot \frac{\Big(\sum_{j:j\neq i}A_{ij}(e^{\theta_i} + e^{\theta_j})\psi'(\theta_i - \theta_j)\Big)^2}{\sum_{j:j\neq i}A_{ij}(e^{\theta_i}+e^{\theta_j})^2\psi'(\theta_i-\theta_j)}}
\end{align*}
or its deterministic version with all $\{A_{ij}\}_{i\neq j}$ replaced by $p$.
\end{proposition}

A direct application of Cauchy-Schwarz yields that
\begin{align*}
\bar{\rho}_i(\theta) \leq \rho_i(\theta),
\end{align*}
and hence the spectral estimator has a larger asymptotic variance than the MLE. We will rigorously justify the optimality of the MLE using the local minimax framework in Section \ref{subsec:app_constant} ahead.

\subsection{Application II: Confidence regions in ranking}\label{subsec:app_rank}

Consider the BTL model in the context of sports analytics, where after observing the outcome of a tournament, one wishes to construct a confidence interval for the rank of her team of interest. More concretely, let $r^{-1}(1),\ldots,r^{-1}(n)$ be the indices of the teams with the highest merit parameter, the second highest merit parameter, and so on, so that
\begin{align*}
\theta^*_{r^{-1}(1)} \geq \theta^*_{r^{-1}(2)} \geq \ldots \geq \theta^*_{r^{-1}(n)},
\end{align*}
where ties are broken arbitrarily. Correspondingly, $\{r(1),\ldots, r(n)\}$ is the rank vector of teams $1$ to $n$, and we are interested in building a (discrete) confidence interval for the rank of a pre-fixed team. Without loss of generality, we choose the first team with merit parameter $\theta^*_1$ (not necessarily the largest one). 

Since rank is a global object, we need to construct confidence intervals for all $n$ merit parameters, instead of just for $\theta^*_1$. To this end, a straightforward idea is to use the $\ell_\infty$ bound in Proposition \ref{prop:mle_prelim}. This approach has two disadvantages: (i) the constant $C$ therein is implicit and could potentially be large after complicated steps of analysis; (ii) the confidence intervals for all teams will have the same length, which is not ideal since certain teams will have more comparisons than others so their confidence intervals should heuristically be shorter. 

We now describe a different procedure based on the non-asymptotic expansions in Section \ref{sec:main}. The key is to exploit the explicit main terms $\{b_i/d_i\}$ therein to obtain short confidence intervals with data-driven lengths. With a fixed confidence level $1-\alpha$, we will construct confidence intervals such that the target $\theta_1^*$ belongs to its interval with asymptotical probability $1-\alpha$, and all other teams belong to their respective (slightly more conservative) intervals with overwhelming probability. For $\theta^*_1$, let
\begin{align*}
\mathcal{C}_1 \equiv \big[\hat{\theta}_1 - \rho_1(\hat{\theta})z_{1-\alpha/2}, \hat{\theta}_1 + \rho_1(\hat{\theta})z_{1-\alpha/2}\big],
\end{align*}
where $z_{1-\alpha/2}$ is the normal quantile such that $\Prob(\mathcal{N}(0,1)\geq z_{1-\alpha/2}) = \alpha/2$. It follows directly from the CLT in Proposition \ref{prop:app_clt_mle} that $\theta^*_1\in \mathcal{C}_1$ with asymptotic probability $1-\alpha$. For all other $i\neq 1$, let
\begin{align*}
\tau_i \equiv (1+c_0)\sqrt{2\log n\cdot \rho_i^{-2}(\hat{\theta})},
\end{align*}
where $\rho_i(\theta)$ denotes the sample version (with $\{A_{ij}\}$) defined in Proposition \ref{prop:app_clt_mle} and $c_0$ is an arbitrarily small but fixed constant. We then define the confidence intervals for $\{\hat{\theta}_i\}_{i=2}^n$ to be
\begin{align*}
\mathcal{C}_i\equiv \big[\hat{\theta_i} - \tau_i, \hat{\theta}_i + \tau_i\big].
\end{align*}
The constant in $\tau_i$ comes simply from the application of a (conditional) Bernstein's inequality to the main term $b_i/d_i$ in the expansion of Theorem \ref{thm:mle_expansion}, so that by construction $\theta_i^*\in \mathcal{C}_i$ simultaneously for $i\in\{2,\ldots,n\}$ with overwhelming probability. Since $\rho_i^{-2}$ is smaller for those teams $i$ with more comparisons, the above confidence intervals have the desirable property of data-driven length. 

The above construction of confidence intervals leads naturally to a confidence interval for the rank $r(1)$. We say two intervals satisfy $\mathcal{C}_i\leq \mathcal{C}_j$ if the upper end of $\mathcal{C}_i$ is smaller than the lower end of $\mathcal{C}_j$. Let $n_1$ be the number of $\mathcal{C}_i$'s such that $\mathcal{C}_1\leq \mathcal{C}_i$, and $n_2$ be the number such that $\mathcal{C}_i\leq \mathcal{C}_1$. By definition, $n_1,n_2$ take integer values between $0$ and $n-1$. The confidence interval for the rank $r(1)$ is then taken to be all integers inside
\begin{align*}
[n_1+1, n-n_2].
\end{align*}
The following result guarantees the confidence level of the above interval. Details of its proof are given in Section \ref{subsec:proof_app_rank}. 
\begin{proposition}\label{prop:app_rank}
Suppose the conditions of Theorem \ref{thm:mle_expansion} hold. Then with asymptotic probability at least $1-\alpha$, we have $r(1)\in [n_1+1, n-n_2]$. 
\end{proposition}


\subsection{Application III: Optimal constant of $\ell_2$ estimation}\label{subsec:app_constant}
With the main expansion in Theorem \ref{thm:mle_expansion}, we are able to pin down the exact constants of $\ell_2$ errors of the MLE $\hat{\theta}$ and the spectral estimator $\tilde{\theta}$. The proofs of the following two results are given in Sections \ref{subsec:proof_app_constant_mle} and \ref{subsec:proof_app_constant_spec} respectively.
\begin{proposition}\label{prop:mle_constant}
Suppose the conditions in (\ref{def:regime}) hold with $\alpha = 1$. Then it holds with probability at least $1 - \mathcal{O}(n^{-10})$ that
\begin{align*}
\pnorm{\hat{\theta}-\theta^*}{}^2 &= \frac{1+\mathfrak{o}(1)}{pL}\cdot\sum_{i=1}^n\Big(\sum_{k:k\neq i}\psi^\prime(\theta_i^* - \theta^*_k)\Big)^{-1},
\end{align*}
where the $\mathfrak{o}(1)$ term is uniform over $\theta^*\in\Theta(\kappa)$ in (\ref{def:theta_space}). 
\end{proposition}
\begin{proposition}\label{prop:spec_constant}
Suppose the conditions in (\ref{def:regime}) hold with $\alpha = 1$. Then it holds with probability at least $1 - \mathcal{O}(n^{-10})$ that
\begin{align*}
\pnorm{\tilde{\theta}-\theta^*}{}^2 &= \frac{1+\mathfrak{o}(1)}{pL}\cdot \sum_{i=1}^n \frac{\sum_{j:j\neq i}(e^{\theta_i^*}+e^{\theta_j^*})^2\psi'(\theta_i^*-\theta_j^*)}{\Big(\sum_{j:j\neq i}(e^{\theta_i^*} + e^{\theta_j^*})\psi'(\theta_i^*-\theta_j^*)\Big)^2}.
\end{align*}
where the $\mathfrak{o}(1)$ term is uniform over $\theta^*\in\Theta(\kappa)$ in (\ref{def:theta_space}). 
\end{proposition}

\begin{remark}
Note that as opposed to the main expansions in Section \ref{sec:main}, the above results only require the (almost) minimal sparsity condition $np\gg \log n$ for the underlying Erd\H{o}s-R\'{e}nyi graph to be connected. This is because the above results only require the $\ell_2$ control of the remainder vector $\delta$ in (\ref{def:delta}) instead of the $\ell_\infty$ control; see Sections \ref{subsec:proof_delta_l2} and \ref{subsec:proof_app_constant_mle} for details in the case of MLE. 
\end{remark}

The above results are much more refined than existing $\ell_2$ guarantees in the literature \cite{negahban2017rank,chen2019spectral,chen2020partial}, whose optimalities are all based on a (global) minimax framework. In particular, they allow us to compare the two estimators $\hat{\theta}$ and $\tilde{\theta}$ on a more fine-grained level. For example, while both $\hat{\theta}$ and $\tilde{\theta}$ achieve the (global) minimax rate $1/\sqrt{pL}$ in terms of $\ell_2$ accuracy \cite{negahban2017rank}, a direct application of Cauchy-Schwarz yields that the error of the spectral estimator in Proposition \ref{prop:spec_constant} is always larger than that of the MLE in Proposition \ref{prop:mle_constant}, and equality only holds in the case where all $\theta_i^*$'s are equal.


In fact, the following result states that the $\ell_2$ accuracy of the MLE cannot be improved asymptotically in terms of the local minimax framework. For any $\theta^*\in\R^n$ and $\epsilon > 0$, we use $B(\theta^*,\epsilon)$ to denote the hyperrectangle $\times_{i=1}^n [\theta^*_i- \epsilon, \theta^*_i + \epsilon]$.

\begin{theorem}\label{thm:lower}
Let $\kappa = \mathcal{O}(1)$ and $\epsilon_n$ be any sequence such that $\epsilon_n \gg (npL)^{-1/2}$. Then for any $\theta^*$ such that $\max_i\theta^*_i - \min_i\theta^*_i \leq \kappa/2$, it holds that
\begin{align*}
\inf_{\hat{\theta}}\sup_{\theta\in B(\theta^*,\epsilon_n)\cap \Theta(\kappa)} \E_\theta\pnorm{\hat{\theta}  - \theta}{}^2\geq \big(1+\mathfrak{o}(1)\big)\cdot \frac{1}{pL}\sum_{i=1}^n \Big(\sum_{j:j\neq i}\psi^\prime(\theta^*_i - \theta^*_j)\Big)^{-1},
\end{align*}
where the infimum of $\hat{\theta}$ is taken over all estimators of $\theta$.
\end{theorem}
\begin{remark}\label{remark:lower_radius}
It is clear that the localization radius $(npL)^{-1/2}$ cannot be improved in terms of its order. Indeed, if localized in a neighborhood of $\theta^*$ of size $\delta_n(npL)^{-1/2}$ for some $\delta_n = \mathfrak{o}(1)$, then the trivial estimator $\hat{\theta}_{\text{fix}}\equiv \theta^*$ satisfies $\E_\theta\pnorm{\hat{\theta}_{\text{fix}} - \theta}{}^2 = \mathfrak{o}\big((pL)^{-1}\big)$ uniformly over this neighborhood.
\end{remark}

The proof of Theorem \ref{thm:lower} is given in Section \ref{subsubsec:proof_lower} and is based on van Trees' inequality \cite{gill1995applications}. Together with Proposition \ref{prop:mle_constant}, they imply that the MLE is asymptotically and locally minimax optimal. This agrees with the conventional wisdom in classical parametric statistics that MLE is the most efficient estimator \cite{van2000asymptotic}.

\subsection{Simulation}\label{subsec:simulation}

We close this section with some simulation studies to verify the theoretical results above. Throughout the simulation, we set the number of comparison to be $L = 1$ and the dynamic range to be $\kappa = 2$. We first verify the finite-dimensional CLTs established in Section \ref{subsec:app_clt}. For a given number $n$ of individuals, we sample $\theta^* = (\theta_1^\ast, \ldots,\theta_n^\ast)^\top$ i.i.d. from $\textrm{Unif}([0,2])$, and examine the distribution of $\hat{\theta}_1$ and $\tilde{\theta}_1$ prescribed by Propositions \ref{prop:app_clt_mle} and \ref{prop:app_clt_spec}. For numerical reasons, we use the slightly larger success probability $p = (\log n)^{3}/n$ for the Erd\H{o}s-R\'{e}nyi comparison graph, and the corresponding QQ-plots against the nominal normal quantiles are presented in Figure \ref{fig:dist}. We see that starting from $n = 500$, both QQ-plots align very well with the diagonal line, validating both CLTs established in Propositions \ref{prop:app_clt_mle} and \ref{prop:app_clt_spec}.

\begin{figure}
	\centering
	\begin{subfigure}[b]{0.48\textwidth}
		\centering
		\includegraphics[width=\textwidth]{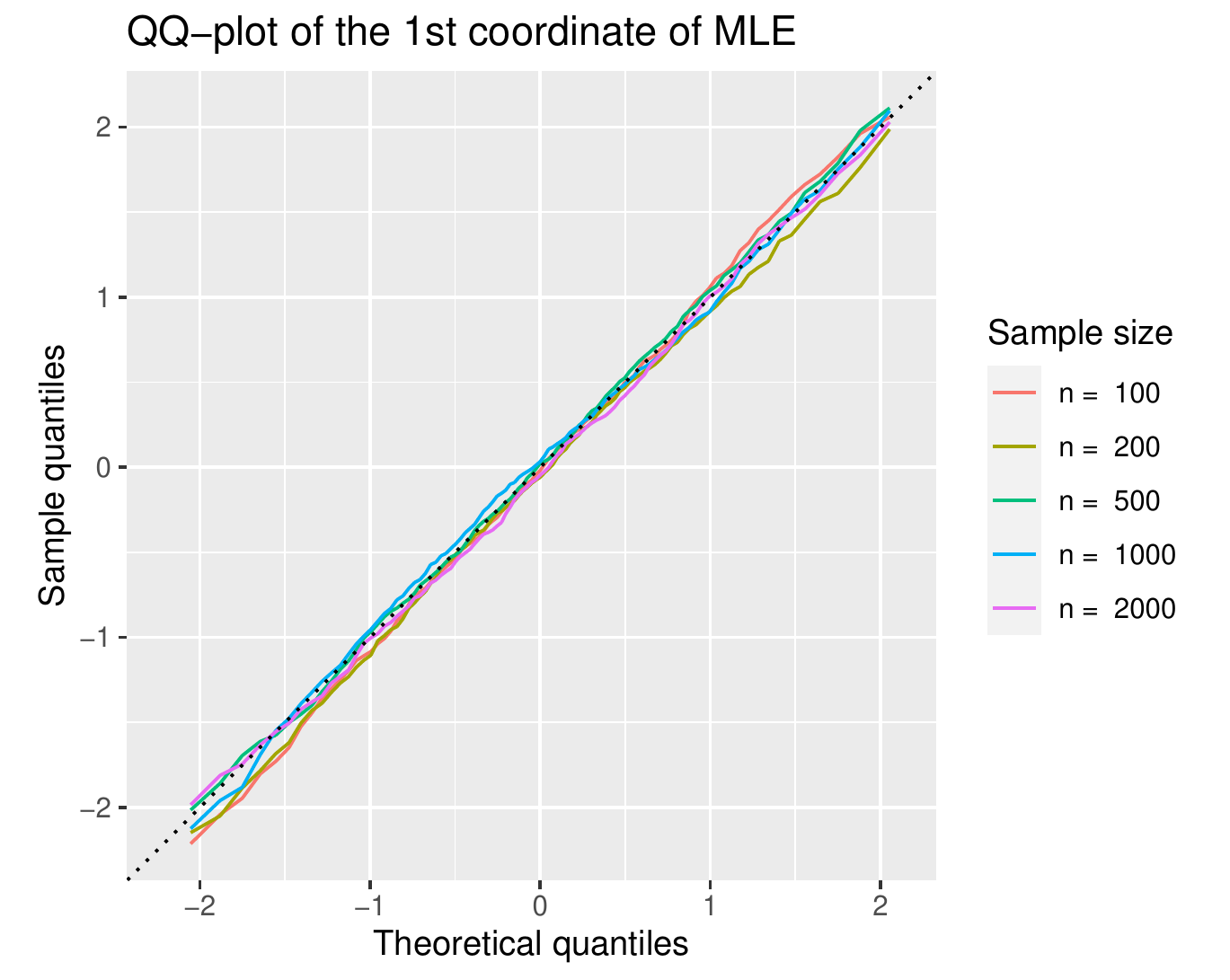}
		\caption{MLE}
		\label{fig:y equals x}
	\end{subfigure}
	\hfill
	\begin{subfigure}[b]{0.48\textwidth}
		\centering
		\includegraphics[width=\textwidth]{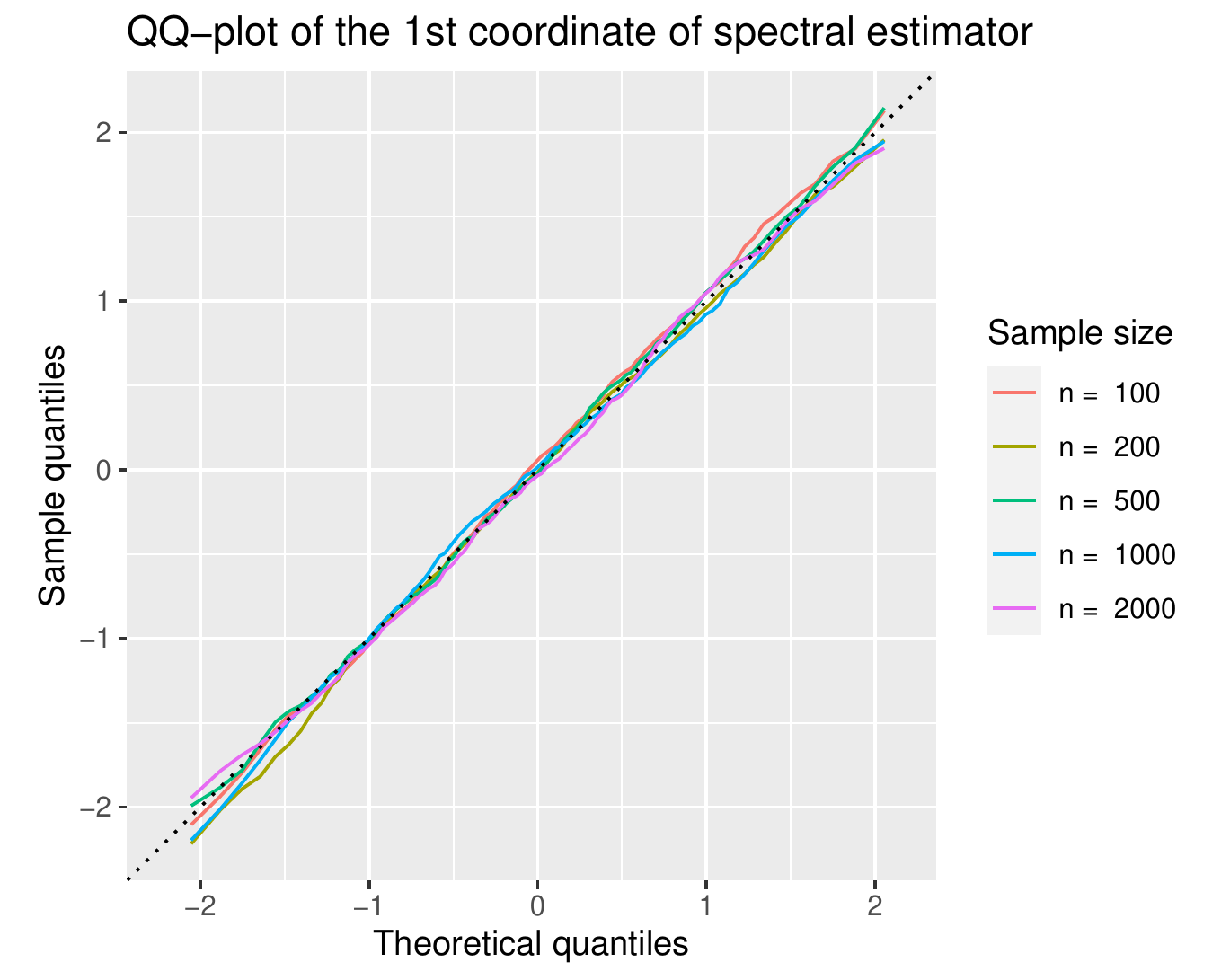}
		\caption{Spectral estimator}
		\label{fig:three sin x}
	\end{subfigure}
	\caption{QQ-plots comparing the theoretical and sample quantiles of $\hat{\theta}_1$ and $\tilde{\theta}_1$. Simulation parameters: $L = 1$, $n \in \{100, 200, 500, 1000, 2000\}$ and $p = (\log n)^{3}/n$. For each $n$, $\theta^\ast = (\theta_1^\ast, \ldots,\theta_n^\ast)^\top$ are i.i.d. from $\text{Unif}([0,2])$. The empirical quantile curves are averaged over $1000$ replications.}
	\label{fig:dist}
\end{figure}

Next we validate the $\ell_2$ risk predictions prescribed by Propositions \ref{prop:mle_constant} and \ref{prop:spec_constant}. For numerical reasons, we take $p = (\log n)^{3.5}/n$ to generate the Erd\H{o}s-R\'{e}nyi comparison graph. From the risk plot in Figure \ref{fig:risk}, we observe that the theoretical predictions in Propositions \ref{prop:mle_constant} and \ref{prop:spec_constant} are very accurate, and also that the MLE indeed achieves a smaller risk than the spectral estimator. 

\begin{figure}
\centering
\includegraphics[width = 0.48\textwidth]{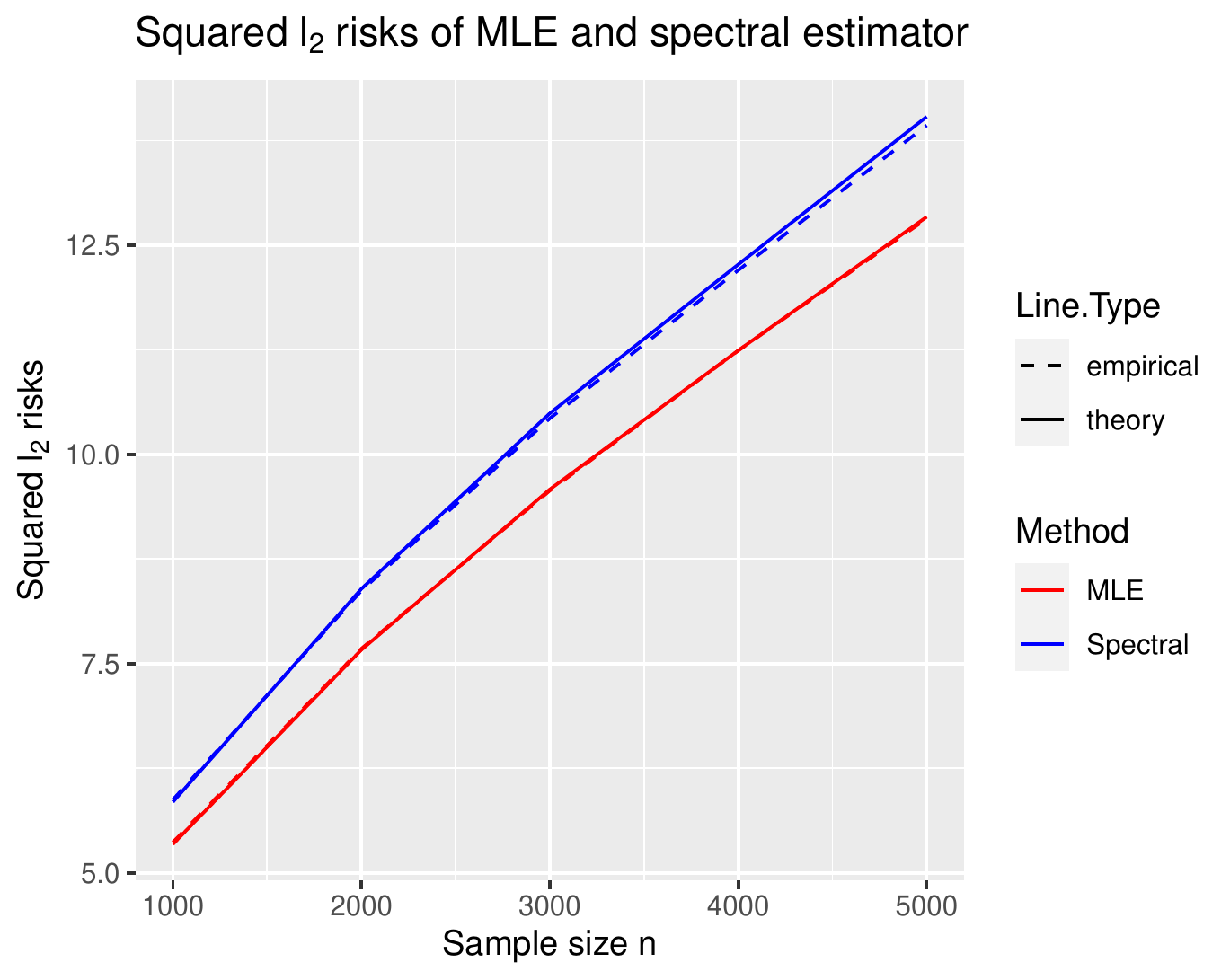}
\caption{Squared $\ell_2$ risks of the MLE $\hat{\theta}$ (red) and the spectral estimator $\tilde{\theta}$ (blue), with solid lines denoting the theoretical predictions by Proposition \ref{prop:mle_constant} and \ref{prop:spec_constant},  and dashed lines denoting the empirical values. Simulation parameters: $L = 1$, $n \in \{1000, 2000, 3000, 4000, 5000\}$ and $p = (\log n)^{3.5}/n$. For each $n$, $\theta^\ast = (\theta_1^\ast, \ldots,\theta_n^\ast)^\top$ are i.i.d. from $\text{Unif}([0,2])$. The empirical risk curves are averaged over $500$ replications.}
\label{fig:risk}
\end{figure}

\section{Proof of the main expansion for MLE}\label{sec:proof_main}

The goal of this section is to prove Theorem \ref{thm:mle_expansion}. We introduce some preliminaries in Section \ref{subsec:mle_prelim}, followed by two main steps of the proof in Sections \ref{subsec:proof_delta_l2} and \ref{subsec:proof_delta_infty} respectively. We refer to Section \ref{sec:proof_idea} for the general proof idea.

\subsection{Preliminary}\label{subsec:mle_prelim}

We will analyze $\hat{\theta}$ in a coordinate-wise manner. Fix any $i\in[n]$, and decompose the total likelihood in (\ref{eq:mle_likelihood}) as $\ell_n(\theta) = \ell_n^{(-i)}(\theta_{-i}) + \ell_n^{(i)}(\theta_i|\theta_{-i})$, where
\begin{align*}
\ell_n^{(-i)}(\theta_{-i}) &\equiv \sum_{j<k:j,k\neq i}A_{jk}\Big[\bar{y}_{jk}\log\frac{1}{\psi(\theta_j - \theta_k)} + \bar{y}_{kj}\log\frac{1}{\psi(\theta_k - \theta_j)}\Big],\\
\ell_n^{(i)}(\theta_i|\theta_{-i}) &\equiv \sum_{j:j\neq i}A_{ij}\Big[\bar{y}_{ij}\log\frac{1}{\psi(\theta_i - \theta_j)} + \bar{y}_{ji}\log\frac{1}{\psi(\theta_j - \theta_i)}\Big].
\end{align*}
Next define
\begin{align}\label{def:fg}
\notag f^{(i)}(\theta_i|\theta_{-i}) &\equiv \frac{\partial}{\partial \theta_i}\ell_n^{(i)}(\theta_i|\theta_{-i}) = -\sum_{j:j\neq i}A_{ij}\big(\bar{y}_{ij} - \psi(\theta_i - \theta_j)\big),\\
g^{(i)}(\theta_i|\theta_{-i}) &\equiv \frac{\partial^2}{\partial\theta_i^2}\ell_n^{(i)}(\theta_i|\theta_{-i}) = \sum_{j:j\neq i}A_{ij}\psi'(\theta_i - \theta_j).
\end{align}
Some preliminary estimates regarding $\hat{\theta}$, $f^{(i)}(\theta_i^*|\hat{\theta}_{-i})$ and $g^{(i)}(\theta_i^*|\hat{\theta}_{-i})$ are summarized in Lemmas \ref{lem:pre_estimate} and \ref{lem:bound_fg} in the Appendix. Let
\begin{align}\label{def:theta_bar}
\bar{\theta}_i &\equiv \theta_i^* - \frac{f^{(i)}(\theta_i^*|\hat{\theta}_{-i})}{g^{(i)}(\theta_i^*|\hat{\theta}_{-i})},
\end{align}
whose definition is motivated by the local quadratic expansion in (\ref{eq:mle_local_quad}). Lastly define the remainder vector $\delta\in\R^n$ via
\begin{align}\label{def:delta}
\hat{\theta}_i - \theta_i^* = -\frac{f^{(i)}(\theta_i^*|\theta^*_{-i})}{g^{(i)}(\theta_i^*|\theta^*_{-i})} + \delta_i.
\end{align}
We have the following $\ell_\infty$ bound for $\delta$.
\begin{proposition}\label{prop:delta_entrywise_bound}
Suppose that $\kappa = \mathcal{O}(1)$ and $np\gg (\log n)^{3/2}$. The following holds with probability $1 - \mathcal{O}(n^{-10})$. 
\begin{align*}
\max_{i}\Big|\delta_i - \epsilon_i\frac{f^{(i)}(\theta_i^*|\theta^*_{-i})}{g^{(i)}(\theta_i^*|\theta^*_{-i})}\Big| = \mathfrak{o}(\sqrt{\frac{1}{npL}}).
\end{align*}
Here $\epsilon\in\R^n$ satisfies $\pnorm{\epsilon}{\infty} = \mathfrak{o}(1)$.
\end{proposition}
Combining the above result with (\ref{def:delta}) immediately completes the proof.
The rest of this section is devoted to proving Proposition \ref{prop:delta_entrywise_bound}. As an essential intermediate step, we need to first give a sharp $\ell_2$ bound for (a leave-one-out version of) $\delta$. We present the $\ell_2$ bound in the next subsection and the main proof of Proposition \ref{prop:delta_entrywise_bound} in Section \ref{subsec:proof_delta_infty} ahead. 

\subsection{Control of $\pnorm{\delta}{}$}\label{subsec:proof_delta_l2}

The goal of this subsection is to prove the following $\ell_2$ bound for $\delta$ defined in (\ref{def:delta}). 
\begin{proposition}\label{prop:delta_bound}
Suppose that $\kappa = \mathcal{O}(1)$ and $np\gg \log n$. Then it holds with probability $1 - \mathcal{O}(n^{-10})$ that
\begin{align*}
\pnorm{\delta}{}\lesssim \frac{1}{\sqrt{pL}}\cdot\bigg[\sqrt{\frac{\log n}{np}}+\Big(\frac{\log n}{npL}\Big)^{1/4}\bigg] = \mathfrak{o}(\frac{1}{\sqrt{pL}}).
\end{align*} 
\end{proposition}
\begin{proof}[Proof of Proposition \ref{prop:delta_bound}]
Fix $m\in[n]$. By definition, $\delta_m$ can be decomposed as follows:
\begin{align}\label{eq:delta_decomp}
\notag\delta_m &= \frac{f^{(m)}(\theta_m^*|\theta^*_{-m})}{g^{(m)}(\theta_m^*|\theta^*_{-m})} - \frac{f^{(m)}(\theta_m^*|\hat{\theta}_{-m})}{g^{(m)}(\theta_m^*|\hat{\theta}_{-m})} + \hat{\theta}_m - \bar{\theta}_m\\
\notag&= \frac{f^{(m)}(\theta_m^*|\theta^*_{-m}) - f^{(m)}(\theta_m^*|\hat{\theta}_{-m})}{g^{(m)}(\theta_m^*|\hat{\theta}_{-m})} \\
&\quad\quad + f^{(m)}(\theta_m^*|\theta^*_{-m})\Big(\frac{1}{g^{(m)}(\theta_m^*|\theta^*_{-m})} - \frac{1}{g^{(m)}(\theta_m^*|\hat{\theta}_{-m})}\Big) + \hat{\theta}_m - \bar{\theta}_m.
\end{align}
For the first term in the above display, Taylor expansion yields that
\begin{align*}
&f^{(m)}(\theta_m^*|\theta^*_{-m})-f^{(m)}(\theta_m^*|\hat{\theta}_{-m}) = \sum_{i:i\neq m}A_{mi}\big(\psi(\theta_m^* - \theta_i^*) - \psi(\theta_m^* - \hat{\theta}_i)\big)\\
&= \sum_{i:i\neq m}A_{mi}\psi'(\theta_m^* - \theta_i^*)(\hat{\theta}_i - \theta_i^*) - \frac{1}{2}\sum_{i:i\neq m}A_{mi}\psi''(\theta_m^* - \xi_{m,i})(\hat{\theta}_i - \theta_i^*)^2,
\end{align*}
where for each $i\neq m$, $\xi_{m,i}$ is some number between $\hat{\theta}_i$ and $\theta_i^*$. Now using again the decomposition of $\hat{\theta}_m$ in (\ref{def:delta}), we have 
\begin{align*}
&f^{(m)}(\theta_m^*|\theta^*_{-m})-f^{(m)}(\theta_m^*|\hat{\theta}_{-m})\\
&= -\sum_{i:i\neq m}A_{mi}\psi'(\theta_m^* - \theta_i^*)\frac{f^{(i)}(\theta_i^*|\theta^*_{-i})}{g^{(i)}(\theta_i^*|\theta^*_{-i})} + \sum_{i:i\neq m}A_{mi}\psi'(\theta_m^* - \theta_i^*)\delta_i \\
&\quad\quad\quad-\frac{1}{2}\sum_{i:i\neq m}A_{mi}\psi''(\theta_m^* - \xi_{m,i})(\hat{\theta}_i - \theta_i^*)^2.
\end{align*}
Go back to the definition of $\delta_m$, and the above calculation yields that
\begin{align}\label{eq:delta_consistent}
\notag g^{(m)}(\theta_m^*|\hat{\theta}_{-m})\delta_m &= \big(f^{(m)}(\theta_m^*|\theta^*_{-m})-f^{(m)}(\theta_m^*|\hat{\theta}_{-m})\big) \\
\notag&\quad\quad +f^{(m)}(\theta_m^*|\theta^*_{-m})\Big(\frac{g^{(m)}(\theta_m^*|\hat{\theta}_{-m})}{g^{(m)}(\theta_m^*|\theta^*_{-m})}-1\Big) + g^{(m)}(\theta_m^*|\hat{\theta}_{-m})\big(\hat{\theta}_m - \bar{\theta}_m\big)\\
\notag&= -\sum_{i:i\neq m}A_{mi}\psi'(\theta_m^* - \theta_i^*)\frac{f^{(i)}(\theta_i^*|\theta^*_{-i})}{g^{(i)}(\theta_i^*|\theta^*_{-i})} + \sum_{i:i\neq m}A_{mi}\psi'(\theta_m^* - \theta_i^*)\delta_i\\
\notag&\quad\quad -\frac{1}{2}\sum_{i:i\neq m}A_{mi}\psi''(\theta_m^* - \xi_{m,i})(\hat{\theta}_i - \theta_i^*)^2 \\
&\quad\quad  +f^{(m)}(\theta_m^*|\theta^*_{-m})\Big(\frac{g^{(m)}(\theta_m^*|\hat{\theta}_{-m})}{g^{(m)}(\theta_m^*|\theta^*_{-m})}-1\Big) + g^{(m)}(\theta_m^*|\hat{\theta}_{-m})\big(\hat{\theta}_m - \bar{\theta}_m\big).
\end{align}
Rearranging the terms yields that
\begin{align*}
&g^{(m)}(\theta_m^*|\theta^*_{-m})\delta_m - \sum_{i:i\neq m}A_{mi}\psi'(\theta_m^* - \theta_i^*)\delta_i + \big(g^{(m)}(\theta_m^*|\hat{\theta}_{-m}) - g^{(m)}(\theta_m^*|\theta^*_{-m})\big)\delta_m\\
&= -\sum_{i:i\neq m}A_{mi}\psi'(\theta_m^* - \theta_i^*)\frac{f^{(i)}(\theta_i^*|\theta^*_{-i})}{g^{(i)}(\theta_i^*|\theta^*_{-i})}-\frac{1}{2}\sum_{i:i\neq m}A_{mi}\psi''(\theta_m^* - \xi_{m,i})(\hat{\theta}_i - \theta_i^*)^2\\
&\quad\quad + f^{(m)}(\theta_m^*|\theta^*_{-m})\Big(\frac{g^{(m)}(\theta_m^*|\hat{\theta}_{-m})}{g^{(m)}(\theta_m^*|\theta^*_{-m})}-1\Big) + g^{(m)}(\theta_m^*|\hat{\theta}_{-m})\big(\hat{\theta}_m - \bar{\theta}_m\big)\\
&\equiv R_{1,m} + R_{2,m} + R_{3,m} + R_{4,m}.
\end{align*}
Here $R_1$-$R_4$ are vectors in $\R^n$. Let $H$ be the Laplacian matrix defined by $H_{ii} \equiv \sum_{j:j\neq i}A_{ij}\psi'(\theta_i^* - \theta_j^*)$ and $H_{ij}\equiv -A_{ij}\psi'(\theta_i^* - \theta_j^*)$ for $i\neq j$, and $D = \textrm{diag}(D_1,\ldots,D_n)$ be a diagonal matrix with $D_m\equiv g^{(m)}(\theta_m^*|\hat{\theta}_{-m}) - g^{(m)}(\theta_m^*|\theta^*_{-m})$. Then in matrix form, 
\begin{align}\label{eq:decomp_delta}
(H+D)\delta = R_1 + R_2 + R_3 + R_4.
\end{align}

Note that: (i) $\lambda_{\min,\perp}(H)\gtrsim np$ with the prescribed probability by Lemma \ref{lem:hessian_eigen}; (ii) By Lemma \ref{lem:bound_fg}, the operator norm of $D$ satisfies 
\begin{align*}
\pnorm{D}{\op} = \max_{m\in[n]}|g^{(m)}(\theta_m^*|\hat{\theta}_{-m}) - g^{(m)}(\theta_m^*|\theta^*_{-m})| \lesssim \sqrt{\frac{np}{L}} + \sqrt{\frac{(\log n)^3}{npL}} = \mathfrak{o}(np).
\end{align*}
Hence for large enough $n$, with
\begin{align}\label{def:avg}
\ave(u) \equiv \frac{1}{n}\sum_{i=1}^n u_i, \quad \forall u\in\R^n,
\end{align}
the left side of (\ref{eq:decomp_delta}) can be lower bounded by
\begin{align}\label{ineq:delta_l2_lower}
\notag&\pnorm{H\delta}{} - \pnorm{D}{\op}\pnorm{\delta}{}= \pnorm{H(\delta - \ave(\delta)\bm{1}_n)}{} - \mathfrak{o}(np)\pnorm{\delta}{}\\
& \gtrsim np\pnorm{\delta - \ave(\delta)\bm{1}_n}{} -\mathfrak{o}(np)\pnorm{\delta}{}\geq \frac{np}{2}\pnorm{\delta}{}-np\sqrt{n}|\ave(\delta)|.
\end{align}  
Hence by (\ref{eq:decomp_delta}), Lemma \ref{lem:delta_avg} for the bound for $|\ave(\delta)|$, and Lemma \ref{lem:R_norm} for the bounds for $\pnorm{R_1}{}$-$\pnorm{R_4}{}$, we have
\begin{align*}
\pnorm{\delta}{}&\lesssim \sqrt{n}|\ave(\delta)| + (np)^{-1}\cdot\big(\pnorm{R_1}{} + \pnorm{R_2}{} + \pnorm{R_3}{} + \pnorm{R_4}{}\big)\\
&\lesssim \frac{1}{\sqrt{pL}}\cdot\sqrt{\frac{\log n}{n}} + (np)^{-1}\cdot\bigg[\sqrt{\frac{n\log n}{L}} + \Big(\frac{n}{L}\sqrt{\frac{np\log n}{L}}\Big)^{1/2}\bigg]\\
&\asymp  \frac{1}{\sqrt{pL}}\cdot\bigg[\sqrt{\frac{\log n}{np}}+\Big(\frac{\log n}{npL}\Big)^{1/4}\bigg],
\end{align*}
as desired.
\end{proof}

\begin{lemma}\label{lem:delta_avg}
Suppose that $\kappa = \mathcal{O}(1)$ and $np\gg \log n$. Then the following holds with probability at least $1-\mathcal{O}(n^{-10})$.
\begin{align*}
|\ave(\delta)| \lesssim \frac{1}{\sqrt{npL}}\cdot\sqrt{\frac{\log n}{n}} = \mathfrak{o}(\frac{1}{\sqrt{npL}}).
\end{align*}
\end{lemma}
\begin{proof}[Proof of Lemma \ref{lem:delta_avg}]
Since both $\hat{\theta}$ and $\theta^*$ are centered by definition, the definition of $\delta$ in (\ref{def:delta}) yields that
\begin{align*}
\ave(\delta) &= \frac{1}{n}\sum_{m=1}^n\frac{f^{(m)}(\theta_m^*|\theta^*_{-m})}{g^{(m)}(\theta^*_m|\theta^*_{-m})}.
\end{align*}
Let $w_{mi} \equiv A_{mi}/\big(\sum_{j:j\neq m}A_{mj}\psi'(\theta_m^*-\theta_j^*)\big)$. Then
\begin{align*}
\ave(\delta) &= \frac{1}{n}\sum_{m=1}^n \frac{\sum_{i:i\neq m}A_{mi}\big(\bar{y}_{mi} - \psi(\theta_m^*-\theta_i^*)\big)}{\sum_{i:i\neq m}A_{mi}\psi'(\theta_m^* -\theta_i^*)}\\
&= \frac{1}{n}\sum_{i\neq m}w_{mi}\big(\bar{y}_{mi} - \psi(\theta_m^*-\theta_i^*)\big)\\
&= \frac{1}{n}\sum_{i<m} w_{mi}\big(\bar{y}_{mi} - \psi(\theta_m^*-\theta_i^*)\big) +  w_{im}\big(\bar{y}_{im} - \psi(\theta_i^*-\theta_m^*)\big)\\
&\stackrel{(*)}{=} \frac{1}{n}\sum_{i<m} (w_{mi}-w_{im})\big(\bar{y}_{mi} - \psi(\theta_m^*-\theta_i^*)\big)\\
&= \frac{1}{nL}\sum_{i<m}\sum_{\ell=1}^L (w_{mi}-w_{im})\big(y_{mi\ell} - \psi(\theta_m^*-\theta_i^*)\big),
\end{align*}
where $(*)$ follows as $\bar{y}_{im} - \psi(\theta_i^*-\theta_m^*) = (1-\bar{y}_{mi}) - (1-\psi(\theta_m^*-\theta_i^*)) = -(\bar{y}_{mi} - \psi(\theta_m^*-\theta_i^*))$. Now conditioning on the graph $A$, $\{w_{mi}\}$ are deterministic and the summands in the above display are independent across $i<m, \ell$. Hence Hoeffding's inequality yields that for any $t > 0$,
\begin{align*}
\Prob\Big(\Big|\frac{1}{n}\sum_{m=1}^n\frac{f^{(m)}(\theta_m^*|\theta^*_{-m})}{g^{(m)}(\theta^*_m|\theta^*_{-m})}\Big| \geq t\Big| A\Big) \leq 2\exp\Big(-\frac{Cn^2Lt^2}{\sum_{i<m}(w_{mi} - w_{im})^2}\Big).
\end{align*}
By Lemma \ref{lem:A_concen}, it holds with the prescribed probability that
\begin{align*}
\sum_{i<m}w_{mi}^2 = \sum_{i<m}\Big(\frac{A_{mi}}{\sum_{j:j\neq m}A_{mj}\psi'(\theta_m^*-\theta_j^*)}\Big)^2 \lesssim (np)^{-2}\sum_{i<m}A_{mi} \lesssim p^{-1}.
\end{align*}
Since a similar bound holds for $\sum_{i<m}w_{im}^2$, we have $\sum_{i<m}(w_{mi} - w_{im})^2 \lesssim p^{-1}$ with the prescribed probability. By choosing $t\asymp \sqrt{\log n/(n^2pL)}$, the above two displays yield that with the prescribed probability $|\ave(\delta)|\lesssim n^{-1}\sqrt{\log n/(Lp)} = \sqrt{\log n /n}\cdot (1/\sqrt{npL})$. The proof is complete.
\end{proof}

\begin{lemma}\label{lem:R_norm}
Recall $R_1$-$R_4$ defined in the proof of Proposition \ref{prop:delta_bound}. Suppose $\kappa = \mathcal{O}(1)$ and $np\gg \log n$. Then the following hold with probability $1 - \mathcal{O}(n^{-10})$.
\begin{align*}
\pnorm{R_1}{}^2\vee \pnorm{R_2}{}^2\vee \pnorm{R_3}{}^2\lesssim \frac{n\log n}{L},\quad \pnorm{R_4}{}^2\lesssim \frac{n}{L}\sqrt{\frac{np\log n}{L}}.
\end{align*}
\end{lemma}
\begin{proof}[Proof of Lemma \ref{lem:R_norm}]
(\textbf{Bound of $R_1$}) 
For each $m\in[n]$, let
\begin{align*}
w_{ij;m}\equiv \frac{\psi'(\theta_m^*-\theta_i^*)A_{mi}A_{ij}}{\sum_{k:k\neq i}A_{ik}\psi'(\theta_i^*-\theta_k^*)}.
\end{align*}
Then we have
\begin{align*}
R_{1,m}&= \sum_{i:i\neq m}A_{mi}\psi'(\theta_m^*-\theta_i^*)\frac{\sum_{j:j\neq i}A_{ij}\big(\bar{y}_{ij} - \psi(\theta_i^*-\theta_j^*)\big)}{\sum_{k:k\neq i}A_{ik}\psi'(\theta_i^*-\theta_k^*)}\\
&= \sum_{i,j:j\neq i, i\neq m} \Big(\frac{\psi'(\theta_m^*-\theta_i^*)A_{mi}A_{ij}}{\sum_{k\neq i}A_{ik}\psi'(\theta_i^*-\theta_k^*)}\Big)\big(\bar{y}_{ij}-\psi(\theta_i^*-\theta_j^*)\big)\\
&= \sum_{i,j:j\neq i, i,j\neq m} w_{ij;m}\big(\bar{y}_{ij}-\psi(\theta_i^*-\theta_j^*)\big) + \sum_{i:i\neq m}w_{im;m}\big(\bar{y}_{im}-\psi(\theta_i^*-\theta_m^*)\big)\\
&= \sum_{i,j:i<j, i,j\neq m} w_{ij;m}\big(\bar{y}_{ij}-\psi(\theta_i^*-\theta_j^*)\big) + w_{ji;m}\big(\bar{y}_{ji}-\psi(\theta_j^*-\theta_i^*)\big) \\
&\quad\quad\quad +\sum_{i:i\neq m}w_{im;m}\big(\bar{y}_{im}-\psi(\theta_i^*-\theta_m^*)\big)\\
&\stackrel{(*)}{=} \frac{1}{L}\sum_{i,j:i<j, i,j\neq m}\sum_{\ell=1}^L (w_{ij;m} - w_{ji;m})\big(y_{ij\ell}-\psi(\theta_i^*-\theta_j^*)\big) \\
&\quad\quad\quad+\frac{1}{L}\sum_{i:i\neq m}\sum_{\ell=1}^L w_{im;m}\big(y_{im\ell}-\psi(\theta_i^*-\theta_m^*)\big).
\end{align*}
Here $(*)$ follows as $\bar{y}_{ji}-\psi(\theta_j^*-\theta_i^*) = (1-\bar{y}_{ij}) - (1-\psi(\theta_i^*-\theta_j^*)) = -(\bar{y}_{ij} - \psi(\theta_i^*-\theta_j^*))$. Now conditioning on the graph $A$, $\{w_{ij;m}\}$ are deterministic and the above two sums are independent across their summands. Hence Hoeffding's inequality yields that 
\begin{align*}
&\Big|\frac{1}{L}\sum_{i,j:i<j, i,j\neq m}\sum_{\ell=1}^L (w_{ij;m} - w_{ji;m})\big(y_{ij\ell}-\psi(\theta_i^*-\theta_j^*)\big)\Big|\\
&\quad\quad\quad\lesssim \sqrt{\frac{\log n}{L}\cdot \sum_{i,j:i<j, i,j\neq m}(w_{ij;m} - w_{ji;m})^2},\\
&\Big|\frac{1}{L}\sum_{i:i\neq m}\sum_{\ell=1}^L w_{im;m}\big(y_{im\ell}-\psi(\theta_i^*-\theta_m^*)\big)\Big| \lesssim \sqrt{\frac{\log n}{L}\cdot \sum_{i:i\neq m}w_{im;m}^2}
\end{align*}
with the prescribed probability. Now by Lemma \ref{lem:A_concen}, we have
\begin{align*}
\sum_{i,j:i<j, i,j\neq m}(w_{ij;m} - w_{ji;m})^2 &\lesssim \sum_{i,j:i<j, i,j\neq m}\frac{A_{ij}A_{mi}}{\big(\sum_{k:k\neq i}A_{ik}\psi'(\theta_i^*-\theta_k^*)\big)^2}\\
&\lesssim (np)^{-2}\sum_{i,j:i<j, i,j\neq m}A_{ij}A_{mi} \lesssim 1,\\
\sum_{i:i\neq m}w_{im;m}^2 &\lesssim \sum_{i:i\neq m}\frac{A_{im}}{\big(\sum_{k:k\neq i}A_{ik}\psi'(\theta_i^*-\theta_k^*)\big)^2}\lesssim \frac{1}{np}.
\end{align*}
Combining the above two displays yields that $\max_{m\in[n]}|R_{1,m}|\lesssim \sqrt{\log n/L}$, and hence $\pnorm{R_1}{} \lesssim \sqrt{n\log n/L}$ with the prescribed probability.


\medskip
\noindent (\textbf{Bound of $R_2$}) Using $\sup_{u\in\R}|\psi''(u)|\lesssim 1$, we have $|R_{2,m}| \lesssim (\sum_{i:i\neq m}A_{im}|\hat{\theta}_i - \theta_i^*|)\cdot \pnorm{\hat{\theta} - \theta^*}{\infty}$. Hence by Lemma  \ref{lem:hessian_eigen} and Proposition \ref{prop:mle_prelim}, we have
\begin{align*}
\pnorm{R_2}{}^2 &\lesssim \sum_{m=1}^n \Big(\sum_{i:i\neq m}A_{mi}|\hat{\theta}_i - \theta_i^*|\Big)^2\cdot \pnorm{\hat{\theta} - \theta^*}{\infty}^2\\
&\leq \pnorm{A}{\op}^2 \cdot \pnorm{\hat{\theta} - \theta^*}{}^2\cdot \pnorm{\hat{\theta} - \theta^*}{\infty}^2\\
&\lesssim (np)^2 \cdot \frac{1}{pL} \cdot \frac{\log n}{npL} = \frac{n\log n}{L^2}.
\end{align*}
\medskip
\noindent (\textbf{Bound of $R_3$}) By Lemma \ref{lem:bound_fg}, it holds with the prescribed probability that
\begin{align*}
|R_{3,m}| &\leq \Big|f^{(m)}(\theta_m^*|\theta^*_{-m})\Big|\cdot \frac{\max_{m\in[n]}\Big|g^{(m)}(\theta_m^*|\theta_{-m}^*) - g^{(m)}(\theta_m^*|\hat{\theta}_{-m})\Big|}{g^{(m)}(\theta_m^*|\theta_{-m}^*)}\\
&\lesssim \Big|f^{(m)}(\theta_m^*|\theta^*_{-m})\Big|\cdot \frac{1}{np}\Big(\sqrt{\frac{np}{L}} + \sqrt{\frac{(\log n)^3}{npL}}\Big).
\end{align*}
Hence again by Lemma \ref{lem:bound_fg},
\begin{align*}
\pnorm{R_3}{}^2&\lesssim \sum_{m=1}^n \Big(f^{(m)}(\theta_m^*|\theta^*_{-m})\Big)^2 \cdot \frac{1}{(np)^2}\Big(\frac{np}{L} + \frac{(\log n)^3}{npL}\Big)\\
&\lesssim \frac{n^2p}{L}\cdot \frac{1}{(np)^2}\Big(\frac{np}{L} + \frac{(\log n)^3}{npL}\Big) = \frac{n}{L^2} + \frac{n\log n}{L} \cdot \Big(\frac{\log n}{np}\Big)^2.
\end{align*}

\medskip
\noindent (\textbf{Bound of $R_4$}) Using the bound $g^{(m)}(\theta_m^*|\theta_{-m}^*) = \sum_{i:i\neq m}A_{mi}\psi^{\prime}(\theta_m^* - \theta_i^*)\lesssim np$ by Lemma \ref{lem:A_concen}, it follows from Lemma \ref{lem:theta_bar_close} that
\begin{align*}
\Big[g^{(m)}(\theta_m^*|\theta_{-m}^*)(\hat{\theta}_m - \bar{\theta}_m)\Big]^2 &\lesssim (np)^2\Big(\frac{|f^{(m)}(\theta_m^*|\hat{\theta}_{-m})|}{np}\Big)^3 = \frac{\big|f^{(m)}(\theta_m^*|\hat{\theta}_{-m})\big|^3}{np}.
\end{align*}
Moreover, it follows from Lemma \ref{lem:bound_fg}-(4) that $g^{(m)}(\theta_m^*|\hat{\theta}_{-m}) = \big(1+\mathfrak{o}(1)\big)g^{(m)}(\theta_m^*|\hat{\theta}_{-m})$. Hence by Lemma \ref{lem:bound_fg}-(1), we have
\begin{align*}
\pnorm{R_4}{}^2 &= \sum_{m=1}^n \Big[g^{(m)}(\theta_m^*|\hat{\theta}_{-m})(\hat{\theta}_m - \bar{\theta}_m)\Big]^2\\
&\lesssim \sum_{m=1}^n \Big[g^{(m)}(\theta_m^*|\theta^*_{-m})(\hat{\theta}_m - \bar{\theta}_m)\Big]^2\lesssim \frac{1}{np}\sum_{m=1}^n \Big|f^{(m)}(\theta_m^*|\hat{\theta}_{-m})\Big|^3\\
&\leq \frac{1}{np}\cdot \max_{m\in[n]}\Big|f^{(m)}(\theta_m^*|\hat{\theta}_{-m})\Big|\cdot \sum_{m=1}^n \Big|f^{(m)}(\theta_m^*|\hat{\theta}_{-m})\Big|^2\\
&\lesssim \frac{1}{np}\cdot\sqrt{\frac{np\log n}{L}}\cdot \frac{n^2p}{L} = \frac{n}{L}\sqrt{\frac{np\log n}{L}}.
\end{align*}
The proof is complete.
\end{proof}

\subsection{Control of $\pnorm{\delta}{\infty}$}\label{subsec:proof_delta_infty}
We will now prove Proposition \ref{prop:delta_entrywise_bound}. We start by introducing the leave-two-out technique and some additional notation in Section \ref{subsubsec:leave_two_out}, followed by some preliminary estimates of the corresponding quantities in Section \ref{subsubsec:loo_delta_bound}. The main proof of Proposition \ref{prop:delta_entrywise_bound} is given in Section \ref{subsubsec:delta_infty_main}.

\subsubsection{Leave-two-out preliminary}\label{subsubsec:leave_two_out}

Recall that $\delta$ satisfies the following self-consistent equation in (\ref{eq:delta_consistent}): for each $i\in[n]$,
\begin{align*}
g^{(i)}(\theta_i^*|\hat{\theta}_{-i})\delta_i = \sum_{j:j\neq i}A_{ij}\psi'(\theta_i^*-\theta_j^*)\delta_j + \bar{R}_i.
\end{align*}
for some error term $\bar{R}\in\R^n$. In the previous section, we derived the bound for $\pnorm{\delta}{}$ by treating the above system as a whole and solving for $\delta$. Here for $\pnorm{\delta}{\infty}$, we need to bound the right side directly, and the key difficulty lies in controlling the term $\sum_{j:j\neq i}A_{ij}\psi'(\theta_i^*-\theta_j^*)\delta_j$. By standard concentration arguments, this term could be properly controlled if the sequence $\{A_{ij}\}_{j:j\neq i}$ was independent of $\{\delta_j\}_{j:j\neq i}$. For this reason, we introduce in the following a leave-one-out version $\delta^{(i)}$ of $\delta$ that is independent of the $i$th individual (and therefore $\{A_{ij}\}_{j:j\neq i}$). 

Fix any index $i\in[n]$ to be left out. Recall the leave-one-out likelihood
\begin{align*}
\ell_n^{(-i)}(\theta_{-i}) = \sum_{j<k:j,k\neq i}A_{jk}\Big[\bar{y}_{jk}\log\frac{1}{\psi(\theta_j - \theta_k)} + \bar{y}_{kj}\log\frac{1}{\psi(\theta_k - \theta_j)}\Big].
\end{align*}
The leave-one-out MLE $\hat{\theta}^{(i)}\in\R^{n-1}$ is defined by
\begin{align}\label{def:loo_mle}
\hat{\theta}^{(i)}\equiv \argmin_{\theta_{-i}\in\R^{n-1}:\ave(\theta_{-i}) = \ave(\theta^*_{-i})}\ell_n^{(-i)}(\theta_{-i}).
\end{align}
We start by finding an approximation $\bar{\theta}_j^{(i)}$ of $\hat{\theta}^{(i)}_j$ for each $j\neq i$, analogous to the way $\bar{\theta}_j$ approximates $\hat{\theta}_j$. By isolating the $j$th individual, we have the following leave-two-out decomposition $\ell_n^{(-i)}(\theta_{-i}) = \ell_n^{(-i,-j)}(\theta_{-i,-j}) + \ell_n^{(-i,j)}(\theta_j|\theta_{-i,-j})$, where
\begin{align*}
\ell_n^{(-i,-j)}(\theta_{-i,-j}) &= \sum_{k<\ell:k,\ell\notin \{i,j\}}A_{k\ell}\Big[\bar{y}_{k\ell}\log\frac{1}{\psi(\theta_k - \theta_\ell)} + \bar{y}_{\ell k}\log\frac{1}{\psi(\theta_\ell - \theta_k)}\Big],\\
\ell_n^{(-i,j)}(\theta_j|\theta_{-i,-j}) &= \sum_{k\notin \{i,j\}}A_{jk}\Big[\bar{y}_{jk}\log\frac{1}{\psi(\theta_j - \theta_k)} + \bar{y}_{kj}\log\frac{1}{\psi(\theta_k - \theta_j)}\Big].
\end{align*}
Analogous to $f^{(i)}(\theta_i|\theta_{-i})$ and $g^{(i)}(\theta_i|\theta_{-i})$, define for each $j\neq i$
\begin{align}\label{def:loo_fg}
\notag f^{(-i,j)}(\theta_j|\theta_{-i,-j}) &\equiv \frac{\partial}{\partial \theta_j}\ell_n^{(-i,j)}(\theta_j|\theta_{-i,-j}) = -\sum_{k\notin \{i,j\}}A_{jk}\big(\bar{y}_{jk} - \psi(\theta_j - \theta_k)\big),\\
g^{(-i,j)}(\theta_j|\theta_{-i,-j}) &\equiv \frac{\partial^2}{\partial\theta_j^2}\ell_n^{(-i,j)}(\theta_j|\theta_{-i,-j}) = \sum_{k\notin \{i,j\}}A_{jk}\psi'(\theta_j - \theta_k).
\end{align}

Using a similar reasoning for $\bar{\theta}$, define for each $j\neq i$
\begin{align*}
\bar{\theta}^{(i)}_j\equiv \theta_j^* - \frac{f^{(-i,j)}(\theta_j^*|\hat{\theta}^{(i)}_{-j})}{g^{(-i,j)}(\theta_j^*|\hat{\theta}^{(i)}_{-j})} 
\end{align*}
This leads to the definition of a leave-one-out version of $\delta$: for each $j\neq i$, let $\delta^{(i)}_j$ be defined by
\begin{align}\label{eq:loo_delta}
\hat{\theta}^{(i)}_j - \theta_j^* &\equiv - \frac{f^{(-i,j)}(\theta_j^*|\theta^*_{-i,-j})}{g^{(-i,j)}(\theta_j^*|\theta^*_{-i,-j})} + \delta^{(i)}_j.
\end{align}

The following lemma is an analogue of Lemma \ref{lem:bound_fg} for the quantities $f^{(-i,j)}(\theta_j|\theta_{-i,-j})$ and $g^{(-i,j)}(\theta_j|\theta_{-i,-j})$. Its proof relies on the analysis of the leave-two-out MLE defined as
\begin{align}\label{def:lto_mle}
\hat{\theta}^{(i,j)}\equiv \argmax_{\theta_{-i,-j}\in \R^{n-2}: \ave(\theta_{-i,-j}) = \ave(\theta^*_{-i,-j})} \ell_n^{(-i,-j)} (\theta_{-i,-j}).
\end{align}
\begin{lemma}\label{lem:bound_loo_fg}
Suppose $\kappa = \mathcal{O}(1)$ and $np\gg \log n$. Then the following hold with probability at least $1 - \mathcal{O}(n^{-10})$.
\begin{enumerate}
\item There exists some $C = C(\kappa)  > 0$ such that 
\begin{align*}
\max_{i\in[n]}\max_{j:j\neq i}\big|f^{(-i,j)}(\theta_j^*|\hat{\theta}^{(i)}_{-j})\big| \leq C\sqrt{\frac{np\log n}{L}}
\end{align*}
and 
\begin{align*}
\max_{i\in[n]}\sum_{j:j\neq i} \big(f^{(-i,j)}(\theta_j^*|\hat{\theta}^{(i)}_{-j})\big)^2 \leq  C\frac{n^2p}{L}.
\end{align*}
\item There exist some positive $c=c(\kappa)$ and $C = C(\kappa)$ such that 
\begin{align*}
cnp\leq \min_{i\in[n]}\min_{j:j\neq i} g^{(-i,j)}(\theta_i^*|\theta^*_{-i,-j}) \leq \max_{i\in[n]}\max_{j:j\neq i} g^{(-i,j)}(\theta_i^*|\theta^*_{-i,-j}) \leq Cnp.
\end{align*}
\item There exists some $C = C(\kappa)  > 0$ such that 
\begin{align*}
\max_{i\in[n]}\max_{j:j\neq i}\Big|g^{(-i,j)}(\theta_j^*|\hat{\theta}^{(i)}_{-j}) - g^{(-i,j)}(\theta_j^*|\theta^*_{-i,-j})\Big| \leq C\Big(\sqrt{\frac{np}{L}} + \sqrt{\frac{(\log n)^3}{npL}}\Big).
\end{align*}
\end{enumerate}
\end{lemma}
\begin{proof}[Proof of Lemma \ref{lem:bound_loo_fg}]
(1) 
For the first inequality, fix any $i\in[n]$ and then $j$ such that $j\neq i$. Then
\begin{align*}
-f^{(-i,j)}(\theta^*_j|\hat{\theta}^{(i)}_{-j}) &= \sum_{k\notin \{i,j\}}A_{jk}\big(\bar{y}_{jk} - \psi(\theta_j^*-\hat{\theta}^{(i)}_k)\big)\\
&=  \sum_{k\notin \{i,j\}}A_{jk}\big(\bar{y}_{jk} - \psi(\theta_j^*-\theta^*_k)\big) +   \sum_{k\notin \{i,j\}}A_{jk}\big(\psi(\theta_j^*-\theta^*_k) - \psi(\theta_j^*-\hat{\theta}^{(i)}_k)\big)\\
&\equiv (I_{ij}) + (II_{ij}).
\end{align*}
By Lemma \ref{lem:mle_concentration}, we have $\max_{i\in[n]}\max_{j:j\neq i}|(I_{ij})| \lesssim \sqrt{np\log n/L}$ with the prescribed probability. On the other hand, as $\psi(\cdot)$ is Lipschitz with a universal constant, Lemmas \ref{lem:A_concen} and \ref{lem:pre_estimate} yield that
\begin{align*}
\max_{i\in[n]}\max_{j:j\neq i}|(II_{ij})| &= \max_{i\in[n]}\max_{j:j\neq i}\Big| \sum_{k\notin \{i,j\}}A_{jk}\big(\psi(\theta_j^*-\theta^*_k) - \psi(\theta_j^*-\hat{\theta}^{(i)}_k)\big)\Big|\\
& \leq \max_{i\in[n]}\pnorm{\hat{\theta}^{(i)}-\theta_{-i}^*}{\infty}\cdot \max_{i\in[n]}\max_{j:j\neq i}\sum_{k\notin\{i,j\}}A_{jk}\\
& \lesssim \sqrt{\frac{\log n}{npL}}\cdot np = \sqrt{\frac{np\log n}{L}},
\end{align*}
with the prescribed probability. Putting together the estimates concludes the first inequality.

For the second inequality, note that again by Lemma \ref{lem:mle_concentration}, $\max_{i\in[n]}\sum_{j:j\neq i}^n (I_{ij})^2 \lesssim n^2p/L$. On the other hand, by Taylor expansion,
\begin{align*}
(II_{ij}) = \sum_{k\notin\{i,j\}}A_{jk}\psi'(\theta_j^* - \xi_{ijk})(\hat{\theta}^{(i)}_k - \theta_k^*)
\end{align*}
for some $\xi_{ijk}$ between $\hat{\theta}^{(i)}_k$ and $\theta_k^*$. Hence using $\sup_{u\in\R}|\psi'(u)|\lesssim 1$, we have
\begin{align*}
\max_{i\in[n]}\sum_{j:j\neq i} (II_{ij})^2 &= \max_{i\in[n]}\sum_{j:j\neq i}\Big(\sum_{k\notin\{i,j\}}A_{jk}\psi'(\theta_j^* - \xi_{ijk})(\hat{\theta}^{(i)}_k - \theta_k^*)\Big)^2\\
&\lesssim \pnorm{A}{\op}^2\cdot\max_{i\in[n]}\pnorm{\hat{\theta}^{(i)} - \theta^*_{-i}}{}^2\stackrel{(*)}{\lesssim} (np)^2\cdot\frac{1}{pL} = \frac{n^2p}{L},
\end{align*}
where $(*)$ is by Lemmas \ref{lem:hessian_eigen} and \ref{lem:pre_estimate}.

\medskip
\noindent (2) This follows directly from Lemma \ref{lem:A_concen} and the fact that $\psi'(\theta_i -\theta_j)\asymp 1$ for any $\theta\in\Theta(\kappa)$.

\medskip
\noindent (3) Fix $i\in[n]$ and then $j$ such that $j\neq i$. Recall the definition of the leave-two-out MLE $\hat{\theta}^{(i,j)}$ in (\ref{def:lto_mle}). Then using analogous arguments for the leave-one-out MLE $\hat{\theta}^{(i)}$, we have the following estimates for the leave-two-out MLE $\hat{\theta}^{(i,j)}$ (cf. Lemma \ref{lem:pre_estimate}):
\begin{align*}
\max_{i,j:j\neq i}\pnorm{\hat{\theta}^{(i,j)} - \theta_{-i,-j}^*}{}\lesssim \sqrt{\frac{1}{pL}},\quad \max_{i,j:j\neq i}\pnorm{\hat{\theta}^{(i,j)} - \theta_{-i,-j}^*}{\infty}\lesssim \sqrt{\frac{\log n}{npL}}
\end{align*}
and
\begin{align*}
\max_{i,j:j\neq i}\pnorm{\hat{\theta}^{(i,j)} - \hat{\theta}^{(i)}_{-j}}{}\lesssim \sqrt{\frac{1}{npL}}.
\end{align*}
Using the above estimates and Lemma \ref{lem:A_concen}, we have
\begin{align*}
\notag&\big|g^{(-i,j)}(\theta_j^*|\hat{\theta}^{(i)}_{-j}) - g^{(-i,j)}(\theta_j^*|\theta^*_{-i,-j})\big|\\
\notag&= \big|\sum_{k\notin\{i,j\}}A_{jk}\big(\psi'(\theta_j^* -\theta_k^*)- \psi'(\theta_j^* -\hat{\theta}^{(i)}_k)\big)\big|\\
\notag&\lesssim\sum_{k\notin\{i,j\}}A_{jk}|\hat{\theta}^{(i)}_k - \theta_k^*| \leq \sum_{k\notin\{i,j\}}A_{jk}|\hat{\theta}^{(i)}_k - \hat{\theta}^{(i,j)}_k| + \sum_{k\notin\{i,j\}}A_{jk}|\hat{\theta}^{(i,j)}_k - \theta_k^*|\\
\notag&\leq \big(\sum_{k\notin \{i,j\}}A_{jk}\big)^{1/2}\pnorm{\hat{\theta}^{(i)}_{-j} - \hat{\theta}^{(i,j)}}{} + p\cdot\sum_{k\notin\{i,j\}}|\hat{\theta}^{(i,j)}_k - \theta_k^*| +\sum_{k\notin\{i,j\}}(A_{jk}-p)|\hat{\theta}^{(i,j)}_k - \theta_k^*|\\
\notag&\leq  \sqrt{np}\sqrt{\frac{1}{npL}} + p\sqrt{n}\pnorm{\hat{\theta}^{(i,j)}-\theta_{-i,-j}^*}{} + \sqrt{p\log n}\pnorm{\hat{\theta}^{(i,j)}-\theta^*_{-i,-j}}{} + \log n\cdot \pnorm{\hat{\theta}^{(i,j)}-\theta^*_{-i,-j}}{\infty}\\
&\lesssim \sqrt{\frac{1}{L}} + \sqrt{\frac{np}{L}} + \sqrt{\frac{\log n}{L}} + \sqrt{\frac{(\log n)^3}{npL}} \asymp \sqrt{\frac{np}{L}} + \sqrt{\frac{(\log n)^3}{npL}}.
\end{align*}
The above bound is uniform over $i,j$ so the proof is complete.
\end{proof}


\subsubsection{Bounds for $\delta^{(i)}$}\label{subsubsec:loo_delta_bound}

We first establish the following analogue of Proposition \ref{prop:delta_bound} for $\pnorm{\delta^{(i)}}{}$. The proof is similar so we only sketch the steps. 
\begin{lemma}\label{lem:loo_delta_l2_bound}
Suppose $\kappa = \mathcal{O}(1)$ and $np\gg \log n$. Then the following holds with probability at least $1-\mathcal{O}(n^{-10})$.
\begin{align*}
\max_{m\in[n]}\pnorm{\delta^{(m)}}{} \lesssim \frac{1}{\sqrt{pL}}\cdot\bigg[\sqrt{\frac{\log n}{np}}+\Big(\frac{\log n}{npL}\Big)^{1/4}\bigg] = \mathfrak{o}(\frac{1}{\sqrt{pL}}).
\end{align*}
\end{lemma}
\begin{proof}[Proof of Lemma \ref{lem:loo_delta_l2_bound}]
Following the arguments in Proposition \ref{prop:delta_bound}, we arrive at the following expansion analogous to (\ref{eq:decomp_delta}):
\begin{align}\label{eq:decomp_loo_delta}
(H^{(m)} + D^{(m)})\delta^{(m)} = R^{(m)}_1 + R^{(m)}_2 + R^{(m)}_3 + R^{(m)}_4.
\end{align}
Here $H^{(m)}$, $D^{(m)}$, and $R^{(m)}_1$-$R^{(m)}_4$ are leave-one-out versions of $H$, $D$, and $R^1$-$R^4$ defined by
\begin{itemize}
\item $H^{(m)}\in\R^{(n-1)\times (n-1)}$ with $H^{(m)}_{ii} = \sum_{j:j\notin\{m,i\}}A_{ij}\psi'(\theta_i^*-\theta_j^*)$ for $i\neq m$ and $H^{(m)}_{ij} = -A_{ij}\psi'(\theta_i^*-\theta_j^*)$ for $j\notin\{m,i\}$.
\item $D^{(m)} \in\R^{(n-1)\times (n-1)}$ is defined by $D^{(m)}= \textrm{\diag}(D_1^{(m)},\ldots, D_{m-1}^{(m)}, D_{m+1}^{(m)},\ldots, D_n^{(m)})$ with $D^{(m)}_i = g^{(-m,i)}(\theta_i^*|\hat{\theta}^{(m)}_{-i}) - g^{(-m,i)}(\theta_i^*|\theta^*_{-m,-i})$.
\item $R^{(m)}_1$ - $R^{(m)}_4$ are defined by: for each $i\neq m$,
\begin{align*}
R^{(m)}_{1,i} &= -\sum_{j\notin\{m,i\}}A_{ij}\psi'(\theta_i^*-\theta_j^*)\frac{f^{(-m,j)}(\theta_j^*|\theta^*_{-m,-j})}{g^{(-m,j)}(\theta_i^*|\theta^*_{-m,-j})},\\
R^{(m)}_{2,i} &=-\frac{1}{2}\sum_{j\notin\{m,i\}}A_{ij}\psi''(\theta_i^*-\xi_{m,j})(\hat{\theta}^{(m)}_j - \theta_j^*)^2,\\
R^{(m)}_{3,i} &= f^{(-m,i)}(\theta_i^*|\hat{\theta}^{(m)}_{-i})\Big(\frac{g^{(-m,i)}(\theta_i^*|\hat{\theta}^{(m)}_{-i})}{g^{(-m,i)}(\theta_i^*|\theta^*_{-m,-i})} - 1\Big),\\
R^{(m)}_{4,i} &=g^{(-m,i)}(\theta_i^*|\hat{\theta}^{(m)}_{-i})\big(\hat{\theta}^{(m)}_i - \bar{\theta}^{(m)}_i\big).
\end{align*}
\end{itemize}
Note that: (i) $\bm{1}_{n-1}$ is in the null space of $H^{(m)}$ and the smallest non-null eigenvalue satisfies $\min_{x\in\R^{n-1}:\pnorm{x}{}=1, \ave(x) = 0}x^\top H^{(m)}x\gtrsim np$ by Lemma \ref{lem:hessian_eigen}; (ii) $\pnorm{D^{(m)}}{\op} = \mathfrak{o}(np)$ by Lemma \ref{lem:bound_loo_fg}-(2). Hence continuing with the arguments in Proposition \ref{prop:delta_bound} leads to
\begin{align*}
\pnorm{\delta^{(m)}}{}\lesssim \sqrt{n}|\ave(\delta^{(m)})| + (np)^{-1}\cdot\big(\pnorm{R^{(m)}_1}{} + \pnorm{R^{(m)}_2}{} + \pnorm{R^{(m)}_3}{} + \pnorm{R^{(m)}_4}{}\big).
\end{align*}
The quantities on the right side can be bounded as follows:
\begin{itemize}
\item Note that by definition $\ave(\hat{\theta}^{(m)}) = \ave(\theta^*_{-m})$, hence by analogous arguments as in Lemma \ref{lem:delta_avg}, we have
\begin{align*}
|\ave(\delta^{(m)})| &= \Big|\frac{1}{n-1}\sum_{i:i\neq m}\frac{f^{(-m,i)}(\theta_i^*|\theta^*_{-m,-i})}{g^{(-m,i)}(\theta_i^*|\theta^*_{-m,i})}\Big|\lesssim \frac{1}{\sqrt{npL}}\cdot\sqrt{\frac{\log n}{n}}.
\end{align*}
\item By analogous arguments as in Lemma \ref{lem:R_norm}, we have $\pnorm{R^{(m)}_1}{}\vee \pnorm{R^{(m)}_2}{} \vee\pnorm{R^{(m)}_3}{}  \lesssim \sqrt{n\log n/L}$.
\item By analogous arguments as in Lemma \ref{lem:theta_bar_close}, we have
\begin{align}\label{ineq:loo_theta_diff}
\big|\hat{\theta}^{(m)}_i - \bar{\theta}^{(m)}_i\big|\lesssim \Big(\frac{\big|f^{(-m,i)}(\theta_i^*|\hat{\theta}^{(m)}_{-i})\big|}{np}\Big)^{3/2}.
\end{align}
The bound $\pnorm{R^{(m)}_4}{}\lesssim \Big(n\sqrt{np\log n}/L^{3/2}\Big)^{1/2}$ now follows from subsequent arguments in Lemma \ref{lem:R_norm} upon using Lemma \ref{lem:bound_loo_fg}.
\end{itemize}
Collecting the estimates yields the desired claim. 
\end{proof}

The following control of $\pnorm{\delta^{(m)}}{\infty}$ is also needed in the bound for $\pnorm{\delta}{\infty}$ below.
\begin{lemma}\label{lem:loo_delta_max_bound}
Under the conditions $\kappa = \mathcal{O}(1)$ and $np\gg \log n$, the following holds with probability at least $1-\mathcal{O}(n^{-10})$.
\begin{align*}
\max_{m\in[n]}\pnorm{\delta^{(m)}}{\infty} \lesssim \sqrt{\frac{\log n}{npL}}.
\end{align*}
\end{lemma}
\begin{proof}[Proof of Lemma \ref{lem:loo_delta_max_bound}]
By combining (\ref{ineq:loo_theta_diff}) and Lemma \ref{lem:bound_loo_fg}-(1), we have
\begin{align}\label{ineq:hat_bar_theta}
\big|\hat{\theta}^{(m)}_i - \bar{\theta}^{(m)}_i\big|\lesssim \Big(\frac{\big|f^{(-m,i)}(\theta_i^*|\hat{\theta}^{(m)}_{-i})\big|}{np}\Big)^{3/2} \lesssim \Big(\frac{\log n}{npL}\Big)^{3/4}.
\end{align} 
Hence for any $i\neq m$, we have
\begin{align*}
|\delta^{(m)}_i| &\leq \Big|\frac{f^{(-m,i)}(\theta_i^*|\hat{\theta}^{(m)}_{-i})-f^{(-m,i)}(\theta_i^*|\theta^*_{-m,-i})}{g^{(-m,i)}(\theta_i^*|\hat{\theta}^{(m)}_{-i})}\Big| +  \big|\hat{\theta}^{(m)}_i - \bar{\theta}^{(m)}_i\big|\\
&\quad\quad\quad+\Big|f^{(-m,i)}(\theta_i^*|\theta^*_{-m,-i})\Big(\frac{1}{g^{(-m,i)}(\theta_m^*|\theta^*_{-i})}-\frac{1}{g^{(-m,i)}(\theta_m^*|\hat{\theta}^{(m)}_{-i})}\Big)\Big|\\
&\stackrel{(*)}{\leq} (np)^{-1}\sum_{j\notin\{m,i\}}A_{ij}|\hat{\theta}^{(m)}_j - \theta_j^*| +  \Big(\frac{\log n}{npL}\Big)^{3/4} \\
&\quad\quad\quad+\sqrt{\frac{np\log n}{L}}(np)^{-2}\Big(\sqrt{\frac{np}{L}} + \sqrt{\frac{(\log n)^3}{npL}}\Big)\\
&\lesssim (np)^{-1}\Big(\sum_{j\notin\{m,i\}}A_{ij}\Big)\cdot \pnorm{\hat{\theta}^{(m)}-\theta_{-m}^*}{\infty} + \Big(\frac{\log n}{npL}\Big)^{3/4} + \frac{\sqrt{\log n}}{npL} + \frac{(\log n)^2}{(np)^2L}\\
&\stackrel{(**)}{\lesssim} \sqrt{\frac{\log n}{npL}} +  \Big(\frac{\log n}{npL}\Big)^{3/4} + \frac{\sqrt{\log n}}{npL} + \frac{(\log n)^2}{(np)^2L} \asymp \sqrt{\frac{\log n}{npL}}.
\end{align*}
Here $(*)$ follows from (\ref{ineq:hat_bar_theta}) and Lemma \ref{lem:bound_loo_fg}, and $(**)$ follows from Lemma \ref{lem:pre_estimate}. The proof is complete.  
\end{proof}

\subsubsection{Main proof}\label{subsubsec:delta_infty_main}
\begin{proof}[Proof of Proposition \ref{prop:delta_entrywise_bound}]
Recall the following decomposition of $\delta$ in (\ref{eq:delta_decomp}): for each $m\in[n]$,
\begin{align*}
\delta_m &= \frac{f^{(m)}(\theta_m^*|\theta_{-m}^*) - f^{(m)}(\theta_m^*|\hat{\theta}_{-m})}{g^{(m)}(\theta_m^*|\hat{\theta}_{-m})} + (\hat{\theta}_m - \bar{\theta}_m)\\
&\quad\quad\quad +f^{(m)}(\theta_m^*|\theta^*_{-m})\Big(\frac{1}{g^{(m)}(\theta_m^*|\theta^*_{-m})}-\frac{1}{g^{(m)}(\theta_m^*|\hat{\theta}_{-m})}\Big).
\end{align*}
%
Recall the definition of the leave-one-out MLE $\hat{\theta}^{(m)}$ in (\ref{def:loo_mle}). Then $\delta_m$ can be further decomposed as
\begin{align*}
\delta_m &= \frac{\sum_{i:i\neq m}A_{mi}\big(\psi(\theta_m^*-\hat{\theta}^{(m)}_i) - \psi(\theta_m^*-\hat{\theta}_i)\big)}{g^{(m)}(\theta_m^*|\hat{\theta}_{-m})} \\
&\quad\quad\quad+\frac{\sum_{i:i\neq m}A_{mi}\big(\psi(\theta_m^*-\theta_i^*) - \psi(\theta_m^*-\hat{\theta}^{(m)}_i)\big)}{g^{(m)}(\theta_m^*|\hat{\theta}_{-m})} + (\hat{\theta}_m - \bar{\theta}_m)\\
&\quad\quad\quad +\frac{f^{(m)}(\theta_m^*|\theta^*_{-m})}{g^{(m)}(\theta_m^*|\theta^*_{-m})}\Big(1-\frac{g^{(m)}(\theta_m^*|\theta^*_{-m})}{g^{(m)}(\theta_m^*|\hat{\theta}_{-m})}\Big)\\
&\equiv \delta_{1,m} + \delta_{2,m} + \delta_{3,m} + \frac{f^{(m)}(\theta_m^*|\theta^*_{-m})}{g^{(m)}(\theta_m^*|\theta^*_{-m})}\Big(1-\frac{g^{(m)}(\theta_m^*|\theta^*_{-m})}{g^{(m)}(\theta_m^*|\hat{\theta}_{-m})}\Big).
\end{align*}
By Lemma \ref{lem:bound_fg}, we have $1-g^{(m)}(\theta_m^*|\theta^*_{-m})/g^{(m)}(\theta_m^*|\hat{\theta}_{-m}) = \mathfrak{o}(1)$ under the condition $np\gg \log n$. On the other hand, it follows from Lemma \ref{lem:delta_main} that under the condition $np\gg (\log n)^{3/2}$, we have $\max_{m\in[n]}|\delta_{1,m}|\vee |\delta_{2,m}| = \mathfrak{o}(1/\sqrt{npL})$ and
\begin{align*}
|\delta_{3,m}| &\lesssim  \epsilon_m\Big|\frac{f^{(m)}(\theta_m^*|\theta^*_{-m})}{g^{(m)}(\theta_m^*|\theta^*_{-m})}\Big| + \mathfrak{o}(\frac{1}{\sqrt{npL}}),
\end{align*}
where $\pnorm{\epsilon}{\infty} = \mathfrak{o}(1)$ and the $\mathfrak{o}(1/\sqrt{npL})$ is uniform over $m\in[n]$. The proof is now complete by collecting the estimates.
\end{proof}

\begin{lemma}\label{lem:delta_main}
Recall the definition of $\delta_{1,m}$-$\delta_{3,m}$ from the proof of Proposition \ref{prop:delta_entrywise_bound}. Suppose $\kappa = \mathcal{O}(1)$ and $np\gg \log n$. Then the following holds with probability at least $1-\mathcal{O}(n^{-10})$ uniformly over $m\in[n]$.
\begin{align*}
|\delta_{1,m}| &\lesssim \frac{1}{\sqrt{npL}}\sqrt{\frac{1}{np}},\\ 
|\delta_{2,m}| &\lesssim \frac{1}{\sqrt{npL}}\cdot \Big(\sqrt{\frac{\log n}{np}}  + \Big(\frac{\log n}{npL}\Big)^{1/4} + \sqrt{\frac{(\log n)^3}{(np)^2}}\Big),\\
|\delta_{3,m}| &\lesssim \epsilon_m\Big|\frac{f^{(m)}(\theta_m^*|\theta^*_{-m})}{g^{(m)}(\theta_m^*|\theta^*_{-m})}\Big|  + \frac{1}{\sqrt{npL}}\cdot\Big[\Big(\frac{\log n}{npL}\Big)^{1/4} + \frac{(\log n)^{7/4}}{(np)^{5/4}L^{1/4}}\Big].
\end{align*}
Here $\epsilon\in\R^n$ satisfies $\pnorm{\epsilon}{\infty} = \mathfrak{o}(1)$. In particular, if $np\gg (\log n)^{3/2}$, we have $\max_{m\in[n]}|\delta_{1,m}|\vee |\delta_{2,m}| = \mathfrak{o}(1/\sqrt{npL})$ and $|\delta_{3,m}| \lesssim \epsilon_m \big|f^{(m)}(\theta_m^*|\theta^*_{-m})/g^{(m)}(\theta_m^*|\theta^*_{-m})\big| + \mathfrak{o}(1/\sqrt{npL})$, where $\pnorm{\epsilon}{\infty} = \mathfrak{o}(1)$ and $\mathfrak{o}(1/\sqrt{npL})$ is uniform over $m\in[n]$.
\end{lemma}
\begin{proof}[Proof of Lemma \ref{lem:delta_main}]
(\textbf{Bound for $\delta_{1,m}$})
By Lemma \ref{lem:pre_estimate}, we have
\begin{align*}
|\delta_{1,m}| &\lesssim (np)^{-1}\sum_{i:i\neq m}A_{mi}\big|\hat{\theta}_i -\hat{\theta}^{(m)}_i\big|\leq (np)^{-1}\Big(\sum_{i:i\neq m}A_{mi}\Big)^{1/2}\pnorm{\hat{\theta}^{(m)}-\hat{\theta}_{-m}}{}\\
&\lesssim (np)^{-1}\cdot\sqrt{np}\cdot\sqrt{\frac{1}{npL}} =\frac{1}{\sqrt{npL}}\sqrt{\frac{1}{np}}.
\end{align*}
\medskip
\noindent (\textbf{Bound for $\delta_{2,m}$}) 
Using Taylor expansion and the expansion in (\ref{eq:loo_delta}) for each $\hat{\theta}^{(m)}_i$, $i\neq m$, we have
\begin{align*}
g^{(m)}(\theta_m^*|\hat{\theta}_{-m})\delta_{2,m} &= \sum_{i:i\neq m}A_{mi}\Big[\psi'(\theta_m^*-\theta_i^*)(\hat{\theta}^{(m)}_i - \theta_i^*) -\frac{1}{2}\psi''(\theta_m^*-\xi_{m,i})(\hat{\theta}^{(m)}_i - \theta_i^*)^2\Big]\\
&= -\sum_{i:i\neq m}A_{mi}\psi'(\theta_m^*-\theta_i^*)\frac{f^{(-m,i)}(\theta_i^*|\theta^*_{-m,-i})}{g^{(-m,i)}(\theta_i^*|\theta^*_{-m,-i})} \\
&\quad\quad + \sum_{i:i\neq m}A_{mi}\psi'(\theta_m^*-\theta_i^*)\delta^{(m)}_i -\frac{1}{2}\sum_{i:i\neq m}A_{mi}\psi''(\theta_m^*-\xi_{m,i})(\hat{\theta}^{(m)}_i - \theta_i^*)^2 \\
&\equiv R^{(m)}_{1} + R^{(m)}_{2} + R^{(m)}_{3}.
\end{align*}
The proof for $\delta_{2,m}$ is now complete by using Lemma \ref{lem:loo_R_norm} and the lower bound $g^{(m)}(\theta_m^*|\hat{\theta}_{-m})\gtrsim np$.

\noindent(\textbf{Bound for $\delta_{3,m}$}) It follows from Lemma \ref{lem:theta_bar_close} that
\begin{align*}
|\delta_{3,m}| = |\hat{\theta}_m - \bar{\theta}_m| &\lesssim \Big(\frac{\big|f^{(m)}(\theta_m^*|\hat{\theta}_{-m})\big|}{np}\Big)^{3/2}.
\end{align*}
By Lemma \ref{lem:bound_fg}-(1), we have $\big|f^{(m)}(\theta_m^*|\hat{\theta}_{-m})\big|/(np)\lesssim \sqrt{\log n/(npL)}$. Furthermore, we have by Lemma \ref{lem:bound_fg}-(2)
\begin{align*}
\big|f^{(m)}(\theta_m^*|\hat{\theta}_{-m})\big| = \big|f^{(m)}(\theta_m^*|\theta^*_{-m})\big| + \mathcal{O}\Big(\sqrt{\frac{np}{L}} + \sqrt{\frac{(\log n)^3}{npL}}\Big).
\end{align*}
Hence we have
\begin{align*}
|\delta_{3,m}| &\lesssim \Big(\frac{\log n}{npL}\Big)^{1/4}\cdot\Big[\frac{\big|f^{(m)}(\theta_m^*|\theta^*_{-m})\big|}{np} + \mathcal{O}\Big(\sqrt{\frac{1}{npL}} + \sqrt{\frac{(\log n)^3}{(np)^3L}}\Big)\Big]\\
&= \mathfrak{o}(1)\Big|\frac{f^{(m)}(\theta_m^*|\theta^*_{-m})}{g^{(m)}(\theta_m^*|\theta^*_{-m})}\Big| + \frac{1}{\sqrt{npL}}\cdot\Big[\Big(\frac{\log n}{npL}\Big)^{1/4} + \frac{(\log n)^{7/4}}{(np)^{5/4}L^{1/4}}\Big],
\end{align*}
as desired. The proof is complete.
%
\end{proof}

\begin{lemma}\label{lem:loo_R_norm}
Recall $R^{(m)}_1$-$R^{(m)}_3$ defined in the proof of Lemma \ref{lem:delta_main}. Suppose $\kappa = \mathcal{O}(1)$ and $np\gg \log n$. Then the following hold with probability $1 - \mathcal{O}(n^{-10})$.
\begin{align*}
&\max_{m\in[n]}|R^{(m)}_{1}|\lesssim \sqrt{\frac{\log n}{L}},\quad\max_{m\in[n]}|R^{(m)}_{2}|\lesssim \sqrt{\frac{\log n}{L}} + \frac{(np\log n)^{1/4}}{L^{3/4}} + \sqrt{\frac{(\log n)^3}{npL}},\\
&\max_{m\in[n]}|R^{(m)}_{3}|\lesssim \frac{\sqrt{\log n}}{L} + \frac{(\log n)^2}{npL}.
\end{align*}
\end{lemma}
\begin{proof}[Proof of Lemma \ref{lem:loo_R_norm}]

(\textbf{Bound of $R^{(m)}_1$}) 
Using a similar argument for the control of $R_{1,m}$ as in Lemma \ref{lem:R_norm}, we have $\max_{m\in[n]}|R^{(m)}_{1}|\lesssim \sqrt{\log n/L}$ with the prescribed probability. 

\noindent(\textbf{Bound of $R^{(m)}_2$}) Note that by definition, $\delta^{(m)}$ is independent from data involving the $m$th individual. Hence by applying Bernstein's inequality conditioning on data without the $m$th observation, we have
\begin{align*}
R^{(m)}_2 &= \sum_{i:i\neq m}A_{mi}\psi'(\theta_m^*-\theta_i^*)\delta^{(m)}_i\\
&= p\cdot\sum_{i:i\neq m}\psi'(\theta_m^*-\theta_i^*)\delta^{(m)}_i + \sum_{i:i\neq m}(A_{mi}-p)\psi'(\theta_m^*-\theta_i^*)\delta^{(m)}_i\\
&\lesssim p\sqrt{n}\pnorm{\delta^{(m)}}{} + \sqrt{p\log n}\pnorm{\delta^{(m)}}{} + \log n\cdot \pnorm{\delta^{(m)}}{\infty}\\
&\stackrel{(*)}{\lesssim} p\sqrt{n}\cdot \frac{1}{\sqrt{pL}}\cdot\Big[\sqrt{\frac{\log n}{np}} + \Big(\frac{\log n}{npL}\Big)^{1/4}\Big] + \log n\cdot\sqrt{\frac{\log n}{npL}}\\
&= \sqrt{\frac{\log n}{L}} + \frac{(np\log n)^{1/4}}{L^{3/4}} + \sqrt{\frac{(\log n)^3}{npL}},
\end{align*}
with the prescribed probability. Here $(*)$ follows from Lemmas \ref{lem:loo_delta_l2_bound} and \ref{lem:loo_delta_max_bound}. 

\noindent(\textbf{Bound for $R^{(m)}_{3}$}) By definition, we have
\begin{align*}
|R^{(m)}_3| \lesssim \Big(\sum_{i:i\neq m}A_{mi}|\hat{\theta}^{(m)}_i-\theta_i^*|\Big) \cdot \pnorm{\hat{\theta}^{(m)}-\theta^*_{-m}}{\infty}.
\end{align*}
The first term satisfies, using the independence between $\{A_{mi}\}$ and $\hat{\theta}^{(m)}$ and Bernstein's inequality, 
\begin{align*}
\sum_{i:i\neq m}A_{mi}|\hat{\theta}^{(m)}_i-\theta_i^*| &= p\cdot \sum_{i:i\neq m}|\hat{\theta}^{(m)}_i-\theta_i^*| + \sum_{i:i\neq m}(A_{mi}-p)|\hat{\theta}^{(m)}_i-\theta_i^*|\\
&\lesssim p\sqrt{n}\pnorm{\hat{\theta}^{(m)}-\theta_{-m}^*}{} + \log n\cdot \pnorm{\hat{\theta}^{(m)}-\theta_{-m}^*}{\infty}\\
&\stackrel{(*)}{\lesssim} p\sqrt{n}\sqrt{\frac{1}{pL}} + \log n\sqrt{\frac{\log n}{npL}} = \sqrt{\frac{np}{L}} + \sqrt{\frac{(\log n)^3}{npL}},
\end{align*}
where $(*)$ follows from Lemma \ref{lem:pre_estimate}. Hence
\begin{align*}
|R^{(m)}_3| \lesssim\Big(\sqrt{\frac{np}{L}} + \sqrt{\frac{(\log n)^3}{npL}}\Big)\cdot \sqrt{\frac{\log n}{npL}} = \frac{\sqrt{\log n}}{L} + \frac{(\log n)^2}{npL}.
\end{align*}
This completes the proof.
\end{proof}

\section{Proof of the main expansion for spectral method}\label{sec:proof_spec}

The goal of this section is to prove Theorem \ref{thm:spec_expansion}. We give a proof outline and introduce some preliminaries in Section \ref{subsec:proof_spec_prelim}, followed by two main steps of the proof in Sections \ref{subsec:proof_spec_delta} and \ref{subsec:proof_spec_Delta} respectively. We then complete the proof in Section \ref{subsec:proof_spec_completion}.

\subsection{Preliminary}\label{subsec:proof_spec_prelim}
First note that by a simple Taylor expansion, we can identify the following main term of $\tilde{\theta}_i - \theta_i^*$:
\begin{align}\label{def:spec_Delta}
\Delta_i \equiv \frac{\hat{\pi}_i - \pi_i^*}{\pi_i^*} - \frac{1}{n}\sum_{k=1}^n \frac{\hat{\pi}_k - \pi_k^*}{\pi_k^*}.
\end{align}
To find a close proxy of $\Delta$ that is tractable for analysis , we note the following property of $\hat{\pi}_i$:
\begin{align*}
\hat{\pi}_i = \frac{\sum_{j:j\neq i}A_{ij}\bar{y}_{ij}\hat{\pi}_j}{\sum_{j:j\neq i}A_{ij}\bar{y}_{ji}},
\end{align*}
which leads to the proxy choice (analogue of $\bar{\theta}$ for MLE)
\begin{align}\label{def:spec_pi_bar}
\bar{\pi}_i \equiv \frac{\sum_{j:j\neq i}A_{ij}\bar{y}_{ij}\pi_j^*}{\sum_{j:j\neq i}A_{ij}\bar{y}_{ji}}. 
\end{align}
The remainder vector $\delta$ for the spectral estimator is then defined by 
\begin{align}\label{eq:spec_expansion}
\Delta_i = \frac{\hat{\pi}_i - \pi^*_i}{\pi^*_i} - \frac{1}{n}\sum_{k=1}^n \frac{\hat{\pi}_k - \pi^*_k}{\pi^*_k} \equiv \frac{\bar{\pi}_i - \pi^*_i}{\pi^*_i} + \delta_i.
\end{align}
Now that the main term $(\bar{\pi}_i - \pi^*_i)/\pi^*_i$ is tractable for analysis, the goal is control $\pnorm{\delta}{\infty}$. Similar to the proof of its counterpart Proposition \ref{prop:delta_entrywise_bound}, an essential intermediate step is to give a tight bound for $\pnorm{\delta}{}$, which we now discuss in detail.


\subsection{Control of $\pnorm{\delta}{}$}\label{subsec:proof_spec_delta}

The goal of this subsection is to prove the following $\ell_2$ bound of $\delta$.
\begin{proposition}\label{prop:delta_norm_spec}
Suppose that $\kappa=\mathcal{O}(1)$ and $np\gg \log n$. Then it holds with probability $1-\mathcal{O}(n^{-10})$ that $\pnorm{\delta}{} = \mathfrak{o}(1/\sqrt{pL})$.
\end{proposition}
\begin{proof}[Proof of Proposition \ref{prop:delta_norm_spec}]
Fix $m\in[n]$. By the definitions of $\delta$ in (\ref{eq:spec_expansion}) and $\bar{\pi}_m$ in (\ref{def:spec_pi_bar}), we have
\begin{align*}
\delta_m  = \frac{\hat{\pi}_m - \bar{\pi}_m}{\pi^*_m} - \frac{1}{n}\sum_{i=1}^n \frac{\hat{\pi}_i - \pi^*_i}{\pi^*_i}= \frac{\sum_{i:i\neq m}A_{im}\bar{y}_{mi}(\hat{\pi}_i - \pi_i^*)}{\pi_m^*\cdot \sum_{i:i\neq m}A_{im}\bar{y}_{im}}-\frac{1}{n}\sum_{i=1}^n \frac{\hat{\pi}_i - \pi^*_i}{\pi^*_i}.
\end{align*}
Using again the expansion (\ref{eq:spec_expansion}) for each $\hat{\pi}_i - \pi^*_i$, we have
\begin{align*}
\delta_m &= \frac{\sum_{i:i\neq m}A_{im}\bar{y}_{mi}(\bar{\pi}_i - \pi_i^*)}{\pi_m^*\cdot \sum_{i:i\neq m}A_{im}\bar{y}_{im}} + \frac{\sum_{i:i\neq m}A_{im}\bar{y}_{mi}\pi^*_i\delta_i}{\pi_m^*\cdot \sum_{i:i\neq m}A_{im}\bar{y}_{im}}\\
&\quad\quad +\frac{1}{n}\sum_{i=1}^n \frac{\hat{\pi}_i - \pi^*_i}{\pi^*_i} \cdot  \frac{\sum_{i:i\neq m}A_{im}(\bar{y}_{mi}\pi_i^* - \bar{y}_{im}\pi^*_m)}{\pi_m^*\cdot \sum_{i:i\neq m}A_{im}\bar{y}_{im}}.
\end{align*}
In matrix form, the above display is equivalent to $L\delta = R_1 + R_2$, where $L$ is defined by $L_{ij} = -A_{ij}\bar{y}_{ij}\pi_j^*$ for $i\neq j$ and $L_{ii} = \sum_{j:j\neq i}A_{ij}\bar{y}_{ji}\pi_i^*$, and $R_1, R_2\in\R^n$ are defined by
\begin{align*}
R_{1,m} &= \sum_{i:i\neq m} A_{im}\bar{y}_{mi}(\bar{\pi}_i - \pi_i^*),\\
R_{2,m} &= \frac{1}{n}\sum_{i=1}^n \frac{\hat{\pi}_i - \pi^*_i}{\pi^*_i} \cdot  \sum_{i:i\neq m}A_{im}(\bar{y}_{mi}\pi_i^* - \bar{y}_{im}\pi^*_m).
\end{align*}
Note that $\E(L_{ij}|A) = -A_{ij}\psi(\theta_i^* - \theta_j^*)\pi_j^* = -A_{ij}\psi(\theta_j^* - \theta_i^*)\pi_i^*$ and $\E(L_{ii}|A) = \sum_{j:j\neq i}A_{ij}\psi(\theta_j^* - \theta_i^*)\pi_i^*$, and hence $\E(L|A)$ is a symmetric Laplacian matrix. Hence
\begin{align*}
\pnorm{L\delta}{} &\geq \pnorm{\E(L|A)\delta}{} - \pnorm{L - \E(L|A)}{\op}\pnorm{\delta}{}\\
&= \pnorm{\E(L|A)\big(\delta-\ave(\delta)\bm{1}_n\big)}{} - \pnorm{L - \E(L|A)}{\op}\pnorm{\delta}{}\\
&\geq \lambda_{\min, \perp}(\E(L|A))\pnorm{\delta-\ave(\delta)\bm{1}_n}{} - \pnorm{L - \E(L|A)}{\op}\pnorm{\delta}{}\\
&\geq \big[ \lambda_{\min, \perp}(\E(L|A)) - \pnorm{L - \E(L|A)}{\op}\big]_+\pnorm{\delta}{} - \sqrt{n}\cdot \lambda_{\min, \perp}(\E(L|A))\cdot |\ave(\delta)|.
\end{align*}
Here for any $M\in\R^{n\times n}$, $\lambda_{\min, \perp}(M) \equiv \min_{x\in\R^n: \bm{1}^\top_n x = 0, \pnorm{x}{} = 1} x^\top M x$. By Lemma \ref{lem:concentrate_L_spec}, $\pnorm{L - \E(L|A)}{\op}\leq \lambda_{\min, \perp}(\E(L|A))/2$ and $\lambda_{\min, \perp}(\E(L|A))\gtrsim p$ with the prescribed probability, hence rearranging the terms yields that
\begin{align*}
\pnorm{\delta}{}\lesssim \sqrt{n}|\ave(\delta)| + p^{-1}\cdot (\pnorm{R_1}{} + \pnorm{R_2}{}).
\end{align*}
The proof is now complete by plugging in the estimates of $|\ave(\delta)|$ in Lemma \ref{lem:delta_ave_spec} and of $\pnorm{R_1}{}$ and $\pnorm{R_2}{}$ in Lemma \ref{lem:bD_norm}.
\end{proof}

\begin{lemma}\label{lem:concentrate_L_spec}
Recall the matrix $L$ defined in the proof of Proposition \ref{prop:delta_norm_spec}. Suppose that $\kappa = \mathcal{O}(1)$ and $np\gg \log n$, then with probability at least $1 - \mathcal{O}(n^{-10})$, 
\begin{align*}
\lambda_{\min,\perp}(\E(L|A)) \geq cp, \quad \pnorm{L - \E(L|A)}{\op} \leq  C\Big(\sqrt{\frac{p\log n}{nL}} \vee \frac{\log n}{n}\Big)
\end{align*}
for some positive $c = c(\kappa)$ and $C = C(\kappa)$. Consequently, $\pnorm{L - \E(L|A)}{\op} \leq \lambda_{\min,\perp}(\E(L|A))/2$ for large enough $n$.
\end{lemma}
\begin{proof}[Proof of Lemma \ref{lem:concentrate_L_spec}]
Recall that $\E(L|A)$ is a symmetric Laplacian matrix with $\E(L_{ij}|A) =  -A_{ij}\psi(\theta_j^* - \theta_i^*)\pi_i^*$ for $i\neq j$ and $\E(L_{ii}|A) = \sum_{j:j\neq i}A_{ij}\psi(\theta_j^* - \theta_i^*)\pi_i^*$. Hence the first claim follows from Lemma \ref{lem:hessian_eigen} by noting that $\psi(\theta_j^* - \theta_i^*)\pi_i^*\asymp n^{-1}$. 

Next we establish the concentration of $L$. Note that $L$ can be written as $L = \sum_{i<j} S_{ij}$, where
\begin{align*}
S_{ij} \equiv A_{ij}\big(-\bar{y}_{ij}\pi_j^* \cdot e_ie_j^\top - \bar{y}_{ji}\pi_i^*\cdot e_je_i^\top + \bar{y}_{ji}\pi^*_i\cdot e_ie_i^\top + \bar{y}_{ij}\pi_j^*\cdot e_je_j^\top\big),
\end{align*}
and $S_{ij}$ are independent across $i < j$. Hence by the Bernstein inequality for asymmetric matrices \cite[Theorem 1.6]{tropp2012user}, we have
\begin{align}\label{ineq:matrix_bern_rec}
\Prob\Big(\pnorm{L - \E(L|A)}{\op} \geq t|A\Big) \leq 2n\cdot\exp\Big(-\frac{Ct^2}{\sigma^2 + Rt}\Big)
\end{align}
for some universal $C > 0$, where $R$ and $\sigma^2$ are such that
\begin{align*}
\max_{i<j}\pnorm{S_{ij}}{\op} \leq R, \quad \sigma^2\equiv \max\Big\{\bigpnorm{\sum_{i<j}\E \big(S_{ij}S_{ij}^\top|A\big)}{\op}, \bigpnorm{\sum_{i<j}\E \big(S_{ij}^\top S_{ij}|A\big)}{\op}\Big\}.
\end{align*}
Obviously $R$ can be taken to be $\mathcal{O}(n^{-1})$. For $\sigma^2$, we have by direct calculation that
\begin{align*}
\E\big(S_{ij}S_{ij}^\top|A\big) = \frac{A_{ij}}{L}\psi'(\theta_i^*-\theta_j^*)\big[(\pi_i^*)^2 + (\pi_j^*)^2\big](e_ie_i^\top + e_je_j^\top - e_ie_j^\top - e_je_i^\top). 
\end{align*}
This implies that, with $w_{ij}\equiv A_{ij}\psi'(\theta_i^*-\theta_j^*)\big[(\pi_i^*)^2 + (\pi_j^*)^2\big]/L$, $W\equiv \sum_{i<j}\E \big(S_{ij}S_{ij}^\top|A\big)$ is a Laplacian matrix with $W_{ij} = -w_{ij}$ for $i\neq j$ and $W_{ii} = \sum_{j:j\neq i}w_{ij}$. Hence using the fact that $w_{ij}/A_{ij} \asymp (n^2L)^{-1}$, Lemma \ref{lem:hessian_eigen} implies that $\pnorm{W}{\op} \lesssim (n^2L)^{-1} \cdot (np) = p/(nL)$. A similar estimate for $\pnorm{\sum_{i<j}\E\big(S_{ij}^\top S_{ij}\big|A)}{\op}$ concludes that $\sigma^2$ can be taken to be $\mathcal{O}(p/(nL))$.

Finally plugging the estimates of $R$ and $\sigma^2$ into (\ref{ineq:matrix_bern_rec}) yields that with the prescribed probability,
\begin{align*}
\pnorm{L - \E(L|A)}{\op} \lesssim \sqrt{\log n}\cdot \sigma + \log n\cdot R \lesssim \sqrt{\frac{p\log n}{nL}} \vee \frac{\log n}{n}.
\end{align*}
The fact that $\pnorm{L - \E(L|A)}{\op} \leq \pnorm{\E(L|A)}{\op}/2$ (for large enough $n$) holds under the condition $np\gg \log n$. The proof is complete.
\end{proof}

\begin{lemma}\label{lem:delta_ave_spec}
Suppose that $\kappa = \mathcal{O}(1)$ and $np\gg \log n$. Then it holds with probability at least $1-\mathcal{O}(n^{-10})$ that
\begin{align*}
|\ave(\delta)| \leq C\sqrt{\frac{1}{npL}}\cdot\Big(\sqrt{\frac{\log n}{n}} + \sqrt{\frac{\log n}{npL}}\Big)
\end{align*}
for some positive $C = C(\kappa)$. Consequently, $|\ave(\delta)| = \mathfrak{o}(1/\sqrt{npL})$.
\end{lemma}
\begin{proof}[Proof of Lemma \ref{lem:delta_ave_spec}]
By the decomposition (\ref{eq:spec_expansion}), we have 
\begin{align*}
-\ave(\delta) &= \ave\Big(\frac{\bar{\pi}-\pi^*}{\pi^*}\Big) = \frac{1}{n}\sum_{m=1}^m \frac{\bar{\pi}_m -\pi^*_m}{\pi^*_m}\\
&= \frac{1}{n}\sum_{m=1}^n \frac{\sum_{i:i\neq m}A_{im}\big(\bar{y}_{mi}\pi_i^* - \bar{y}_{im}\pi_m^*\big)}{\pi_m^*\cdot \sum_{i:i\neq m} A_{im}\bar{y}_{im}}\\
&= \frac{1}{n}\sum_{m=1}^n \frac{\sum_{i:i\neq m}A_{im}\big(\bar{y}_{mi}\pi_i^* - \bar{y}_{im}\pi_m^*\big)}{\pi_m^*\cdot \sum_{i:i\neq m} A_{im}\psi(\theta_i^*-\theta_m^*)}\\
&\quad+ \frac{1}{n}\sum_{m=1}^n\frac{\sum_{i:i\neq m}A_{im}(\bar{y}_{mi}\pi_i^* - \bar{y}_{im}\pi_m^*)}{\pi_m^*}\cdot\Big(\frac{1}{\sum_{i:i\neq m} A_{im}\bar{y}_{im}}-\frac{1}{\sum_{i:i\neq m} A_{im}\psi(\theta_i^*-\theta_m^*)}\Big)\\
&\equiv (I) + (II).
\end{align*}
To control $(I)$, let $w_{im}\equiv A_{im}/\big(\pi_m^*\cdot \sum_{k:k\neq m}A_{km}\psi(\theta_k^* - \theta_m^*)\big)$. Then
\begin{align*}
n\cdot (I) = \sum_{i\neq m} w_{im}\big(\bar{y}_{mi}\pi_i^* - \bar{y}_{im}\pi_m^*\big) = \frac{1}{n}\sum_{i<m}(w_{im} - w_{mi})(\bar{y}_{mi}\pi_i^* - \bar{y}_{im}\pi_m^*).
\end{align*}
It can be readily checked that the summands $\bar{y}_{mi}\pi_i^* - \bar{y}_{im}\pi_m^*$ are independent across the indices $i<m$, centered, and sub-Gaussian with variance proxy bounded by $L^{-1}\pnorm{\pi^*}{\infty}^2\lesssim (n^2L)^{-1}$. Hence conditioning on the graph $A$, Hoeffding's inequality yields that
\begin{align*}
\Prob(n\cdot |(I)| \geq t|A)\leq 2\exp\Big(-\frac{Ct^2}{(n^2L)^{-1}\cdot \sum_{i<m} (w_{im} - w_{mi})^2}\Big).
\end{align*}
By Lemma \ref{lem:A_concen} and the lower bound $\min_{m\in[n]}\pi_m^*\gtrsim n^{-1}$, we have $\sum_{i<m} (w_{im} - w_{mi})^2\lesssim \sum_{i<m} (w_{im}^2 + w_{mi}^2) \lesssim n^2p/p^2 = n^2/p$. Hence by choosing $t\asymp \sqrt{\log n/(pL)}$, the above display yields that with the prescribed probability,
\begin{align*}
|(I)| \lesssim \frac{1}{n}\sqrt{\frac{\log n}{pL}} = \sqrt{\frac{1}{npL}}\cdot\sqrt{\frac{\log n}{n}}.
\end{align*}

To control $(II)$, using the lower bound $\min_{m\in[n]}\pi^*_m\gtrsim n^{-1}$ and standard concentration on the graph $A$, we have
\begin{align*}
|(II)| &\leq (np)^{-2}\times\max_{m\in[n]}\Big|\sum_{i:i\neq m}A_{im}\big(\bar{y}_{mi}\pi_i^* - \bar{y}_{im}\pi_m^*\big)\Big|\\
&\quad\quad\times \sum_{m=1}^n \Big|\sum_{i:i\neq m}A_{im}(\bar{y}_{im}-\psi(\theta_i^*-\theta_m^*))\Big|\\
&\leq (np)^{-2}\sqrt{n}\times\max_{m\in[n]}\Big|\sum_{i:i\neq m}A_{im}\big(\bar{y}_{mi}\pi_i^* - \bar{y}_{im}\pi_m^*\big)\Big|\\
&\quad\quad\times \Big[\sum_{m=1}^n \Big|\sum_{i:i\neq m}A_{im}(\bar{y}_{im}-\psi(\theta_i^*-\theta_m^*))\Big|^2\Big]^{1/2}\\
&\stackrel{(*)}{\lesssim} (np)^{-2}\sqrt{n}\cdot \sqrt{\frac{p\log n}{nL}}\cdot \sqrt{\frac{n^2p}{L}} = \sqrt{\frac{1}{npL}}\cdot \sqrt{\frac{\log n}{npL}},
\end{align*}
where $(*)$ follows from Lemma \ref{lem:mle_concentration}. The proof is complete. 
\end{proof}

\begin{lemma}\label{lem:bD_norm}
Recall the definitions of $R_1$ and $R_2$ in the proof of Proposition \ref{prop:delta_norm_spec}. Suppose that $\kappa = \mathcal{O}(1)$ and $np\gg \log n$. Then the following holds with probability at least $1-\mathcal{O}(n^{-10})$ for some $C = C(\kappa) > 0$.
\begin{align*}
\pnorm{R_1}{\infty}\vee \pnorm{R_2}{\infty} \leq C\frac{\sqrt{\log n}}{n\sqrt{L}}, \quad \pnorm{R_1}{} \vee \pnorm{R_2}{} = \mathfrak{o}(\sqrt{\frac{p}{L}}).
\end{align*}
\end{lemma}
\begin{proof}[Proof of Lemma \ref{lem:bD_norm}]
It suffices to bound $\pnorm{R_1}{\infty}$ and $\pnorm{R_2}{\infty}$. We first bound $\pnorm{R_2}{\infty}$. By definition, we have
\begin{align*}
\pnorm{R_2}{\infty}^2&\lesssim \Big(\sqrt{n}\pnorm{\hat{\pi} - \pi^*}{}\Big)^2\cdot \max_{m\in[n]}\Big(\sum_{i:i\neq m}A_{im}(\bar{y}_{mi}\pi_i^* - \bar{y}_{im}\pi_m^*)\Big)^2\\
&\stackrel{(*)}{\lesssim} \frac{1}{npL} \cdot \frac{p\log n}{nL} = \frac{\log n}{n^2L}\cdot \frac{1}{L}.
\end{align*}
Here $(*)$ follows from Lemmas \ref{lem:spec_rate} and \ref{lem:mle_concentration}. 

Next we bound $\pnorm{R_1}{\infty}$. By definition of $\bar{\pi}$ in (\ref{def:spec_pi_bar}), we have
\begin{align*}
R_{1,m} &= \sum_{i:i\neq m}A_{im}\bar{y}_{mi}\frac{\sum_{j:j\neq i}A_{ij}(\bar{y}_{ij}\pi_j^* - \bar{y}_{ji}\pi_i^*)}{\sum_{j:j\neq i}A_{ij}\bar{y}_{ji}}\\
&= \sum_{i:i\neq m}A_{im}\psi(\theta_m^*-\theta_i^*)\frac{\sum_{j:j\neq i}A_{ij}(\bar{y}_{ij}\pi_j^* - \bar{y}_{ji}\pi_i^*)}{\sum_{j:j\neq i}A_{ij}\psi(\theta_j^*-\theta_i^*)}\\
&\quad\quad + \sum_{i:i\neq m}A_{im}\big(\bar{y}_{mi}-\psi(\theta_m^*-\theta_i^*)\big)\frac{\sum_{j:j\neq i}A_{ij}(\bar{y}_{ij}\pi_j^* - \bar{y}_{ji}\pi_i^*)}{\sum_{j:j\neq i}A_{ij}\psi(\theta_j^*-\theta_i^*)}\\
&\quad\quad + \sum_{i:i\neq m}A_{im}\bar{y}_{mi}\Big[\sum_{j:j\neq i}A_{ij}(\bar{y}_{ij}\pi_j^* - \bar{y}_{ji}\pi_i^*)\Big(\frac{1}{\sum_{j:j\neq i}A_{ij}\bar{y}_{ji}}-\frac{1}{\sum_{j:j\neq i}A_{ij}\psi(\theta_j^*-\theta_i^*)}\Big)\Big]\\
&\equiv R^1_{1,m} + R^2_{1,m} + R^3_{1,m}.
\end{align*}
(\textbf{Bound of $\pnorm{R^1_1}{\infty}$}) Fix $m\in[n]$. Define $w_{ij}\equiv A_{im}A_{ij}\psi(\theta_m^*-\theta_i^*)/\big(\sum_{j:j\neq i}A_{ij}\psi(\theta_j^*-\theta_i^*)\big)$. Then $R^1_{1,m}$ can be written as
\begin{align*}
R^1_{1,m} &= \sum_{i,j:i\neq m,i\neq j}w_{ij}\big(\bar{y}_{ij}\pi_j^*-\bar{y}_{ji}\pi_i^*\big)\\
&= \sum_{i,j:i,j\neq m,i\neq j}w_{ij}\big(\bar{y}_{ij}\pi_j^*-\bar{y}_{ji}\pi_i^*\big) + \sum_{i:i\neq m}w_{im}\big(\bar{y}_{im}\pi_m^*-\bar{y}_{mi}\pi_i^*\big)\\
&= \sum_{i,j:i,j\neq m,i< j}(w_{ij}-w_{ji})\big(\bar{y}_{ij}\pi_j^*-\bar{y}_{ji}\pi_i^*\big) + \sum_{i:i\neq m}w_{im}\big(\bar{y}_{im}\pi_m^*-\bar{y}_{mi}\pi_i^*\big).
\end{align*}
Now conditioning on the graph $A$, $\{w_{ij}\}$ are deterministic and the above two sums are independent across their summands, which are sub-Gaussian with variance proxy bounded by $L^{-1}\pnorm{\pi^*}{\infty}^2\lesssim (n^2L)^{-1}$. Hence Hoeffding's inequality yields that with the prescribed probability, 
\begin{align*}
\Big|\sum_{i,j:i,j\neq m,i< j}(w_{ij}-w_{ji})\big(\bar{y}_{ij}\pi_j^*-\bar{y}_{ji}\pi_i^*\big)\Big|&\lesssim\sqrt{\frac{\log n}{n^2L}\cdot \sum_{i,j:i,j\neq m,i<j}(w_{ij}-w_{ji})^2},\\
\Big|\sum_{i:i\neq m}w_{im}\big(\bar{y}_{im}\pi_m^*-\bar{y}_{mi}\pi_i^*\big)\Big| &\lesssim\sqrt{\frac{\log n}{n^2L}\cdot \sum_{i:i\neq m}w_{im}^2}.
\end{align*}
Now by Lemma \ref{lem:A_concen}, we have
\begin{align*}
\sum_{i,j:i,j\neq m,i<j}(w_{ij}-w_{ji})^2 &\lesssim \sum_{i,j:i,j\neq m,i<j}\Big(\frac{A_{im}A_{ij}\psi(\theta_m^*-\theta_i^*)}{\sum_{j\neq i}A_{ij}\psi(\theta_j^*-\theta_i^*)}\Big)^2\\
&\lesssim (np)^{-2}\sum_{i,j:i,j\neq m,i<j}A_{im}A_{ij}  \lesssim 1,\\
\sum_{i:i\neq m}w_{im}^2 &\lesssim\sum_{i:i\neq m}\Big(\frac{A_{im}\psi(\theta_m^*-\theta_i^*)}{\sum_{j\neq i}A_{ij}\psi(\theta_j^*-\theta_i^*)}\Big)^2\lesssim (np)^{-1}.
\end{align*}
Combining the two estimates yields that $\max_{m\in[n]}|R^1_{1,m}|\lesssim \sqrt{\log n/(n^2L)}$. 

\noindent (\textbf{Bound of $\pnorm{R^2_1}{\infty}$}) Fix $m\in[n]$. Note that $R^2_{1,m}$ can be further decomposed as
\begin{align*}
R^2_{1,m} &= \sum_{i:i\neq m}A_{im}(\bar{y}_{mi} - \psi(\theta_m^*-\theta_i^*))\frac{\sum_{j:j\neq i,m}A_{ij}(\bar{y}_{ij}\pi_j^*-\bar{y}_{ji}\pi_i^*)}{\sum_{j:j\neq i}A_{ij}\psi(\theta_j^*-\theta_i^*)} \\
&\quad\quad\quad +\sum_{i:i\neq m}A_{im}(\bar{y}_{mi} - \psi(\theta_m^*-\theta_i^*))\frac{A_{im}(\bar{y}_{im}\pi_m^*-\bar{y}_{mi}\pi_i^*)}{\sum_{j:j\neq i}A_{ij}\psi(\theta_j^*-\theta_i^*)}\\
&\equiv R^{2,1}_{1,m} + R^{2,2}_{1,m}.
\end{align*}
To deal with $R^{2,1}_{1,m}$, note that 
\begin{align*}
R^{2,1}_{1,m} = \sum_{i,j:i,j\neq m,i\neq j}\sum_{\ell=1}^Lw_{ij}z_{ij\ell} =  \sum_{i,j:i,j\neq m,i< j}\sum_{\ell=1}^L(w_{ij} - w_{ji})z_{ij\ell}
\end{align*}
where
\begin{align*}
w_{ij} = \frac{A_{im}A_{ij}\big(\bar{y}_{mi} - \psi(\theta_m^*-\theta_i^*)\big)}{\sum_{k:k\neq i}A_{ik}\psi(\theta_k^*-\theta_i^*)},\quad z_{ij\ell} = \frac{1}{L}\big(y_{ij\ell}\pi_j^*-y_{ji\ell}\pi_i^*\big).
\end{align*}
Now conditioning on $A$ and all comparisons involving the $m$th individual, $\{w_{ij}\}$ is deterministic, and $z_{ij\ell}$ are independent across $i<j$ and sub-Gaussian with variance proxy of the order $(nL)^{-2}$, hence Hoeffding's inequality applies to conclude that with the prescribed probability, 
\begin{align*}
|R^{2,1}_{1,m}|\lesssim \sqrt{\frac{\log n}{n^2L}\cdot \sum_{i,j:i,j\neq m,i\neq j}w_{ij}^2} + \frac{\log n}{nL}\cdot\max_{i,j:i,j\neq m,i\neq j}|w_{ij}|.
\end{align*}
By definition, we have $\max_{i,j:i,j\neq m,i\neq j}|w_{ij}| \lesssim (np)^{-1}$ by Lemma \ref{lem:A_concen}, and by Lemma \ref{lem:mle_concentration} we have
\begin{align*}
\sum_{i,j:i,j\neq m,i\neq j}w_{ij}^2 \lesssim (np)^{-2} \cdot np \cdot \sum_{i:i\neq m}A_{im}\big(\bar{y}_{mi} - \psi(\theta_m^*-\theta_i^*)\big)^2\lesssim \frac{1}{L}.
\end{align*}
Plugging in the above estimates yields that $\max_{m\in[n]}|R^{2,1}_{1,m}|\lesssim \sqrt{\log n}/(nL)$. 

To deal with $R^{2,2}_{1}$, note that $R^{2,2}_{1,m} = \sum_{i:i\neq m}w_iz_i$, where
\begin{align*}
w_i = \frac{A_{im}}{\sum_{j:j\neq i}A_{ij}\psi(\theta_j^*-\theta_i^*)},\quad z_i = \big(\bar{y}_{mi} -\psi(\theta_m^*-\theta_i^*)\big)\big(\bar{y}_{im}\pi_m^* - \bar{y}_{mi}\pi_i^*\big).
\end{align*}
Further decompose 
\begin{align*}
R^{2,2}_{1,m} = \sum_{i:i\neq m}w_i\E z_i + \sum_{i:i\neq m} w_i(z_i - \E z_i) \equiv (I) + (II).
\end{align*}
Since $\E z_i = -(\pi_i^*+\pi_m^*)\psi'(\theta_m^*-\theta_i^*)/L$ so that $|\E z_i|\lesssim (nL)^{-1}$, Lemma \ref{lem:A_concen} yields that with the prescribed probability, 
\begin{align*}
|(I)|\lesssim \sum_{i:i\neq m} w_i|\E z_i| \lesssim (nL)^{-1}\cdot \sum_{i:i\neq m} \frac{A_{im}}{\sum_{j:j\neq i}A_{ij}\psi(\theta_j^*-\theta_i^*)} \lesssim (nL)^{-1}.
\end{align*}
For $(II)$, note that $\{z_i\}_{i:i\neq m}$ are independent, and each $z_i$ is the product of two sub-Gaussian terms with variance proxies $L^{-1}$ and $(n^2L)^{-1}$ respectively so that $z_i$ is sub-exponential with norm $K \lesssim (nL)^{-1}$. Therefore Bernstein's inequality yields that
\begin{align*}
\Prob\Big(|(II)| \geq t\Big| A\Big) \leq \exp\Big(-\frac{Ct^2}{\pnorm{w}{}^2K^2 + \pnorm{w}{\infty}K}\Big).
\end{align*} 
By Lemma \ref{lem:A_concen}, we have $\pnorm{w}{\infty}\lesssim (np)^{-1}$ and $\pnorm{w}{}^2 \lesssim (np)^{-1}$ with the prescribed probability. Hence by choosing $t\asymp (nL)^{-1}\cdot\sqrt{\log n/(np)}$, the above estimate yields that $|(II)|\lesssim (nL)^{-1}\cdot\sqrt{\log n/(np)}$ with the prescribed probability. This concludes that $\pnorm{R^{2,2}_1}{\infty}\lesssim 1/(nL)$. Combining the estimates for $R^{2,1}_1$ and $R^{2,2}_1$ yields that $\pnorm{R^2_1}{\infty}\lesssim \sqrt{\log n}/(nL)$ with the prescribed probability.

\noindent(\textbf{Bound of $\pnorm{R^3_1}{\infty}$}) By definition, we have
\begin{align*}
|R^3_{1,m}| &\lesssim (np)^{-2}\cdot\sum_{i:i\neq m}A_{im}\Big|\sum_{j:j\neq i}A_{ij}(\bar{y}_{ij}\pi_j^* - \bar{y}_{ji}\pi_i^*)\Big|\Big|\sum_{j:j\neq i}A_{ij}\big(\bar{y}_{ij}-\psi(\theta_i^*-\theta_j^*)\big)\Big|\\
&\lesssim (np)^{-2}\Big(\sum_{i:i\neq m}A_{im}U_i^2\Big)^{1/2}\cdot  \Big(\sum_{i:i\neq m}A_{im}V_i^2\Big)^{1/2},
\end{align*}
where 
\begin{align*}
U_i\equiv \sum_{j:j\neq i}A_{ij}(\bar{y}_{ij}\pi_j^* - \bar{y}_{ji}\pi_i^*), \quad V_i\equiv \sum_{j:j\neq i}A_{ij}\big(\bar{y}_{ij}-\psi(\theta_i^*-\theta_j^*).
\end{align*}
Now we bound $\sum_{i:i\neq m}A_{im}U_i^2$ for each $m\in[n]$. Since $\bar{y}_{ij}\pi_j^* - \bar{y}_{ji}\pi_i^* = (\pi^*_i + \pi^*_j)\big(\bar{y}_{ij} -\psi(\theta_i^*-\theta_j^*)\big)$, we have
\begin{align*}
\Big(\sum_{i:i\neq m}A_{im}U_i^2\Big)^{1/2} &= \Big[\sum_{i:i\neq m}A_{im}\Big(\sum_{j:j\neq i}A_{ij}(\pi^*_i + \pi^*_j)\big(\bar{y}_{ij} -\psi(\theta_i^*-\theta_j^*)\big)\Big)^2\Big]^{1/2}\\
&= \sup_{u\in\mathcal{U}} \sum_{i:i\neq m}A_{im}u_i\cdot \sum_{j:j\neq i}A_{ij}(\pi^*_i + \pi^*_j)\big(\bar{y}_{ij} -\psi(\theta_i^*-\theta_j^*)\big),
\end{align*}
where $\mathcal{U}\equiv \{\sum_{u\in\R^n}: \sum_{i:i\neq m} A_{im}u_i^2 \leq 1\}$ is defined conditioning on $A$. Let $\bar{\mathcal{U}}$ be a $1/2$-covering of $\mathcal{U}$ in the sense that for any $u\in\mathcal{U}$, there exists some $u'\in\bar{\mathcal{U}}$ such that $\sqrt{\sum_{i:i\neq m}A_{im}(u_i - u'_i)^2}\leq 1/2$. Since covering $\mathcal{U}$ is equivalent to covering the unit ball in dimension $\sum_{i:i\neq m}A_{im}$, we can find such a $\bar{\mathcal{U}}$ with $|\bar{\mathcal{U}}|\leq \exp(C\sum_{i:i\neq m}A_{im}) \leq \exp(C'np)$ with the prescribed probability. Then a standard covering argument yields that
\begin{align*}
\sqrt{\sum_{i:i\neq m}A_{im}U_i^2} &\leq 2\sup_{u\in \bar{\mathcal{U}}}  \sum_{i:i\neq m}A_{im}u_i\cdot \sum_{j:j\neq i}A_{ij}(\pi^*_i + \pi^*_j)\big(\bar{y}_{ij} -\psi(\theta_i^*-\theta_j^*)\big)\\
&= 2\sup_{u\in \bar{\mathcal{U}}} \Big\{\frac{1}{L}\sum_{\ell=1}^L\sum_{i,j:i\neq m,j\neq m,i\neq j}A_{im}A_{ij}u_i(\pi_i^* + \pi_j^*)(y_{ij\ell} - \psi(\theta_i^* - \theta_j^*))\\
&\quad\quad + \frac{1}{L}\sum_{\ell=1}^L\sum_{i:i\neq m}A_{im}u_i(\pi_i^* +\pi_m^*)\big(y_{im\ell} - \psi(\theta_i^* - \theta_m^*)\big)\Big\}.
\end{align*}
Now by Hoeffding's inequality as in the analysis of $R^1_1$, we have for any fixed $u\in\bar{\mathcal{U}}$ and $t > 0$
\begin{align*}
&\Prob\Big( \Big|\frac{1}{L}\sum_{\ell=1}^L\sum_{i,j:i\neq m,j\neq m,i\neq j}A_{im}A_{ij}u_i(\pi_i^* + \pi_j^*)(y_{ij\ell} - \psi(\theta_i^* - \theta_j^*))\Big| \geq t|A\Big)\\
&\leq \exp\Big(-\frac{CLt^2}{\sum_{i,j:i\neq m,j\neq m,i\neq j}A_{im}A_{ij}u_i^2(\pi_i^* + \pi_j^*)^2}\Big). 
\end{align*}
Hence by a union bound, Lemma \ref{lem:A_concen}, and choosing $t\asymp p/\sqrt{L}$, we have with the prescribed probability
\begin{align*}
\sup_{u\in \bar{\mathcal{U}}} \Big|\frac{1}{L}\sum_{\ell=1}^L\sum_{i,j:i\neq m,j\neq m,i\neq j}A_{im}A_{ij}u_i(\pi_i^* + \pi_j^*)(y_{ij\ell} - \psi(\theta_i^* - \theta_j^*))\Big| \lesssim \frac{p}{\sqrt{L}}.
\end{align*}
Using a similar analysis, we have
\begin{align*}
\sup_{u\in \bar{\mathcal{U}}}\Big|\frac{1}{L}\sum_{i:i\neq m}A_{im}u_i(\pi_i^* +\pi_m^*)\big(y_{im\ell} - \psi(\theta_i^* - \theta_m^*)\big)\Big| \lesssim \sqrt{\frac{p}{nL}} = \mathfrak{o}(\frac{p}{\sqrt{L}}).
\end{align*}
Putting together the two estimates yields that
\begin{align*}
\max_{m\in[n]}\Big(\sum_{i:i\neq m}A_{im}U_i^2\Big)^{1/2} \lesssim \frac{p}{\sqrt{L}}.
\end{align*}
A similar calculation yields that $\max_{m\in[n]}\Big(\sum_{i:i\neq m}A_{im}V_i^2\Big)^{1/2} \lesssim np/\sqrt{L}$, so combining the estimates yields that $\max_{m\in[n]}|R^3_{1,m}| \lesssim (nL)^{-1}$ with the prescribed probability.
Combining the estimates for $\pnorm{R^1_1}{\infty}$ - $\pnorm{R^3_1}{\infty}$ concludes the proof.
\end{proof}

\subsection{Entrywise expansion for the main term $\Delta$}\label{subsec:proof_spec_Delta}

Recall from Section \ref{subsec:proof_spec_prelim} that $\Delta_i$ defined in (\ref{def:spec_Delta}) is the main term of $\tilde{\theta}_i - \theta_i^*$. The goal of this subsection is to prove the following expansion. 

\begin{proposition}\label{prop:spec_Delta_expansion}
Suppose that $\kappa = \mathcal{O}(1)$ and $np\gg (\log n)^{3/2}$. Then the following expansion holds with probability $1-\mathcal{O}(n^{-10})$.
\begin{align*}
\Delta_i = \big(1+\epsilon_{1,i}\big)\frac{\sum_{j:j\neq i}A_{ij}\big(\bar{y}_{ij}\pi_j^* - \bar{y}_{ji}\pi_i^*\big)}{\pi_i^*\cdot \sum_{j:j\neq i}A_{ij}\psi(\theta_j^* - \theta_i^*)} +\epsilon_{2,i},
\end{align*}
where $\epsilon_1,\epsilon_2\in\R^n$ satisfy $\pnorm{\epsilon_1}{\infty} = \mathfrak{o}(1)$ and $\pnorm{\epsilon_2}{\infty} = \mathfrak{o}(1/\sqrt{npL})$.
\end{proposition}

As in Section \ref{subsec:proof_delta_infty} for the MLE, we will first perform some preliminary analysis in Section \ref{subsubsec:proof_spec_infty_loo}, and the main proof of Proposition \ref{prop:spec_Delta_expansion} will be given in Section \ref{subsubsec:proof_spec_infty_main}. 

\subsubsection{Some preliminary analysis}\label{subsubsec:proof_spec_infty_loo}

We first introduce a leave-one-out version of the spectral estimate $\hat{\pi}$. Fix an index $m\in[n]$ to be left out, and define a new transition probability matrix $P^{(m)}\in\R^{n\times n}$ as follows. For off-diagonal elements, let
\begin{align*}
P^{(m)}_{ij}\equiv
\begin{cases}
P_{ij}, & i\neq m\text{ and } j\neq m,\\
\E P_{ij} = \frac{p}{d}\psi(\theta_j^* - \theta_i^*), & i = m \text{ or } j = m. 
\end{cases}
\end{align*}
This leads to the choice of diagonal elements $P^{(m)}_{ii} \equiv 1 - \sum_{j:j\neq i}P^{(m)}_{ij}$. Note that in the case of $i=m$ or $j = m$, we have taken unconditional expectation of $P_{ij}$ so that $P^{(m)}$ is independent of the data involving the $m$th individual, including the comparison indicators $\{A_{mi}\}_{i\neq m}$.  With the above definition, let $\hat{\pi}^{(m)}$ be the stationary measure of $P^{(m)}$, i.e. $\hat{\pi}^{(m)}$ is defined by
\begin{align}\label{def:spec_loo}
(\hat{\pi}^{(m)})^\top P^{(m)} = (\hat{\pi}^{(m)})^\top,
\end{align}
which, after some similar manipulations for $\hat{\pi}$, leads to
\begin{align*}
\hat{\pi}^{(m)}_i = \frac{\sum_{j:j\neq i}\hat{\pi}^{(m)}_jP^{(m)}_{ji}}{\sum_{j:j\neq i}P^{(m)}_{ij}} = \frac{\sum_{j:j\neq i,m}A_{ij}\bar{y}_{ij}\hat{\pi}^{(m)}_j + \mathbb{I}_{i\neq m}\cdot p\psi(\theta_i^*-\theta_m^*)\hat{\pi}^{(m)}_m}{\sum_{j:j\neq i,m}A_{ij}\bar{y}_{ji} + \mathbb{I}_{i\neq m}\cdot p\psi(\theta_m^*-\theta_i^*)}.
\end{align*}
Note that different from the analogue $\hat{\theta}^{(m)}\in\R^{n-1}$ in (\ref{def:loo_mle}) for the MLE, $\hat{\pi}^{(m)}$ is still in $\R^n$. Analogous to $\bar{\pi}$ in (\ref{def:spec_pi_bar}), define $\bar{\pi}^{(m)}\in\R^n$ as
\begin{align*}
\bar{\pi}^{(m)}_i = \frac{\sum_{j:j\neq i,m}A_{ij}\bar{y}_{ij}\pi^*_j + \mathbb{I}_{i\neq m}\cdot p\psi(\theta_i^*-\theta_m^*)\pi^*_m}{\sum_{j:j\neq i,m}A_{ij}\bar{y}_{ji} + \mathbb{I}_{i\neq m}\cdot p\psi(\theta_m^*-\theta_i^*)}.
\end{align*}
Lastly, define the leave-one-out analogue of $\delta$ in (\ref{eq:spec_expansion}) as
\begin{align}\label{def:spec_delta_loo}
\frac{\hat{\pi}^{(m)}_i - \pi^*_i}{\pi^*_i} - \frac{1}{n}\sum_{k=1}^n\frac{\hat{\pi}^{(m)}_k - \pi^*_k}{\pi^*_k} \equiv \frac{\bar{\pi}^{(m)}_i - \pi^*_i}{\pi^*_i} + \delta^{(m)}_i.
\end{align}
Note that by construction, all of $\hat{\pi}^{(m)}$, $\bar{\pi}^{(m)}$, and $\delta^{(m)}$ do not depend on the $m$th individual. 
We have the following estimates for $\delta^{(m)}$.
\begin{lemma}\label{lem:loo_delta_bound}
Suppose that $\kappa = \mathcal{O}(1)$ and $np\gg\log n$. Then the following holds with probability at least $1-\mathcal{O}(n^{-10})$ for some $C =C(\kappa) > 0$.
\begin{align*}
\max_{m\in[n]}\pnorm{\delta^{(m)}}{} = \mathfrak{o}(1/\sqrt{pL}) \quad \text{ and }\quad\max_{m\in[n]}\pnorm{\delta^{(m)}}{\infty} \leq C\sqrt{\log n/(npL)}.
\end{align*}
\end{lemma}
\begin{proof}[Proof of Lemma \ref{lem:loo_delta_bound}]
The $\ell_2$ bound follows from analogous arguments as in Proposition \ref{prop:delta_norm_spec}, except now using the estimate $\max_{m\in[n]}\pnorm{\hat{\pi}^{(m)} - \pi^*}{}\lesssim 1/\sqrt{n^2pL}$ (instead of $\pnorm{\hat{\pi} - \pi^*}{}\lesssim 1/\sqrt{n^2pL}$) in Lemma \ref{lem:bD_norm}.
The $\ell_\infty$ bound follows from the following simple estimate: for any $i\in[n]$,
\begin{align*}
|\delta^{(m)}_i| \lesssim \frac{\pnorm{\hat{\pi}^{(m)} - \pi^*}{\infty}}{\pi_i^*} + \frac{1}{n}\sum_{k=1}^n \Big|\frac{\hat{\pi}^{(m)}_k - \pi^*_k}{\pi_k^*}\Big| \lesssim \sqrt{\frac{\log n}{npL}},
\end{align*}
using Lemma \ref{lem:spec_rate} in the last inequality. 
\end{proof}
\begin{remark}
In contrast to its counterpart Lemma \ref{lem:loo_delta_l2_bound} for the MLE, we do not need to analyze a leave-two-out version of $\hat{\pi}$ here for the spectral estimator. The reason is that in the proof of Lemma \ref{lem:loo_delta_l2_bound}, we need a leave-two-out analysis for the control of $g^{(-m,i)}(\theta_i^*|\hat{\theta}^{(m)}_{-i})$ therein, which does not appear in the analysis of the spectral estimator. On the other hand, if we were to perform a leave-$k$-out analysis to improve the exponent in the regime (\ref{def:regime}), a corresponding leave-$k$-out version of $\hat{\pi}$ would become necessary for the spectral method as well.
\end{remark}

\subsubsection{Main proof of Proposition \ref{prop:spec_Delta_expansion}}\label{subsubsec:proof_spec_infty_main}
\begin{proof}[Proof of Proposition \ref{prop:spec_Delta_expansion}]
Fix $m\in[n]$. Recall the definition of $\delta$ in (\ref{eq:spec_expansion}). Then we have
\begin{align*}
\Delta_m &= \frac{\sum_{i:i\neq m}A_{mi}\big(\bar{y}_{mi}\pi_i^* - \bar{y}_{im}\pi_m^*\big)}{\pi_m^* \cdot \sum_{i:i\neq m}A_{mi}\bar{y}_{im}} + \delta_m\\
&= \frac{\sum_{i:i\neq m}A_{mi}\big(\bar{y}_{mi}\pi_i^* - \bar{y}_{im}\pi_m^*\big)}{\pi_m^* \cdot \sum_{i:i\neq m}A_{mi}\psi(\theta_i^*-\theta_m^*)}\Big(1-\frac{\sum_{i:i\neq m}A_{mi}\big(\bar{y}_{im}-\psi(\theta_i^*-\theta_m^*)\big)}{\sum_{i:i\neq m}A_{mi}\bar{y}_{im}}\Big) + \delta_m. 
\end{align*}
Using standard concentration in Lemma \ref{lem:mle_concentration}, we have
\begin{align*}
\max_{m\in[n]}\Big|\frac{\sum_{i:i\neq m}A_{mi}\big(\bar{y}_{im}-\psi(\theta_i^*-\theta_m^*)\big)}{\pi_m^* \cdot \sum_{i:i\neq m}A_{mi}\bar{y}_{im}}\Big| \lesssim \frac{\sqrt{np\log n/L}}{np} = \sqrt{\frac{\log n}{npL}} = \mathfrak{o} (1).
\end{align*}
Hence it remains to show that $\pnorm{\delta}{\infty} = \mathfrak{o}(1/\sqrt{npL})$ with the prescribed probability. To this end, recall the definition of $\hat{\pi}^{(m)}$ in (\ref{def:spec_loo}). Then we have
\begin{align*}
\delta_m &= \frac{\sum_{i:i\neq m}A_{im}\bar{y}_{mi}(\hat{\pi}_i - \pi^*_i)}{\pi_m^* \cdot \sum_{i:i\neq m}A_{im}\bar{y}_{im}} - \frac{1}{n}\sum_{k=1}^n\frac{\hat{\pi}_k -\pi_k^*}{\pi_k^*}\\
&= \Big(\frac{\sum_{i:i\neq m}A_{im}\bar{y}_{mi}(\hat{\pi}_i - \hat{\pi}^{(m)}_i)}{\pi_m^* \cdot \sum_{i:i\neq m}A_{im}\bar{y}_{im}} - \frac{1}{n}\sum_{k=1}^n\frac{\hat{\pi}_k -\hat{\pi}^{(m)}_k}{\pi_k^*}\Big)\\
&\quad\quad\quad + \Big(\frac{\sum_{i:i\neq m}A_{im}\bar{y}_{mi}(\hat{\pi}^{(m)}_i - \pi^*_i)}{\pi_m^* \cdot \sum_{i:i\neq m}A_{im}\bar{y}_{im}} - \frac{1}{n}\sum_{k=1}^n\frac{\hat{\pi}^{(m)}_k -\pi_k^*}{\pi_k^*}\Big)\\
&\equiv \delta_{1,m} + \delta_{2,m}.
\end{align*}
Using Lemma \ref{lem:spec_rate}, the two terms inside $\delta_{1,m}$ can be bounded as follows:
\begin{itemize}
\item The first term satisfies
\begin{align*}
&\Big|\frac{\sum_{i:i\neq m}A_{im}\bar{y}_{mi}(\hat{\pi}_i - \hat{\pi}^{(m)}_i)}{\pi_m^* \cdot \sum_{i:i\neq m}A_{im}\bar{y}_{im}}\Big| \lesssim \frac{1}{p}\sum_{i:i\neq m}A_{im}|\hat{\pi}_i - \hat{\pi}^{(m)}_i|\\
&\leq \frac{1}{p}\Big(\sum_{i:i\neq m}A_{im}\Big)^{1/2}\cdot \pnorm{\hat{\pi}^{(m)} - \hat{\pi}}{} \lesssim \frac{1}{p}\cdot \sqrt{np}\cdot n^{-1}\sqrt{\frac{\log n}{npL}} = \mathfrak{o}(\frac{1}{\sqrt{npL}}).
\end{align*}
\item Using $\pi^*_m\gtrsim n^{-1}$, the second term satisfies
\begin{align*}
\Big|\frac{1}{n}\sum_{k=1}^n\frac{\hat{\pi}_k -\hat{\pi}^{(m)}_k}{\pi_k^*}\Big| \lesssim \sqrt{n}\pnorm{\hat{\pi}^{(m)} - \hat{\pi}}{} \lesssim \sqrt{\frac{\log n}{n}}\sqrt{\frac{1}{npL}}= \mathfrak{o}(\frac{1}{\sqrt{npL}}).
\end{align*}
\end{itemize}
We hence conclude $\pnorm{\delta_1}{\infty} = \mathfrak{o}(1/\sqrt{npL})$. For $\delta_{2,m}$, using the definition of $\delta^{(m)}$ in (\ref{def:spec_delta_loo}), we have
\begin{align*}
\Big(\pi_m^*\cdot \sum_{i:i\neq m}A_{im}\bar{y}_{im}\Big)\delta_{2,m} &= \sum_{i:i\neq m}A_{im}\bar{y}_{mi}(\bar{\pi}^{(m)}_i - \pi_i^*) + \sum_{i:i\neq m}A_{im}\bar{y}_{mi}\pi_i^*\delta^{(m)}_i + \\
&\quad\quad\quad \sum_{i:i\neq m}A_{im}(\bar{y}_{mi}\pi_i^* - \bar{y}_{im}\pi_m^*)\cdot \frac{1}{n}\sum_{k=1}^n \frac{\hat{\pi}^{(m)}_k - \pi_k^*}{\pi_k^*}\\
&\equiv R^{(m)}_{1,m} + \sum_{i:i\neq m}A_{im}\bar{y}_{mi}\pi_i^*\delta^{(m)}_i + R^{(m)}_{2,m}.
\end{align*}
By analogous arguments as in the proof of Lemma \ref{lem:bD_norm}, we have 
\begin{align*}
\max_{m\in[n]}\frac{|R^{(m)}_{1,m}| + |R^{(m)}_{2,m}|}{\pi_m^*\cdot \sum_{i:i\neq m}A_{im}\bar{y}_{im}} \lesssim \sqrt{\frac{\log n}{np}}\cdot \frac{1}{\sqrt{npL}} = \mathfrak{o}(\frac{1}{\sqrt{npL}}),
\end{align*}
under the condition $np\gg \log n$. Lastly, for the middle term, by Lemma \ref{lem:loo_delta_bound} and Bernstein's inequality applied conditionally to $\{A_{im}\}_{i\neq m}$ (note its independence from $\delta^{(m)}$ by construction), 
\begin{align*}
&\Big|\sum_{i:i\neq m}A_{im}\bar{y}_{mi}\pi_i^*\delta^{(m)}_i\Big| \leq \Big|p\cdot\sum_{i:i\neq m}\bar{y}_{mi}\pi_i^*\delta^{(m)}\Big| + \Big|\sum_{i:i\neq m}(A_{im}-p)\bar{y}_{mi}\pi_i^*\delta^{(m)}_i\Big|\\
&\lesssim (p/n)\sqrt{n}\pnorm{\delta^{(m)}}{} + \sqrt{p\log n\cdot \sum_{i:i\neq m}(\delta^{(m)}_i\pi_i^*)^2} + \log n\cdot n^{-1}\pnorm{\delta^{(m)}}{\infty}\\
&= \mathfrak{o}(\sqrt{\frac{p}{nL}}) + \sqrt{\frac{p}{nL}}\cdot \frac{(\log n)^{3/2}}{np}.
\end{align*}
Putting together the pieces, we conclude that $\pnorm{\delta}{\infty} = \mathfrak{o}(1/\sqrt{npL})$ under the condition $np\gg (\log n)^{3/2}$. The proof is complete.
\end{proof}

\subsection{Completion of the proof}\label{subsec:proof_spec_completion}
\begin{proof}[Proof of Theorem \ref{thm:spec_expansion}]
Fix $i\in[n]$. Recall the definition of $\Delta_i$ in (\ref{def:spec_Delta}). Then by definition of $\theta^*$ and the condition $\bm{1}_n^\top \theta^* = 0$, we have
\begin{align*}
\tilde{\theta}_i - \theta_i^* &= \big(\log \hat{\pi}_i - \log \pi_i^*\big) - \ave\big(\log\hat{\pi} - \log\pi^*\big)\\
&= \log\Big(\frac{\hat{\pi}_i-\pi^*_i}{\pi^*_i} + 1\Big) - \frac{1}{n}\sum_{k=1}^n \log\Big(\frac{\hat{\pi}_k - \pi_k^*}{\pi_k^*} + 1\Big) \equiv \Delta_i + R_i.
\end{align*}
Here with $h(x)\equiv \log(x+1) - x$ for $x\geq -1$,
\begin{align*}
R_i = h\Big(\frac{\hat{\pi}_i - \pi^*_i}{\pi^*_i}\Big) - \frac{1}{n}\sum_{k=1}^n h\Big(\frac{\hat{\pi}_k-\pi_k^*}{\pi_k^*}\Big).
\end{align*}
Define the event
\begin{align}\label{def:spec_event}
\mathcal{A}\equiv \Big\{\max_{k\in [n]}\Big|\frac{\hat{\pi}_k - \pi_k^*}{\pi_k^*}\Big|\leq \frac{1}{4}\Big\}.
\end{align}
Then it follows from Lemma \ref{lem:spec_rate} that $\mathcal{A}$ holds with probability $1-\mathcal{O}(n^{-10})$. On the event $\mathcal{A}$, using $|h(x)|\leq x^2$ for $x\in(-1/2,1/2)$ and Lemma \ref{lem:spec_rate} again, we have
\begin{align*}
\frac{1}{n}\sum_{k=1}^n h\Big(\frac{\hat{\pi}_k-\pi_k^*}{\pi_k^*}\Big) \leq \frac{1}{n}\sum_{k=1}^n\Big(\frac{\hat{\pi}_k - \pi_k^*}{\pi_k^*}\Big)^2 \lesssim \frac{1}{npL} = \mathfrak{o}(\frac{1}{\sqrt{npL}}).
\end{align*}
On the other hand, by definition of $\Delta_i$ in (\ref{def:spec_Delta}), we have
\begin{align*}
h\Big(\frac{\hat{\pi}_i - \pi^*_i}{\pi^*_i}\Big) \leq \Big(\frac{\hat{\pi}_i - \pi_i^*}{\pi_i^*}\Big)^2 &= \frac{\hat{\pi}_i - \pi_i^*}{\pi_i^*} \cdot \Big(\Delta_i + \frac{1}{n}\sum_{k=1}^n \frac{\hat{\pi}_k - \pi_k^*}{\pi_k^*}\Big) \\
&= \mathfrak{o}(\Delta_i) + \mathfrak{o}(\frac{1}{\sqrt{npL}})
\end{align*}
by the estimates in Lemma \ref{lem:spec_rate}, and thus we have $|R_i| = \mathfrak{o}(\Delta_i) + \mathfrak{o}(1/\sqrt{npL})$.
The proof is now completed by invoking Proposition \ref{prop:spec_Delta_expansion}.
\end{proof}

\section{Proofs of Application I}\label{subsec:proof_app_clt}
In this section, we provide proofs for the application in Section \ref{subsec:app_clt}.

\begin{proof}[Proofs of Propositions \ref{prop:app_clt_mle} and \ref{prop:app_clt_spec}]
As the two proofs are similar, we only present the proof for the MLE. We start by proving the CLT with $\bar{\theta} = \theta^*$. By the expansion in Theorem \ref{thm:mle_expansion}, we have for each $i\in[n]$
\begin{align*}
\rho_i(\theta^*)(\hat{\theta}_i - \theta_i^*) = (1+\epsilon_{1,i})f_i + \epsilon_{2,i}'.
\end{align*}
Here $\epsilon_{2,i}' = \rho_i(\theta^*)\epsilon_{2,i}$ satisfies $\pnorm{\epsilon_{2,i}'}{\infty} = \mathfrak{o}(1)$ with probability $1-\mathcal{O}(n^{-10})$ since standard concentration yields that $\max_{i\in[k]}\rho_i(\theta^*) = \mathcal{O}(\sqrt{npL})$ with the same probability, and the sequence $\{f_i\}$ is given by
\begin{align*}
f_i = \rho_i(\theta^*)\frac{b_i}{d_i} = \frac{\sqrt{L}\cdot \sum_{j:j\neq i}A_{ij}\big(\bar{y}_{ij} - \psi(\theta_i^*-\theta_j^*)\big)}{\sqrt{\sum_{j:j\neq i}A_{ij}\psi'(\theta_i^*-\theta_j^*)}}.
\end{align*}

Let $Z\stackrel{d}{=} \mathcal{N}_k(0,I_k)$. By the Cram\'{e}r-Wold theorem, it suffices to prove that for any fixed $a\in \R^k$,
\begin{align*}
a^\top\big(f_1,\ldots,f_k\big) \leadsto a^\top Z.
\end{align*}
Let $\mathcal{A}_n$ be the event $\{A:\min_{m\in[n]}\sum_{i:i\neq m}A_{mi}\psi'(\theta_m^*-\theta_i^*) \geq c_0np\}$ for some small enough constant $c_0$ such that $\mathcal{A}$ holds with probability $1-\mathcal{O}(n^{-10})$ by Lemma \ref{lem:A_concen}. We first prove a conditional version of the above result by verifying the Lindeberg-Feller condition. With
\begin{align*}
U_{\ell i} \equiv \frac{a_\ell\sqrt{L}A_{\ell i}\big(\bar{y}_{\ell i} - \psi(\theta_\ell^* - \theta_i^*)\big)}{\sqrt{\sum_{m:m\neq \ell}A_{\ell m}\psi'(\theta_\ell^*-\theta_m^*)}}.
\end{align*}
we have
\begin{align*}
a^\top\big(f_1,\ldots,f_m\big) &= \sum_{\ell=1}^ka_\ell\frac{\sqrt{L}\cdot \sum_{i:i\neq \ell}A_{\ell i}\big[\bar{y}_{\ell i} - \psi(\theta_\ell^*-\theta_i^*)\big]}{\sqrt{\sum_{m:m\neq \ell}A_{\ell m} \psi'(\theta_\ell^*-\theta_m^*)}}\\
&\equiv \sum_{\ell=1}^k \sum_{i=1,i\neq \ell}^n U_{\ell i} = \sum_{\ell=1}^k\sum_{i=k+1}^n U_{\ell i} + \sum_{1\leq \ell < i\leq k}(U_{\ell i} + U_{i \ell})\\
&\equiv \sum_{\ell=1}^k \sum_{i=k+1}^n V_{\ell i} + \sum_{1\leq \ell < i\leq k}V_{\ell i}.
\end{align*}
Here $\{V_{\ell i}\}$ with index set $\{\ell\leq k, k+1\leq i\leq n\}\cup \{1\leq \ell<i\leq k\}$ are defined by: if $\ell\leq k, k+1\leq i\leq n$, 
\begin{align*}
V_{\ell i} &= U_{\ell i} = \frac{a_\ell \sqrt{L}A_{\ell i}\big(\bar{y}_{\ell i} - \psi(\theta_\ell^*-\theta_i^*)\big)}{\sqrt{\sum_{m:m\neq \ell}A_{\ell m}\psi'(\theta_\ell^*-\theta_m^*)}},
\end{align*}
and if $1\leq \ell < i\leq k$, 
\begin{align*}
V_{\ell i} &= \sqrt{L}A_{\ell i}\big(\bar{y}_{\ell i} - \psi(\theta_\ell^*-\theta_i^*)\big)\cdot \Big(\frac{a_\ell}{\sqrt{\sum_{m:m\neq \ell}A_{\ell m}\psi'(\theta_\ell^*-\theta_m^*)}} - \frac{a_i}{\sqrt{\sum_{m:m\neq i}A_{i m}\psi'(\theta_i^*-\theta_m^*)}}\Big).
\end{align*}
Note that $\{V_{\ell i}\}$ are independent across its index set, centered when conditioning on $A$, and satisfy
\begin{align*}
&\Big(\sum_{\ell=1}^k \sum_{i=k+1}^n  + \sum_{1\leq \ell < i\leq k}\Big)\E(V^2_{\ell i}|A)\\
& = \pnorm{a}{}^2 + \sum_{1\leq \ell<i\leq k}\frac{2a_ia_\ell}{\sqrt{\sum_{m:m\neq \ell}A_{\ell m}\psi'(\theta_\ell^*-\theta_m^*)}\sqrt{\sum_{m:m\neq i}A_{i m}\psi'(\theta_i^*-\theta_m^*)}}\\
&\rightarrow  \pnorm{a}{}^2 = \E (a^\top Z)^2,
\end{align*}
where the convergence holds under the condition $\min_{m\in[n]}\sum_{i:i\neq m}A_{mi}\psi'(\theta_m^*-\theta_i^*) \geq c_0np$ on the event $\mathcal{A}_n$. On the other hand, the Lindeberg-Feller condition holds trivially since $\{V_{\ell, i}\}$ are bounded. Hence the Lindeberg-Feller CLT applies to conclude that for any $A\in\mathcal{A}_n$,
\begin{align*}
a^\top\big(f_1,\ldots,f_k\big)|A \leadsto a^\top Z.
\end{align*}
For the unconditional version, it holds for every $t\in\R$ that
\begin{align*}
&\Prob(a^\top\big(f_1,\ldots,f_k\big) \leq t)\\
&= \E \Big[\Prob(a^\top\big(f_1,\ldots,f_k\big) \leq t| A)\mathbb{I}_{A\in\mathcal{A}_n}\Big] + \E \Big[\Prob(a^\top\big(f_1,\ldots,f_k\big) \leq t| A)\mathbb{I}_{A\in\mathcal{A}^c_n}\Big]\\
&\rightarrow \Prob(a^\top Z\leq t) + 0 = \Prob(a^\top Z\leq t)
\end{align*}
by the dominated convergence theorem. This concludes the claimed CLT for $\bar{\theta} = \theta^*$. The claim for $\bar{\theta} = \hat{\theta}$ follows from the fact that $\rho_k(\hat{\theta}) = \big(1+\mathfrak{o}_P(1)\big)\rho_k(\theta^*)$ uniformly over $k\in[n]$. Indeed, using the lower bound $\min_{i\in[n]}\rho_i(\theta^*) \wedge \rho_i(\hat{\theta}) \gtrsim \sqrt{npL}$ with probability $1-\mathcal{O}(n^{-10})$, we have
\begin{align*}
\big|\rho_i(\theta^*) - \rho_i(\hat{\theta})\big| &= \frac{\big|\rho_i^2(\theta^*) - \rho_i^2(\hat{\theta})\big|}{\rho_i(\theta^*) + \rho_i(\hat{\theta})} \lesssim \sqrt{\frac{L}{np}}\cdot \Big|\sum_{j:j\neq i}A_{ij}\big(\psi'(\theta_i^*-\theta_j^*) - \psi'(\hat{\theta}_i - \hat{\theta}_j)\big)\Big| \\
&\lesssim \sqrt{\frac{L}{np}} \Big(\sum_{j:j\neq i}A_{ij}\Big)\cdot \pnorm{\hat{\theta}-\theta^*}{\infty} \lesssim \sqrt{\frac{L}{np}}\cdot np\cdot \sqrt{\frac{\log n}{npL}}\\
&= \sqrt{\log n} = \mathfrak{o}(\sqrt{np L}).
\end{align*}
The case where $\{A_{ij}\}$ are replaced by $p$ can be dealt with similar arguments so the proof is complete. 
\end{proof}

\section{Proofs of Application II}\label{subsec:proof_app_rank}

In this section, we provide proofs for the application in Section \ref{subsec:app_rank}.

\begin{proof}[Proof of Proposition \ref{prop:app_rank}]
It can be readily verified that the event $\{r(1) \notin [n_1+1, n-n_2]\}$ is contained in the event $\{\theta_i^*\notin \mathcal{C}_i\text{ for some $i$}\}$, whose probability can be bounded by
\begin{align*}
\Prob\Big(\{\theta_i^*\notin \mathcal{C}_i\text{ for some $i$}\}\Big) \leq \Prob(\theta_1^* \notin \mathcal{C}_1) + \sum_{i\neq 1} \Prob(\theta_i^* \notin \mathcal{C}_i).
\end{align*}
Since $\Prob(\theta_1^* \notin \mathcal{C}_1) \rightarrow \alpha$ by construction of $\mathcal{C}_1$ and the CLT in Proposition \ref{prop:app_clt_mle}, it suffices to show the second probability is vanishing. 

Using preliminary estimates in Proposition \ref{prop:mle_prelim}, it is not hard to see that $\rho_i(\hat{\theta})$ are uniformly close to their populations: $\rho_i(\hat{\theta})= (1+\mathfrak{o}(1))\rho_i(\theta^*)$ with probability $1-\mathcal{O}(n^{-10})$ and $\mathfrak{o}(1)$ uniform over $i\in[n]$.  Note that $\rho_i(\theta^*)\asymp \sqrt{npL}$. Hence by the expansion in Theorem \ref{thm:mle_expansion}, for each $i\neq 1$, we have for large enough $n$
\begin{align*}
\Prob(\theta_i^*\notin \mathcal{C}_i) &= \Prob\Big(\Big|(1+\epsilon_{1,i})\frac{b_i}{d_i} + \epsilon_{2,i}\Big|\geq \tau_i\Big) \\
&\leq \Prob\Big(\Big|\frac{b_i}{d_i}\Big|\geq \Big(1+\frac{c_0}{2}\Big)\sqrt{2\log n \cdot \rho_i^{-2}(\theta^*)}\Big) + \mathcal{O}(n^{-10}).
\end{align*}
Now note that $b_i/d_i = \sum_{j:j\neq i}\sum_{\ell=1}^L Z_{ij\ell}$, where $Z_{ij\ell}\equiv A_{ij}(y_{ij\ell} - \psi(\theta_i^*-\theta_j^*))/\big(L\cdot\sum_{k:k\neq i}A_{ik}\psi'(\theta_i^*-\theta_k^*)\big)$ are centered and independent when conditioning on $A$, and satisfy
\begin{align*}
\sum_{j:j\neq i}\sum_{\ell=1}^L \E(Z_{ij\ell}^2|A) = \rho_i^{-2}(\theta^*), \quad \max_{j:j\neq i}\max_{1\leq \ell\leq L}|Z_{ij\ell}| \leq 2\rho_i^{-2}(\theta^*).
\end{align*}
Hence for large enough $n$, it follows from Bernstein's inequality (with the prescribed constants in \cite[Theorem 2.10]{boucheron2013concentration}) applied conditionally on $A$ that $\Prob\big(|b_i/d_i|\geq \tilde{\tau}_i\big) \leq 2n^{-(1+c_0/2)}$. Take the union bound to complete the proof.
\end{proof}

\section{Proofs of Application III}\label{sec:proof_app_constant}

In this section, we provide proofs for the application in Section \ref{subsec:app_constant}.

\subsection{Proof of Proposition \ref{prop:mle_constant}}\label{subsec:proof_app_constant_mle}

\begin{proof}[Proof of Proposition \ref{prop:mle_constant}]
Recall the definition of $\delta$ in (\ref{def:delta}), and the terms $b_i,d_i$ in Theorem \ref{thm:mle_expansion}. Then for any $\epsilon > 0$, with the prescribed probability, we have
\begin{align*}
&\pnorm{\hat{\theta} -\theta^*}{}^2\leq (1+\epsilon)\sum_{m=1}^n \Big( \frac{b_m}{d_m}\Big)^2 + (1+\epsilon^{-1})\pnorm{\delta}{}^2.
\end{align*}
By Proposition \ref{prop:delta_bound}, we have $\pnorm{\delta}{}^2 = \mathfrak{o}((pL)^{-1})$. For the main term, let $D \equiv \textrm{diag}(d_1,\ldots,d_n)$ be a diagonal matrix, which is invertible with the prescribed probability, and when this happens,
\begin{align}\label{eq:main_decomp}
\sum_{m=1}^n \Big( \frac{b_m}{d_m}\Big)^2 = \pnorm{D^{-1}b}{}^2 = \E\big(\pnorm{D^{-1}b}{}^2|A\big) + \big(\pnorm{D^{-1}b}{}^2-\E\big(\pnorm{D^{-1}b}{}^2|A\big)\big).
\end{align}
The second term in the above display can be bounded by Lemma \ref{lem:quad_form_concentration} as follows. Note that $b_m = \sum_{i:i\neq m}A_{mi}Z_{mi}$ with $Z_{mi} = \bar{y}_{mi}-\psi(\theta_m^*-\theta_i^*)$ therein satisfying the anti-symmetry condition $Z_{mi} = -Z_{im}$. Moreover, up to anti-symmetry, $Z_{mi}$ are independent sub-Gaussian variables with variance proxy bounded by a constant multiple (only depending on $\kappa$) of $L^{-1}$. Hence with the prescribed probability, Lemma \ref{lem:quad_form_concentration} implies
\begin{align*}
\big(\pnorm{D^{-1}b}{}^2-\E\big(\pnorm{D^{-1}b}{}^2|A\big)\big) &\lesssim L^{-1}\sqrt{\log n}\Big(\min_{i\in[n]}d_i\Big)^{-2} \sqrt{n} \Big(\max_{i\in[n]}\sum_{j:j\neq i}A_{ij}\Big) \\
&\lesssim  \sqrt{\log n} \frac{\sqrt{n} np}{(np)^2L}=\mathfrak{o}((pL)^{-1}),
\end{align*}
where we use Lemma \ref{lem:A_concen} in the second inequality.
Next, for the first term in (\ref{eq:main_decomp}), it holds by direct calculation that
\begin{align*}
\E\big(\pnorm{D^{-1}b}{}^2|A\big) &= \frac{\tr\big(D^{-1}\big)}{L} \stackrel{(*)}{=} \big(1+\mathfrak{o}(1)\big)\frac{\tr\big((\E D)^{-1}\big)}{L}\\
&= \big(1+\mathfrak{o}(1)\big)\frac{1}{pL}\sum_{m=1}^n\frac{1}{\sum_{i:i\neq m}\psi'(\theta_m^*-\theta_i^*)}, 
\end{align*}
where $(*)$ follows from the fact that uniformly over $i\in[n]$, $d_i \asymp np$ and $\pnorm{D - \E D}{\op} = \max_i |d_i - \E d_i|\lesssim \sqrt{np\log n}$ by standard concentration. Putting together the two estimates, we have
\begin{align*}
\pnorm{\hat{\theta} - \theta^*}{}^2 \leq \frac{(1+\epsilon)}{pL}(1+\mathfrak{o}(1))\cdot\sum_{i=1}^n\Big( \sum_{j:j\neq i}\psi'(\theta_i^*-\theta_j^*)\Big)^{-1} + C_\kappa'(1+\epsilon^{-1})\mathfrak{o}\big(\frac{1}{pL}\big).
\end{align*}
The proof for the upper bound is now complete by choosing $\epsilon\downarrow 0$ slow enough (e.g. $\epsilon \asymp (\log n/n)^{1/4}$) such that the second term is still $\mathfrak{o}\big((pL)^{-1}\big)$. The proof for the lower bound is analogous by noting that for any $\epsilon > 0$,
\begin{align*}
\pnorm{\hat{\theta} -\theta^*}{}^2&\geq (1-\epsilon)\sum_{m=1}^n \Big( \frac{b_m}{d_m}\Big)^2 - (\epsilon^{-1} - 1)\pnorm{\delta}{}^2,
\end{align*}
and then using the same estimates established in the upper bound proof. 
\end{proof}

\begin{lemma}\label{lem:quad_form_concentration}
Let $A\in\{0,1\}^{n\times n}$ be a symmetric matrix and $D\in\mathbb{R}^{n\times n}$ be a diagonal matrix with positive diagonal entries $d_1,\ldots, d_n$. Define $b\in\R^n$ such that $b_i = \sum_{j:j\neq i}A_{ij}Z_{ij}$, where $\{Z_{ij}\}_{i,j:i<j}$ are a sequence of independent sub-Gaussian random variables with variance proxy $\tau^2$ and $Z_{ij} = -Z_{ji}$ for any $i>j$. Then there exists some universal $C > 0$ such that
\begin{align*}
\abs{\pnorm{D^{-1}b}{}^2 - \E\pnorm{D^{-1}b}{}^2} \leq  C \tau^2\sqrt{\log n}\Big(\min_{i\in[n]}d_i\Big)^{-2} \sqrt{n} \Big(\max_{i\in[n]}\sum_{j:j\neq i}A_{ij}\Big),
\end{align*}
with probability at least $1 - \mathcal{O}(n^{-10})$.
\end{lemma}
\begin{proof}
Using the symmetry of $A$ and the anti-symmetry of $\{Z_{ij}\}_{i,j:i<j}$, we have
\begin{align}
\nonumber &\pnorm{D^{-1}b}{}^2 = \sum_{i=1}^n \Big(\frac{b_i}{d_i}\Big)^2 =\sum_{i=1}^n  d_i^{-2} \Big(\sum_{j:j\neq i} A_{ij}Z_{ij}\Big)^2\\
\label{eq:9.1eq1}& = \sum_{i=1}^nd_i^{-2} \Big(\sum_{j:j> i} A_{ij}Z_{ij} +\sum_{j:j< i} A_{ji} (-Z_{ji})  \Big)^2\\
\nonumber & = \sum_{i=1}^nd_i^{-2}\Big( \sum_{j,j':j,j'> i} A_{ij}A_{ij'}Z_{ij}Z_{ij'} +  \sum_{j,j':j,j'< i} A_{ji}A_{j'i}Z_{ji}Z_{j'i} -  \sum_{j,j':j,j'>i} A_{ij}A_{j'i}Z_{ij}Z_{j'i}  \Big).
\end{align}
which  can be written as a quadratic form of $\{Z_{ij}\}_{i,j:i<j}$.  Let $\mathcal{I}$ be the set of double indices defined by $\mathcal{I}\equiv\{(i,j): i,j\in[n], i<j\}$ with $N \equiv |\mathcal{I}| = {n\choose 2}$. Let $Z \equiv (Z_{ij})_{(i,j)\in \mathcal{I}}\in \R^N$ such that it has independent sub-Gaussian coordinates with variance proxy  $\tau^2$.  Then there exists a matrix $S= (S_{(i,j), (i',j')})\in\mathbb{R}^{N\times N}$ such that $\pnorm{D^{-1}b}{}^2 = Z^\top SZ = \sum_{(i,j),(i',j')\in\mathcal{I}}S_{(i,j),(i',j')}Z_{(i,j)}Z_{(i',j')}$ where $S$ is defined as
\begin{align*}
S_{(i,j),(i',j')} = \begin{cases}
A_{ij}A_{i'j'} (d_i^{-2} + d_j^{-2}),\quad \text{ if }i'=i,j'=j,\\
A_{ij}A_{i'j'}d_i^{-2},\quad\quad \quad\quad\;\;  \text{ if }i'=i,j'\neq j,\\
A_{ij}A_{i'j'}d_j^{-2},\quad\quad \quad\quad\;\;   \text{ if }i'\neq i,j'= j,\\
-A_{ij}A_{i'j'}d_i^{-2},\quad\quad\quad\;\;  \;  \text{ if }j'=i,\\
-A_{ij}A_{i'j'}d_j^{-2},\quad\quad\quad\;\;  \;  \text{ if }i'=j,\\
0,\quad\quad\quad\quad\quad\quad\quad\quad\quad \text{ otherwise.}
\end{cases}
\end{align*}
That is, the entry in $S$ is zero if the two corresponding edges $(i,j),(i',j')$ have no joint vertex, as the matrix $A$ can be interpreted as an adjacency matrix of a graph. Let $h_{\max}= \max_{i\in[n]}\sum_{j:j\neq i}A_{ij}$ which be understood as the maximum degree of the graph.

To apply the Hanson-Wright inequality (cf., \cite[Theorem 1.1]{rudelson2013hanson}) to control the deviation of $Z^\top SZ$, we now relate the quantities $\pnorm{S}{F}^2$ and $\pnorm{S}{\op}^2$. We first bound  $\pnorm{S}{F}^2$.
Denote $d_{\min} = \min_{i\in[n]}d_i$. We have
\begin{align}\label{ineq:hs_norm}
\nonumber \pnorm{S}{F}^2&\leq 4d_{\min}^{-4} \sum_{(i,j),(i',j')\in \mathcal{I}}A_{ij}A_{i'j'} \mathbb{I}\left\{\{i,j\}\cap \{i',j'\}\neq \emptyset\right\} \\
&\leq 8d_{\min}^{-4} \sum_{(i,j)\in \mathcal{I}}A_{ij} h_{\max} \leq 8d_{\min}^{-4} nh^2_{\max}.
\end{align}

For $\pnorm{S}{\op}$, let $v = (v_{(i,j)})\in \R^N$ be any vector such that $\pnorm{a}{}^2 = \sum_{i<j}a_{(i,j)}^2 \leq 1$. Note that $S$ is symmetric and  in the above derivation of $S$ we have (\ref{eq:9.1eq1}) being equal to $Z^\top SZ$ which holds without any condition on $Z$. Hence,
\begin{align*}
v^\top S v &=  \sum_{i=1}^nd_i^{-2} \Big(\sum_{j:j> i} A_{ij}v_{ij} +\sum_{j:j< i} A_{ji} (-v_{ji})  \Big)^2 \\
&\leq 4 \sum_{i=1}^nd_i^{-2} \Big(\Big(\sum_{j:j> i} A_{ij}\Big) \Big(\sum_{j:j> i} v_{ij}^2\Big)  + \Big(\sum_{j:j< i} A_{ij}\Big) \Big(\sum_{j:j< i} v_{ij}^2\Big)  \Big)\\
&\leq 4d_{\min}^{-2} h_{\max}  \sum_{i=1}^n\Big(\Big(\sum_{j:j> i} v_{ij}^2\Big)  +\Big(\sum_{j:j< i} v_{ij}^2\Big)  \Big)\\
&\leq 8d_{\min}^{-2} h_{\max}.
\end{align*}
Then we have
\begin{align}\label{ineq:op_norm}
\pnorm{S}{\op}\leq 8d_{\min}^{-2} h_{\max}.
\end{align}

Since each $Z_{jk}$ is sub-Gaussian with variance proxy $\tau^2$, the Hanson-Wright inequality now applies to conclude that for any $t\geq 0$,
\begin{align*}
\Prob\big(\abs{Z^\top SZ^\top - \E(Z^\top SZ^\top)}\geq t\big) \leq 2\exp\Big(-C \frac{t^2}{\tau^4\pnorm{S}{F}^2}\wedge \frac{t}{\tau^2\pnorm{S}{\op}}\Big).
\end{align*}
for some universal $C > 0$. So by choosing $t = \tau^2\cdot\mathcal{O}\big(\sqrt{\log n}\pnorm{S}{F}\vee (\log n)\pnorm{S}{\op}\big)$ and plugging in the estimates in (\ref{ineq:hs_norm}) and (\ref{ineq:op_norm}), we obtain
\begin{align*}
\abs{\pnorm{D^{-1}b}{}^2 - \E\pnorm{D^{-1}b}{}^2} &\lesssim \tau^2\Big(\sqrt{\log n}d_{\min}^{-2} \sqrt{n} h_{\max}  +( \log n)d_{\min}^{-2} h_{\max} \Big)\\
&\leq 2 \tau^2\sqrt{\log n}d_{\min}^{-2} \sqrt{n} h_{\max},
\end{align*}
with the prescribed probability.
\end{proof}

\subsection{Proof of Proposition \ref{prop:spec_constant}}\label{subsec:proof_app_constant_spec}
\begin{proof}[Proof of Proposition \ref{prop:spec_constant}]

We work on the event $\mathcal{A}$ defined in (\ref{def:spec_event}) which holds with the prescribed probability. Using $x-x^2 \leq \log(1+x)\leq x$ for $x\in (-1/2, 1/2)$ and the fact that
\begin{align*}
\tilde{\theta}_m - \theta^*_m = \log\Big(\frac{\hat{\pi}_m-\pi_m^*}{\pi_m^*}+1\Big) - \frac{1}{k}\sum_{k=1}^n\log\Big(\frac{\hat{\pi}_k-\pi_k^*}{\pi_k^*}+1\Big),
\end{align*}
we have (recall the definition of $\Delta$ in (\ref{def:spec_Delta}))
\begin{align*}
\Delta_m - \Big(\frac{\hat{\pi}_m - \pi^*_m}{\pi^*_m}\Big)^2\leq \tilde{\theta}_m-\theta_m^*\leq  \Delta_m+ \frac{1}{n}\sum_{i=1}^n \Big(\frac{\hat{\pi}_i - \pi^*_i}{\pi^*_i}\Big)^2.
\end{align*}
This implies that for any $\epsilon > 0$, 
\begin{align*}
\pnorm{\tilde{\theta}-\theta^*}{}^2 &\leq (1+\epsilon)\pnorm{\Delta}{}^2 + \\
&\quad\quad\quad\mathcal{O}(\epsilon^{-1})\cdot \max\Big\{\sum_{m=1}^n \Big(\frac{\hat{\pi}_m - \pi^*_m}{\pi^*_m}\Big)^4, n\Big(\frac{1}{n}\sum_{m=1}^n \Big(\frac{\hat{\pi}_m - \pi^*_m}{\pi^*_m}\Big)^2\Big)^2\Big\}\\
&= (1+\epsilon)\pnorm{\Delta}{}^2 + \mathcal{O}(\epsilon^{-1})\cdot \sum_{m=1}^n \Big(\frac{\hat{\pi}_m - \pi^*_m}{\pi^*_m}\Big)^4.
\end{align*}
Using Lemma \ref{lem:spec_rate} and the lower bound $\min_{m\in[n]}\pi_m^* \gtrsim n^{-1}$, we have
\begin{align*}
\sum_{m=1}^n \Big(\frac{\hat{\pi}_m - \pi^*_m}{\pi^*_m}\Big)^4 &\leq n^4\cdot \bigpnorm{\frac{\hat{\pi}-\pi^*}{\pi^*}}{\infty}^2\cdot \bigpnorm{\frac{\hat{\pi}-\pi^*}{\pi^*}}{}^2\\
&\lesssim n^4 \cdot n^{-2}\frac{\log n}{npL} \cdot n^{-2}\frac{1}{pL} = \mathfrak{o}(\frac{1}{pL}). 
\end{align*}
Hence the upper bound follows from Proposition \ref{prop:Delta_risk} by choosing $\epsilon \downarrow 0$ slow enough, e.g. of the order $\sqrt{\log n/(npL)}$. The lower bound proof is analogous so the proof is complete.
\end{proof}

\begin{proposition}\label{prop:Delta_risk}
Recall the definition of $\Delta$ in (\ref{def:spec_Delta}). Suppose that $\kappa = \mathcal{O}(1)$ and $np\gg \log n$. Then the following holds with probability $1-\mathcal{O}(n^{-10})$.
\begin{align*}
\pnorm{\Delta}{}^2 = \frac{1+\mathfrak{o}(1)}{pL}\cdot \sum_{m=1}^n \frac{\sum_{i:i\neq m}(e^{\theta_i^*}+e^{\theta_m^*})^2\psi'(\theta_i^*-\theta_m^*)}{\Big(\sum_{i:i\neq m}(e^{\theta_i^*} + e^{\theta_m^*})\psi'(\theta_i^*-\theta_m^*)\Big)^2}.
\end{align*}
\end{proposition}
\begin{proof}[Proof of Proposition \ref{prop:Delta_risk}]
Recall the definition of $\bar{\pi}_m$ in (\ref{def:spec_pi_bar}). By (\ref{eq:spec_expansion}), we have
\begin{align*}
\pnorm{\Delta}{}^2 \leq (1+ \epsilon)\sum_{m=1}^n \Big(\frac{\bar{\pi}_m-\pi_m^*}{\pi_m^*}\Big)^2 + \mathcal{O}(\epsilon^{-1})\pnorm{\delta}{}^2.
\end{align*}
By Proposition \ref{prop:delta_norm_spec}, we have $\pnorm{\delta}{} = \mathfrak{o}(1/\sqrt{pL})$ with the prescribed probability. For the main term, we can further decompose it as
\begin{align*}
\frac{\bar{\pi}_m - \pi^*_m}{\pi^*_m} &= \frac{\sum_{i:i\neq m}A_{mi}\big(\bar{y}_{mi}\pi_i^* - \bar{y}_{im}\pi_m^*\big)}{\pi_m^*\cdot \sum_{i:i\neq m}A_{mi}\bar{y}_{im}}\\
&= \frac{\sum_{i:i\neq m}A_{mi}\big(\bar{y}_{mi}\pi_i^* - \bar{y}_{im}\pi_m^*\big)}{\pi_m^*\cdot \sum_{i:i\neq m}A_{mi}\psi(\theta_i^*-\theta_m^*)} + \\
&\quad\quad\quad\sum_{i:i\neq m}\frac{A_{mi}\big(\bar{y}_{mi}\pi_i^* - \bar{y}_{im}\pi_m^*\big)}{\pi_m^*}\Big(\frac{1}{\sum_{i:i\neq m}A_{mi}\bar{y}_{im}} - \frac{1}{\sum_{i:i\neq m}A_{mi}\psi(\theta_i^*-\theta_m^*)}\Big)\\
&\equiv (I)_m + (II)_m. 
\end{align*}
For $(II)_m$ in the above display, standard concentration and the lower bound $\min_{m\in[n]}\pi_m^*\gtrsim n^{-1}$ yield that 
\begin{align*}
|(II)_m| \lesssim n\cdot(np)^{-2}\cdot \Big|\sum_{i:i\neq m}A_{mi}\big(\bar{y}_{mi}\pi_i^* - \bar{y}_{im}\pi_m^*\big)\Big|\cdot\Big|\sum_{i:i\neq m}A_{mi}\big(\bar{y}_{im} - \psi(\theta_i^*-\theta_m^*)\big)\Big|.
\end{align*}
Hence by Lemma \ref{lem:mle_concentration}, 
\begin{align*}
\sum_{m=1}^n (II)_m^2 &\lesssim \frac{1}{n^2p^4}\cdot \max_{m\in[n]}\Big|\sum_{i:i\neq m}A_{mi}\big(\bar{y}_{mi}\pi_i^* - \bar{y}_{im}\pi_m^*\big)\Big|^2\times\\
&\quad\quad\quad \sum_{m=1}^n \Big|\sum_{i:i\neq m}A_{mi}\big(\bar{y}_{im} - \psi(\theta_i^*-\theta_m^*)\big)\Big|^2\\
& \lesssim \frac{1}{n^2p^4} \cdot \frac{p\log n}{nL}\cdot \frac{n^2p}{L} = \frac{\log n}{npL}\cdot \frac{1}{pL} = \mathfrak{o}(\frac{1}{pL}).
\end{align*}
For $(I)_m$, we have by definition of $\bar{\pi}_m$ that
\begin{align*}
(I)_m = \frac{\sum_{i:i\neq m}A_{mi}\big(\bar{y}_{mi}\pi_i^* - \bar{y}_{im}\pi_m^*\big)}{\pi_m^*\cdot \sum_{i:i\neq m}A_{mi}\psi(\theta_i^*-\theta_m^*)} = \big(D^{-1}b\big)_m,
\end{align*}
where $D = \diag(\tilde{d}_1,\ldots,\tilde{d}_n)$ with $\tilde{b},\tilde{d}\in\R^n$ defined in Theorem \ref{thm:spec_expansion}. This implies 
\begin{align}\label{eq:main_decomp_spec}
\pnorm{D^{-1}\tilde{b}}{}^2 = \E(\pnorm{D^{-1}\tilde{b}}{}^2|A) + \big(\pnorm{D^{-1}\tilde{b}}{}^2 - \E(\pnorm{D^{-1}\tilde{b}}{}^2|A)\big).
\end{align}
The second term in the above display can be bounded by Lemma \ref{lem:quad_form_concentration} as follows. Note that $\tilde{b}_m = \sum_{i:i\neq m}A_{mi}Z_{mi}$ with $Z_{mi} = \bar{y}_{mi}\pi_i^* - \bar{y}_{im}\pi_m^*$ therein satisfying the anti-symmetry condition $Z_{mi} = -Z_{im}$. Moreover, up to anti-symmetry, $Z_{mi}$ are independent sub-Gaussian variables with variance proxy bounded by a constant multiple (only depending on $\kappa$) of $(n^2L)^{-1}$. Hence with the prescribed probability, Lemma \ref{lem:quad_form_concentration} implies
\begin{align*}
\big(\pnorm{D^{-1}\tilde{b}}{}^2-\E\big(\pnorm{D^{-1}\tilde{b}}{}^2|A\big)\big) &\lesssim  \frac{1}{n^2L}\sqrt{\log n}\Big(\min_{i\in[n]}\tilde d_i\Big)^{-2} \sqrt{n} \Big(\max_{i\in[n]}\sum_{j:j\neq i}A_{ij}\Big) \\
&\lesssim  \sqrt{\log n} \frac{\sqrt{n} np}{(p)^2n^2L}=\mathfrak{o}((pL)^{-1}),
\end{align*}
where we use Lemma \ref{lem:A_concen} in the second inequality.

Lastly, by direct calculation, the main term $\E(\pnorm{D^{-1}\tilde{b}}{}^2|A)$ in (\ref{eq:main_decomp_spec}) satisfies 
\begin{align*}
\E(\pnorm{D^{-1}\tilde{b}}{}^2|A) &= \frac{1}{L}\sum_{m=1}^n \frac{\sum_{i:i\neq m}A_{im}(\pi_i^*+\pi^*_m)^2\psi'(\theta_i^*-\theta_m^*)}{\Big(\pi_m^*\cdot \sum_{i:i\neq m}A_{im}\psi(\theta_i^*-\theta_m^*)\Big)^2}\\
&= \frac{1}{L}\sum_{m=1}^n \frac{\sum_{i:i\neq m}A_{im}(e^{\theta_i^*}+e^{\theta_m^*})^2\psi'(\theta_i^*-\theta_m^*)}{\Big(\sum_{i:i\neq m}A_{im}(e^{\theta_i^*} + e^{\theta_m^*})\psi'(\theta_i^*-\theta_m^*)\Big)^2}.
\end{align*}
It remains to show that both the numerator and the denominator concentrate around their mean value with respect to $A$. For the denominator, it can be readily verified that $\E\Big(\sum_{i:i\neq m}A_{im}(e^{\theta_i^*} + e^{\theta_m^*})\psi'(\theta_i^*-\theta_m^*)\Big) \asymp np$, and Bernstein's inequality yields that with the prescribed probability, 
\begin{align*}
\Big|\sum_{i:i\neq m}(A_{im}-p)(e^{\theta_i^*} + e^{\theta_m^*})\psi'(\theta_i^*-\theta_m^*)\Big| \lesssim \sqrt{np\log n} + \log n = \mathfrak{o}(np).
\end{align*}
A similar calculation holds for the numerator. In summary, we have for any $\epsilon > 0$,
\begin{align*}
\pnorm{\Delta}{}^2 \leq (1+\epsilon)\frac{(1+\mathfrak{o}(1))}{pL}\sum_{m=1}^n \frac{\sum_{i:i\neq m}(e^{\theta_i^*}+e^{\theta_m^*})^2\psi'(\theta_i^*-\theta_m^*)}{\Big(\sum_{i:i\neq m}(e^{\theta_i^*} + e^{\theta_m^*})\psi'(\theta_i^*-\theta_m^*)\Big)^2} +\mathcal{O}(\epsilon^{-1})\mathfrak{o}(\frac{1}{pL}).
\end{align*}
A lower bound can be similarly proved, and hence the proof is complete by choosing $\epsilon\downarrow 0$ slow enough (e.g. $\epsilon\asymp (\log n/n)^{1/4}$). 
\end{proof}

\subsection{Proof of Lower bound (Theorem \ref{thm:lower})}\label{subsubsec:proof_lower}
\begin{proof}[Proof of Theorem \ref{thm:lower}]
Let $r_n$ be a sequence such that $(npL)^{-1/2} \ll r_n \ll \epsilon_n\wedge 1$. For each $i\in[n]$, let $\pi_i(\cdot)$ be a prior density defined by $\pi_i(x) = r_nK\big((x - \theta^*_i)/r_n\big)$, where $K(\cdot)$ is a $C^\infty$ kernel function supported on $[-1,1]$ such that $K(\pm 1) = 0$.

By definition of $\Theta(\kappa)$, we have
\begin{align*}
\inf_{\hat{\theta}}\sup_{\theta\in B(\theta^*,\epsilon_n)\cap \Theta(\kappa)} \E_\theta\pnorm{\hat{\theta}-\theta}{}^2 = \inf_{\hat{\theta}}\sup_{\theta\in B(\theta^*,\epsilon_n)\cap \Theta(\kappa)}\E_\theta\pnorm{\big(\hat{\theta} - \ave(\hat{\theta})\bm{1}_n\big) - \big(\theta - \ave(\theta)\bm{1}_n\big)}{}^2.
\end{align*}
Then using the fact that $\pnorm{\theta - \ave(\theta)\bm{1}_n}{}^2 = n^{-1}\cdot\sum_{i<j}(\theta_i - \theta_j)^2$ for any $\theta\in\R^n$, we have
\begin{align}\label{ineq:lower_reduction}
\notag&\inf_{\hat{\theta}}\sup_{\theta\in B(\theta^*,\epsilon_n)\cap \Theta(\kappa)} \E_\theta\pnorm{\big(\hat{\theta} - \ave(\hat{\theta})\bm{1}_n\big) - \big(\theta - \ave(\theta)\bm{1}_n\big)}{}^2\\
\notag&= \frac{1}{n}\inf_{\hat{\theta}}\sup_{\theta\in B(\theta^*,\epsilon_n)\cap \Theta(\kappa)}\sum_{i<j}\E_\theta\big[\big(\hat{\theta}_i - \hat{\theta}_j\big) - \big(\theta_i - \theta_j\big)\big]^2\\
\notag&\stackrel{(*)}{\geq} \frac{1}{n}\inf_{\hat{\theta}}\sup_{\theta\in B(\theta^*,\epsilon_n/2)\cap \tilde{\Theta}(\kappa)}\sum_{i<j}\E_\theta\big[\big(\hat{\theta}_i - \hat{\theta}_j\big) - \big(\theta_i - \theta_j\big)\big]^2\\
\notag&\stackrel{(**)}{\geq} \frac{1}{n}\inf_{\hat{\theta}}\int_{\theta\in[\theta^* - r_n\bm{1}_n, \theta^*+r_n\bm{1}_n]}\sum_{i<j}\E_\theta\big[\big(\hat{\theta}_i - \hat{\theta}_j\big) - \big(\theta_i - \theta_j\big)\big]^2\big(\prod_{i=1}^n\pi_i(\theta_i)\big)\big(\prod_{i=1}^n \d\theta_i\big)\\
\notag&= \frac{1}{n}\inf_{\hat{\theta}}\sum_{i<j}\int \big(\prod_{k:k\neq i,j}\pi_k(\theta_k)\big)\big(\prod_{k:k\neq i,j}\d\theta_k\big) \Big[\int \E_\theta\big[\big(\hat{\theta}_i - \hat{\theta}_j\big) - \big(\theta_i - \theta_j\big)\big]^2\pi_i(\theta_i)\pi_j(\theta_j)\d\theta_i\d\theta_j\Big]\\
\notag&\stackrel{(***)}{\geq} \frac{1}{n}\sum_{i<j}\int \big(\prod_{k:k\neq i,j}\pi_k(\theta_k)\big)\big(\prod_{k:k\neq i,j}\d\theta_k\big) \Big[\inf_{\tilde{T}_{ij}}\int \E_\theta\big[\tilde{T}_{ij} - \big(\theta_i - \theta_j\big)\big]^2\pi_i(\theta_i)\pi_j(\theta_j)\d\theta_i\d\theta_j\Big]\\
&\equiv \frac{1}{n}\sum_{i<j}\int \big(\prod_{k:k\neq i,j}\pi_k(\theta_k)\big)R_{ij}(\theta_{-i,-j})\big(\prod_{k:k\neq i,j}\d\theta_k\big).
\end{align}
Here $(*)$ holds as follows with $\tilde{\Theta}(\kappa) \equiv \{\theta'\in\R^n:\max_i\theta'_i - \min_i\theta'_i \leq \kappa\}$ dropping the centering condition. For any $\hat{\theta}$, with
\begin{align*}
\theta^* \equiv \argmax_{\theta\in B(\theta^*,\epsilon_n/2)\cap \tilde{\Theta}(\kappa)}\E_\theta\big[\big(\hat{\theta} - \ave(\hat{\theta})\bm{1}_n\big) - \big(\theta- \ave(\theta)\bm{1}_n\big)\big]^2
\end{align*}
and $\tilde{\theta}\equiv \theta^* - \ave(\theta^*)\bm{1}_n$, we have by the translation invariance of the model 
\begin{align*}
\E_{\theta^*}\big[\big(\hat{\theta} - \ave(\hat{\theta})\bm{1}_n\big) - \big(\theta^*- \ave(\theta^*)\bm{1}_n\big)\big]^2 = \E_{\tilde{\theta}}\big[\big(\hat{\theta} - \ave(\hat{\theta})\bm{1}_n\big) - \big(\tilde{\theta}- \ave(\tilde{\theta})\bm{1}_n\big)\big],
\end{align*}
and since $\theta^*\in B(\theta^*, \epsilon_n/2)$ and $\theta^*\in\Theta(\kappa)$,
\begin{align*}
|\ave(\theta^*)| = |\frac{1}{n}\sum_{i=1}^n (\theta^*_i - \theta^*_i)| \leq \pnorm{\theta^* - \theta^*}{\infty} \leq \epsilon_n/2.
\end{align*}
This entails that $\pnorm{\tilde{\theta} - \theta^*}{\infty} \leq \epsilon_n$, and hence $\tilde{\theta}\in B(\theta^*,\epsilon_n) \cap \Theta(\kappa)$, which concludes the proof of $(*)$. Next $(**)$ follows with $[\theta^* - r_n\bm{1}_n, \theta^*+r_n\bm{1}_n]\equiv \times_{i=1}^n [\theta_i^* - r_n, \theta^*_i + r_n]$ by bounding the (local) maximal risk by the Bayes risk induced by the product prior $\prod_{i=1}^n \pi_i(\theta_i)$ on $[\theta^* - r_n\bm{1}_n, \theta^* + r_n\bm{1}_n]$. Note that in $(**)$ we have $B(\theta^*,\epsilon_n/2)\subset \tilde{\Theta}(\kappa)$ for large enough $n$ due to the condition $\max_{i}\theta^*_i - \min_i\theta^*_i\leq \kappa/2$. Lastly the infimum over $\tilde{T}_{ij}$ in $(***)$ is taken over all estimators of $\theta_i - \theta_j$ with the knowledge of $\theta_{-i,-j}$. 

Let $\mu \equiv (\theta_i, \theta_j)\in \R^2$ and $T_{ij}\equiv T_{ij}(\mu) \equiv \theta_i - \theta_j \in \R$. Now by the multivariate van Tree's inequality in \cite[Equation (11)]{gill1995applications}, each $R_{ij}(\theta_{-i},\theta_{-j})$ can be lower bounded by
\begin{align}\label{ineq:Rij_lower}
R_{ij}(\theta_{-i},\theta_{-j}) \geq \frac{1}{n}\int \mathcal{J}_{ij}(\mu;\theta_{-i,-j})\pi_i(\theta_i)\pi_j(\theta_j)\d\theta_i\d\theta_j - \frac{1}{n^2}\tilde{\mathcal{I}}_{ij}(\pi_i,\pi_j), 
\end{align}
where with $\mathcal{I}(\mu)$ being the Fisher information matrix of $\mu = (\theta_i,\theta_j)$ defined below,
\begin{align*}
\mathcal{J}_{ij}(\mu;\theta_{-i,-j})&\equiv \Big(\frac{\partial T_{ij}}{\partial \mu}\Big)^\top\mathcal{I}(\mu)^{-1}\Big(\frac{\partial T_{ij}}{\partial \mu}\Big),
\end{align*}
and $\tilde{\mathcal{I}}_{ij}(\pi_i,\pi_j)$ is also defined below after we introduce a few other notations. We now estimate the two terms on the right side of (\ref{ineq:Rij_lower}).

(\textbf{First term in (\ref{ineq:Rij_lower})}) The fisher information of the parameter $\mu$ can be calculated as
\begin{align*}
\mathcal{I}(\mu) &\equiv \E_{\mu}\Big[\Big(\frac{\partial \ell_n(\mu;\theta_{-i,-j})}{\partial \mu}\Big)^\top\Big(\frac{\partial \ell_n(\mu;\theta_{-i,-j})}{\partial \mu}\Big)\Big]\\
&= (pL)\cdot 
\begin{bmatrix}
\sum_{k:k\neq i}\psi^\prime(\theta_i - \theta_k) & -\psi^\prime(\theta_i - \theta_j)\\
-\psi^\prime(\theta_i - \theta_j) & \sum_{k:k\neq j}\psi^\prime(\theta_j - \theta_k)
\end{bmatrix}.
\end{align*}
This entails that, with $D_{ij}\equiv \sum_{k:k\neq i,j}\big(\psi^\prime(\theta_i - \theta_k)+\psi^\prime(\theta_j - \theta_k)\big)$,
\begin{align*}
\mathcal{J}_{ij}(\mu;\theta_{-i,-j})&\equiv \Big(\frac{\partial T_{ij}}{\partial \mu}\Big)^\top\mathcal{I}(\mu)^{-1}\Big(\frac{\partial T_{ij}}{\partial \mu}\Big)\\
&= \frac{1}{pL}\cdot \frac{D_{ij}}{\big(\sum_{k:k\neq i,j}\psi^\prime(\theta_i - \theta_k)\big)\big(\sum_{k:k\neq i,j}\psi^\prime(\theta_j - \theta_k)\big) + \psi^\prime(\theta_i - \theta_j)D_{ij}}\\
&= (1+\mathfrak{o}(1))\cdot\frac{1}{pL}\cdot \frac{D_{ij}}{\big(\sum_{k:k\neq i,j}\psi^\prime(\theta_i - \theta_k)\big)\big(\sum_{k:k\neq i,j}\psi^\prime(\theta_j - \theta_k)\big)}\\
&= (1+\mathfrak{o}(1))\cdot\frac{1}{pL}\cdot\Big(\frac{1}{\sum_{k:k\neq i,j}\psi^\prime(\theta_i - \theta_k)} + \frac{1}{\sum_{k:k\neq i,j}\psi^\prime(\theta_j - \theta_k)}\Big),
\end{align*}
where the $\mathfrak{o}(1)$ term is uniform over all $\theta\in \Theta(\kappa)$ and $i,j$. Moreover, by definition of $\pi_i(\cdot)$ and the fact that $r_n = \mathfrak{o}(1)$, direct calculation yields that
\begin{align*}
&\int\frac{1}{\sum_{k:k\neq i,j}\psi^\prime(\theta_i - \theta_k)}\pi_i(\theta_i)\pi_j(\theta_j)\d\theta_i\d\theta_j\\
& = \int \frac{1}{\sum_{k:k\neq i,j}\psi^\prime(\theta_i - \theta_k)}K_{r_n}\big(\theta_i - \theta^*_i\big)\d\theta_i = \frac{1+\mathfrak{o}(1)}{\sum_{k:k\neq i,j}\psi^\prime(\theta^*_i - \theta_k)}
\end{align*}
uniformly overall $\theta\in \Theta(\kappa)$ and $i,j$. With a similar bound for the term involving index $j$, we have
\begin{align}\label{ineq:vt_main}
\notag&\int \mathcal{J}_{ij}(\mu;\theta_{-i,-j})\pi_i(\theta_i)\pi_j(\theta_j)\d\theta_i\d\theta_j\\
& = \big(1+\mathfrak{o}(1)\big)\cdot\frac{1}{pL}\cdot\Big(\frac{1}{\sum_{k:k\neq i,j}\psi^\prime(\theta^*_i - \theta_k)} + \frac{1}{\sum_{k:k\neq i,j}\psi^\prime(\theta^*_j - \theta_k)}\Big).
\end{align}
Here the $\mathfrak{o}(1)$ is uniform overall $\theta\in \Theta(\kappa)$ and $i,j$.

(\textbf{Second term in (\ref{ineq:Rij_lower})}) On the other hand, let $\Delta \equiv \det\big((pL)^{-1}\mathcal{I}(\mu)\big)$. Then
\begin{align*}
C(\mu)\equiv \Big(\frac{\partial T_{ij}(\mu)}{\partial \mu}\Big)^\top\mathcal{I}^{-1}(\mu) = \frac{1}{pL\Delta}\Big(\sum_{k:k\neq i,j}\psi^\prime(\theta_j-\theta_k),-\sum_{k:k\neq i,j}\psi^\prime(\theta_i-\theta_k)\Big).
\end{align*}
Hence by definition of $\pi_i(\cdot)$ and $\pi_j(\cdot)$, we have
\begin{align*}
&\frac{\partial \big(C_{11}(\mu)\pi_i(\theta_i)\pi_j(\theta_j)\big)}{\partial \theta_i}\\
& = \frac{\big(\sum_{k:k\neq i,j}\psi^\prime(\theta_j - \theta_k)\big)\cdot K_{r_n}(\theta_j - \theta^*_j)}{pL\Delta^2}\cdot \Big[r_n^{-1}K^\prime_{r_n}(\theta_i - \theta^*_i)\Delta - K_{r_n}(\theta_i - \theta^*_i)\frac{\partial \Delta}{\partial \theta_i}\Big].
\end{align*}
With a similar bound for $\partial \big(C_{12}(\mu)\pi_i(\theta_i)\pi_j(\theta_j)\big)/\partial \theta_j$, and the simple bounds that $\Delta \gtrsim n^2$ and $|\partial \Delta/\partial \theta_i|\vee |\partial \Delta/\partial \theta_j| = \mathcal{O}(n^2)$ uniformly over all $\theta\in\Theta(\kappa)$ and $i,j$, we have
\begin{align}\label{ineq:vt_remainder}
\notag\tilde{\mathcal{I}}_{ij}(\pi_i,\pi_j) &\equiv \int \frac{1}{\pi_i(\theta_i)\pi_j(\theta_j)}\Big[\frac{\partial \big(C_{11}(\mu)\pi_i(\theta_i)\pi_j(\theta_j)\big)}{\partial \theta_i} + \frac{\partial \big(C_{12}(\mu)\pi_i(\theta_i)\pi_j(\theta_j)\big)}{\partial \theta_j}\Big]^2 \d\theta_i\d\theta_j\\
&\lesssim \frac{1}{n^2p^2L^2}\Big(1\vee \frac{1}{r_n^2}\int \frac{\big(K^\prime(x)\big)^2}{K(x)}\d x\Big) \lesssim \frac{1}{n^2p^2L^2r_n^2}.
\end{align}

Now apply (\ref{ineq:Rij_lower}) with (\ref{ineq:vt_main}) and (\ref{ineq:vt_remainder}) to yield that
\begin{align*}
R_{ij}(\theta_{-i,-j}) \geq \big(1+\mathfrak{o}(1)\big)\cdot\frac{1}{pL}\cdot\Big(\frac{1}{\sum_{k:k\neq i,j}\psi^\prime(\theta^*_i - \theta_k)} + \frac{1}{\sum_{k:k\neq i,j}\psi^\prime(\theta^*_j - \theta_k)}\Big)
\end{align*}
under the condition $r_n\ll (npL)^{-1/2}$. By (\ref{ineq:lower_reduction}) and the above lower bound, the local minimax risk is lower bounded by
\begin{align*}
\frac{1+\mathfrak{o}(1)}{npL}\cdot\sum_{i<j}\int \Big(\frac{1}{\sum_{k:k\neq i,j}\psi^\prime(\theta^*_i - \theta_k)} + \frac{1}{\sum_{k:k\neq i,j}\psi^\prime(\theta^*_j - \theta_k)}\Big)\big(\prod_{k:k\neq i,j}\pi_k(\theta_k)\big)\d\theta_{-i,-j},
\end{align*}
where the $\mathfrak{o}(1)$ is uniform over all $i,j$. Note that for any $\theta_{-i,-j}$ in the support of the product prior, we have
\begin{align*}
\sum_{k:k\neq i,j}\psi^\prime(\theta^*_i - \theta_k) = \big(1+\mathfrak{o}(1)\big)\sum_{k:k\neq i,j}\psi^\prime(\theta^*_i - \theta^*_k)\quad \text{as } r_n = \mathfrak{o}(1),
\end{align*}
where the $\mathfrak{o}(1)$ is uniform over all $i,j$. Hence the local minimax risk is further lower bounded by
\begin{align*}
&\frac{1+\mathfrak{o}(1)}{npL}\cdot\sum_{i<j}\Big(\frac{1}{\sum_{k:k\neq i,j}\psi^\prime(\theta^*_i - \theta^*_k)} + \frac{1}{\sum_{k:k\neq i,j}\psi^\prime(\theta^*_j - \theta^*_k)}\Big)\\
&=\frac{1+\mathfrak{o}(1)}{npL}\cdot\sum_{i<j}\Big(\frac{1}{\sum_{k:k\neq i}\psi^\prime(\theta^*_i - \theta^*_k)} + \frac{1}{\sum_{k:k\neq j}\psi^\prime(\theta^*_j - \theta^*_k)}\Big)\\
&=\frac{1+\mathfrak{o}(1)}{pL}\cdot\sum_{i=1}^n\frac{1}{\sum_{k:k\neq i}\psi^\prime(\theta^*_i - \theta^*_k)}, 
\end{align*}
as desired.
\end{proof}

\appendix

\section{Additional auxiliary results}\label{sec:appendix}

\begin{lemma}\label{lem:A_concen}
Suppose that $np\gg \log n$. Then there exist some universal positive constants $c, C$ such that the following hold with probability $1 - \mathcal{O}(n^{-10})$.
\begin{align*} 
cnp\leq \min_{i\in[n]}\sum_{j:j\neq i}A_{ij} \leq \max_{i\in[n]}\sum_{j:j\neq i}A_{ij} \leq Cnp,\\
\max_{m\in[n]}\sum_{i,j:i<j, i,j\neq m}A_{ij}A_{mi} \lesssim C(np)^2.
\end{align*}
\end{lemma}
\begin{proof}
For any fixed $i\in[n]$, we have $cnp\leq \sum_{j:j\neq i}A_{ij} \leq Cnp$ with the prescribe probability by Hoeffding's inequality. The first claim now follows by taking the union bound over $i\in[n]$. The second claim follows since
\begin{align*}
\max_{m\in[n]}\sum_{i,j:i<j, i,j\neq m}A_{ij}A_{mi} \leq \max_{m\in[n]}\sum_{i:i\neq m}A_{mi}\cdot \max_{i\in[n]}\sum_{j:j\neq i}A_{ij} \lesssim (np)^2
\end{align*}
by the first claim. 
\end{proof}

The following lemma gives some standard concentration results in the analysis of both MLE and the spectral method. 
\begin{lemma}\label{lem:mle_concentration}
Assume $np\geq c_0\log n$ for some sufficiently large $c_0 > 0$ and $\kappa = \mathcal{O}(1)$. Then the following hold with probability $1 - \mathcal{O}(n^{-10})$ uniformly over $\theta^*\in\Theta(\kappa)$.
\begin{align*}
&\sum_{i=1}^n \Big(\sum_{j:j\neq i}A_{ij}\big(\bar{y}_{ij} - \psi(\theta_i^* - \theta_j^*)\big)\Big)^2 \leq C\frac{n^2p}{L},\\
&\max_{i\in[n]}\Big|\sum_{j:j\neq i}A_{ij}\big(\bar{y}_{ij} - \psi(\theta_i^* - \theta_j^*)\big)\Big|^2 \leq C\frac{np\log n}{L},\\
&\max_{m\in[n]} \Big|\sum_{i:i\neq m}A_{im}\big(\bar{y}_{mi}\pi_i^* - \bar{y}_{im}\pi_m^*\big)\Big|^2 \leq  C\frac{p\log n}{nL}.
\end{align*}
Here $C = C(\kappa) > 0$.
\end{lemma}
\begin{proof}[Proof of Lemma \ref{lem:mle_concentration}]
The first two claims are proved in \cite[Lemma 7.4]{chen2020partial}. For the third claim, note that the summands $\bar{y}_{mi}\pi_i^* - \bar{y}_{im}\pi_m^*$ are independent, centered, and sub-Gaussian with variance proxy bounded by $L^{-1}\pnorm{\pi^*}{\infty}^2\lesssim (n^2L)^{-1}$. Hence the claim follows from Hoeffding's inequality
\begin{align*}
\Prob\Big(\Big|\sum_{i:i\neq m}A_{im}\big(\bar{y}_{mi}\pi_i^* - \bar{y}_{im}\pi_m^*\big)\Big|\geq t|A\Big) \leq 2\exp\Big(-\frac{Ct^2}{(n^2L)^{-1}\cdot \sum_{i:i\neq m}A_{im}}\Big),
\end{align*}
the bound $\max_{m\in[n]}\sum_{i:i\neq m}A_{mi}\lesssim np$ by Lemma \ref{lem:A_concen}, and a union bound. 
\end{proof}

The following lemma controls the spectrum of Laplacian matrices corresponding to a weighted Erd\H{o}s-R\'{e}nyi graph. 
\begin{lemma}\label{lem:hessian_eigen}
Let $A$ be the Erd\H{o}s-R\'{e}nyi graph adjacency matrix with connection probability $p$, and $\{w_{ij}\}_{i\neq j}$ be a non-negative weight sequence such that $w_{ij} = w_{ji}$. Let $H$ be a weighted Laplacian matrix defined by $H_{ij} \equiv -A_{ij}w_{ij}$ for $i\neq j$ and $H_{ii}\equiv \sum_{j\neq i}A_{ij}w_{ij}$. Assume that $p\geq c_0\log n/n$ for some sufficiently large $c_0 > 0$ and $\kappa_1\leq \min_{i\neq j}w_{ij} \leq \max_{i\neq j}w_{ij} \leq \kappa_2$ for some $\kappa_1,\kappa_2 > 0$. Then for some positive $c,C$ only depending on $c_0$, 
\begin{align*}
(c\kappa_1)np\leq\lambda_{\min,\perp}(H) \leq \lambda_{\max,\perp}(H) \leq (C\kappa_2)np
\end{align*}
with probability at least $1 - \mathcal{O}(n^{-10})$. Here $\lambda_{\min,\perp}$ stands for the smallest eigenvalue along directions orthogonal to $\bm{1}_n$.
\end{lemma}
\begin{proof}[Proof of Lemma \ref{lem:hessian_eigen}]
Let $\mathcal{L}_A$ be the graph Laplacian induced by $A$. For any unit vector $c\in\R^n$ such that $\bm{1}_n^\top c = 0$, we have
\begin{align*}
c^\top Hc = \sum_{i\neq j}\frac{1}{2}(c_i-c_j)^2A_{ij}w_{ij} \geq \kappa_1\cdot  \sum_{i\neq j}\frac{1}{2}(c_i-c_j)^2A_{ij} \geq \kappa_1 \cdot \lambda_{\min, \perp}(\mathcal{L}_A).  
\end{align*}
The claim now follows from standard results of the Erd\H{o}s-R\'{e}nyi graph \cite{tropp2015introduction}. The upper bound proof is similar. 
\end{proof}

The following lemma provides rate-optimal bounds for the global MLE $\hat{\theta}$ and the leave-one-out MLE $\hat{\theta}^{(m)}$ in (\ref{def:loo_mle}) for each $m\in[n]$.
\begin{lemma}\label{lem:pre_estimate}
Suppose that $\kappa = \mathcal{O}(1)$ and $np\gg \log n$. There exists some $C = C(\kappa) > 0$ such that the following hold with probability $1 - \mathcal{O}(n^{-10})$ uniformly over all $\theta^*\in\Theta(\kappa)$.
\begin{enumerate}
\item We have $\max_{m\in[n]}\pnorm{\hat{\theta}^{(m)} - \theta^*_{-m}}{}^2\leq C(pL)^{-1}$ and $\max_{m\in[n]}\pnorm{\hat{\theta}^{(m)} - \theta^*_{-m}}{\infty}^2\leq C\log n/(npL)$.
\item We have $\max_{m\in[n]}\pnorm{\hat{\theta}^{(m)} - \hat{\theta}_{-m}}{}^2\leq C(npL)^{-1}$.
\end{enumerate}
\end{lemma}
\begin{proof}
These are essentially proved in \cite[Lemmas 8.6 and 8.7]{chen2020partial}. Note that there are two differences between our leave-one-out MLE $\hat{\theta}^{(m)}$ in (\ref{def:loo_mle}) and the one in Equation (57) of \cite{chen2020partial}: (i) our (\ref{def:loo_mle}) does not have the $\ell_\infty$ constraint $\pnorm{\theta_{-m}-\theta_{-m}^*}{\infty}\leq 5$ in its definition; (ii) our (\ref{def:loo_mle}) satisfies $\ave(\theta_{-m}) = \ave(\theta^*_{-m})$ by definition. The above bounds continue to hold with the prescribed probability since for (i), the constraint is satisfied with the prescribed probability by (a leave-one-out version of) \cite[Proposition 8.1]{chen2020partial}, and for (ii), Lemmas 8.6 and 8.7 in \cite{chen2020partial} hold for $\hat{\theta}^{(m)}$ only after the additional centering $\hat{\theta}^{(m)} - \ave(\hat{\theta}^{(m)} - \theta^*_{-m})$, which coincides with the leave-one-out MLE in our definition (\ref{def:loo_mle}).
\end{proof}

The following lemma summarizes some rate-optimal bounds of the spectral estimate $\hat{\pi}$.

\begin{lemma}[Lemma 9.1 of \cite{chen2020partial}, see also \cite{chen2019spectral}]\label{lem:spec_rate}
Suppose that $\kappa = \mathcal{O}(1)$ and $np\gg \log n$. Then the following hold with probability at least $1- \mathcal{O}(n^{-10})$ for some positive $C = C(\kappa)$.
\begin{enumerate}
\item We have $\max_{m\in[n]} \pnorm{\hat{\pi}^{(m)} - \hat{\pi}}{}\leq Cn^{-1}\sqrt{\log n/(npL)}$.
\item We have $\max_{m\in[n]} \pnorm{\hat{\pi}^{(m)} - \pi^*}{\infty} \leq Cn^{-1}\sqrt{\log n/(npL)}$ and $\pnorm{\hat{\pi} - \pi^*}{\infty} \leq Cn^{-1}\sqrt{\log n/(npL)}$.
\item We have $\max_{m\in[n]} \pnorm{\hat{\pi}^{(m)} - \pi^*}{} \leq Cn^{-1}\sqrt{1/(pL)}$ and $\pnorm{\hat{\pi} - \pi^*}{} \leq Cn^{-1}\sqrt{1/(pL)}$.
\end{enumerate}
\end{lemma}

The following lemma summarizes some useful estimates regarding $f^{(m)}(\theta_m|\theta_{-m})$ and $g^{(m)}(\theta_m|\theta_{-m})$ defined in (\ref{def:fg}) for each $m\in[n]$. 
\begin{lemma}\label{lem:bound_fg}
Suppose $\kappa = \mathcal{O}(1)$ and $np\gg \log n$. Then the following hold with probability $1 - \mathcal{O}(n^{-10})$ for each $m\in[n]$.
\begin{enumerate}
\item We have $\max_{m\in[n]}\Big|f^{(m)}(\theta^*_m|\hat{\theta}_{-m})\Big|\lesssim \sqrt{np\log n/L}$ and $\sum_{m=1}^n \Big(f^{(m)}(\theta^*_m|\hat{\theta}_{-m})\Big)^2\lesssim n^2p/L$.
\item We have $\max_{m\in[n]}\Big|f^{(m)}(\theta^*_m|\hat{\theta}_{-m}) - f^{(m)}(\theta^*_m|\theta^*_{-m})\Big|\lesssim \sqrt{np/L} + \sqrt{(\log n)^3/(npL)}$.
\item We have $g^{(m)}(\theta_m^*|\theta_{-m}^*) \asymp np$.
\item We have $\big|g^{(m)}(\theta_m^*|\hat{\theta}_{-m}) - g^{(m)}(\theta_m^*|\theta_{-m}^*)\big| \lesssim \sqrt{np/L} + \sqrt{(\log n)^3/(npL)}$. Consequently, $g^{(m)}(\theta_m^*|\hat{\theta}_{-m}) = \big(1+\epsilon_m\big)g^{(m)}(\theta_m^*|\theta_{-m}^*)$ under the condition $np\gg \log n$ for some $\epsilon\in\R^n$ such that $\pnorm{\epsilon}{\infty} = \mathfrak{o}(1)$.
\end{enumerate}
\end{lemma}
\begin{proof}[Proof of Lemma \ref{lem:bound_fg}]
Claims (1)(3)(4) follow from analogous arguments as in the proof of Lemma \ref{lem:bound_loo_fg}, and it remains to prove Claim (2). By definition, we have
\begin{align*}
f^{(m)}(\theta^*_m|\hat{\theta}_{-m}) - f^{(m)}(\theta^*_m|\theta^*_{-m}) = \sum_{i:i\neq m}A_{mi}\big(\psi(\theta_m^*-\hat{\theta}_i) - \psi(\theta_m^*-\theta^*_i)\big).
\end{align*}
Recall that $\hat{\theta}^{(m)}$ is the leave-one-out MLE defined in (\ref{def:loo_mle}). Then we have
\begin{align*}
&\sum_{i:i\neq m}A_{mi}\big(\psi(\theta_m^*-\hat{\theta}_i) - \psi(\theta_m^*-\theta^*_i)\big)\\
&= \sum_{i:i\neq m}A_{mi}\big(\psi(\theta_m^*-\hat{\theta}^{(m)}_i) - \psi(\theta_m^*-\theta^*_i)\big)\\
&\quad\quad\quad + \sum_{i:i\neq m}A_{mi}\big(\psi(\theta_m^*-\hat{\theta}_i) - \psi(\theta_m^*-\hat{\theta}^{(m)}_i)\big) \equiv \Delta_{m,1} + \Delta_{m,2}.
\end{align*}
The term $\Delta_{m,1}$ is controlled with the prescribed probability by
\begin{align}\label{ineq:f_bound_II}
\notag|\Delta_{m,1}| &\leq p\cdot\Big|\sum_{i:i\neq m}\big(\psi(\theta_m^* - \theta_i^*) - \psi(\theta_m^* - \hat{\theta}^{(m)}_i)\big)\Big|\\
\notag&\quad + \Big|\sum_{i:i\neq m}(A_{mi}-p)\big(\psi(\theta_m^* - \theta_i^*) - \psi(\theta_m^* - \hat{\theta}^{(m)}_i)\big)\Big|\\
\notag&\stackrel{(*)}{\lesssim} p\sqrt{n}\pnorm{\hat{\theta}^{(m)} - \theta^*_{-m}}{} + \sqrt{p\log n}\pnorm{\hat{\theta}^{(m)} - \theta^*_{-m}}{} + \log n\cdot\pnorm{\hat{\theta}^{(m)} - \theta^*_{-m}}{\infty}\\
&\stackrel{(**)}{\lesssim} p\sqrt{n} \cdot \sqrt{\frac{1}{pL}} + \log n\cdot\sqrt{\frac{\log n}{npL}} = \sqrt{\frac{np}{L}} + \sqrt{\frac{(\log n)^3}{npL}}.
\end{align}
Here $(*)$ follows from Bernstein's inequality applied conditioning on data without the $m$th observation, and $(**)$ follows from Lemma \ref{lem:pre_estimate}. 

On the other hand, $\Delta_{m,2}$ is controlled by 
\begin{align*}
|\Delta_{m,2}| &\leq \sum_{i:i\neq m}A_{mi}|\hat{\theta}^{(m)}_i- \hat{\theta}_j| \leq \Big(\sum_{i:i\neq m}A_{mi}\Big)^{1/2}\pnorm{\hat{\theta}^{(m)} - \hat{\theta}_{-m}}{}\\
& \stackrel{(*)}{\lesssim} \sqrt{np}\cdot\sqrt{\frac{1}{npL}} = \sqrt{\frac{1}{L}},
\end{align*}
where $(*)$ follows by Lemma \ref{lem:pre_estimate}. Since both estimates are uniform over $m\in[n]$, the proof is complete.
\end{proof}

Recall $\bar{\theta}$ defined in (\ref{def:theta_bar}) and its motivation of local quadratic expansion in (\ref{eq:mle_local_quad}). The following lemma quantifies the closeness between $\bar{\theta}$ and $\hat{\theta}$.
\begin{lemma}\label{lem:theta_bar_close}
Suppose that $\kappa = \mathcal{O}(1)$ and $np\gg \log n$. Then there exists some $C = C(\kappa) > 0$ such that the following holds with probability $1-\mathcal{O}(n^{-10})$.
\begin{align*}
\big|\hat{\theta}_m - \bar{\theta}_m\big|^2 \leq C\Big(\frac{\big|f^{(m)}(\theta_m^*|\hat{\theta}_{-m})\big|}{np}\Big)^3, \quad \forall m\in[n].
\end{align*}
\end{lemma}
\begin{proof}[Proof of Lemma \ref{lem:theta_bar_close}]
This follows directly from the estimates before Equation (78) of \cite{chen2020partial}.
\end{proof}

\bibliographystyle{amsalpha}
\bibliography{mybib}

\end{document}